\documentclass[11pt]{amsart}
\usepackage{amsmath,amsfonts,amssymb,amsthm,epsfig,color, mathrsfs,eucal, amscd} 
\usepackage{verbatim}
\usepackage{hyperref} 
\usepackage{bbm}

\usepackage{hyperref}
\usepackage{graphicx} 
\usepackage{booktabs} 
\usepackage{wrapfig} 
\usepackage[labelfont=bf]{caption} 
\usepackage[top=0.9in,bottom=0.9in,left=1in,right=1in]{geometry} 

\usepackage{tensor} 

\usepackage{mathabx}

\newcommand{\bb}{\mathbb}

\newcommand{\h}{\mathfrak{h}}

\newcommand{\mL}{\mathcal{L}}
\newcommand{\Pp}{\mathcal{P}}
\newcommand{\Hc}{\mathcal{H}}
\newcommand{\Rr}{\mathcal{R}}
\newcommand{\Nn}{\mathcal{N}}

\newcommand{\ZZ}{\bb Z}

\newcommand{\RR}{\bb R}
\newcommand{\TT}{\bb T}

\newcommand{\NN}{\bb N}

\newcommand{\In}{\mathbbm{1}}

\newcommand{\diag}{\operatorname{diag}}

\newcommand{\mU}{\mathcal U}

\newcommand{\Ss}{\mathcal S}
\newcommand{\Aa}{\mathcal A}
\newcommand{\Bb}{\mathcal B}
\newcommand{\Dd}{\mathcal D}
\newcommand{\Ee}{\mathcal E}
\newcommand{\Kk}{\mathcal K}
\newcommand{\Tt}{\mathcal T}
\newcommand{\I}{\mathcal{I}}
\newcommand{\J}{\mathcal{J}}

\newcommand{\Ga}{\Gamma}

\newcommand{\wrt}[1]{\mathrm{d}{#1}}

\newcommand{\supp}{\operatorname{supp}}

\newcommand{\vol}{\operatorname{vol}}
\newcommand{\SL}{\operatorname{SL}}
\newcommand{\GL}{\operatorname{GL}}

\newcommand{\PSL}{\operatorname{PSL}}

\newcommand{\SO}{\operatorname{SO}}
\newcommand{\Oo}{\operatorname{O}}
\newcommand{\Aut}{\operatorname{Aut}}

\newcommand{\F}{\mathcal F}

\newcommand{\pf}{\mathfrak p}
\newcommand{\qf}{\mathfrak q}
\newcommand{\Cf}{\mathfrak C}
\newcommand{\If}{\mathfrak{I}}

\newcommand{\Ff}{\mathscr {F}}
\newcommand{\SP}{\mathscr {SP}}
\newcommand{\Cc}{\mathcal {C}}

\newcommand{\Sie}{\operatorname{Sie}}

\newtheorem{Theorem}{Theorem}
\numberwithin{Theorem}{section}

\newtheorem{coro}[Theorem]{Corollary}

\newtheorem{theo}[Theorem]{Theorem}
\newtheorem{prop}[Theorem]{Proposition}

\newtheorem{lemm}[Theorem]{Lemma}

\newtheorem*{lemma*}{Lemma}
\newtheorem*{question*}{Question}

\newtheorem*{theorem*}{Theorem}

\theoremstyle{remark}

\newtheorem{rema}[Theorem]{\sc Remark}

\numberwithin{equation}{section}
\begin{document}

\title{Shrinking target horospherical equidistribution via translated Farey sequences}
\author{Jimmy Tseng}
\email{j.tseng@exeter.ac.uk}
\address{Department of Mathematics and Statistics, University of Exeter, Exeter EX4 4QF, UK}

\thanks{The author was supported by EPSRC grant EP/T005130/1.}

\begin{abstract} For a certain diagonal flow on $\SL(d, \ZZ)  \backslash \SL(d, \RR)$ where $d \geq 2$, we show that any bounded subset (with measure zero boundary) of the horosphere or a translated horosphere equidistributes, under a suitable normalization, on a target shrinking into the cusp.  This type of equidistribution is {\em shrinking target horospherical equidistribution (STHE)}, and we show STHE for several types of shrinking targets.  Our STHE results extend known results for $d=2$ and $\mL \backslash \PSL(2, \RR)$ where $\mL$ is any cofinite Fuchsian group with at least one cusp.

The two key tools needed to prove our STHE results for the horosphere are a renormalization technique and Marklof's result on the equidistribution of the Farey sequence on distinguished sections.   For our STHE results for translated horospheres, we introduce {\em translated Farey sequences}, develop some of their geometric and dynamical properties, generalize Marklof's result by proving the equidistribution of translated Farey sequences for the same distinguished sections, and use this equidistribution of translated Farey sequences along with the renormalization technique to prove our STHE results for translated horospheres.

\end{abstract}


\maketitle

\tableofcontents

\section {Introduction}\label{secBackground}

Let $d \geq 2$ be an integer, $G:=\SL(d, \RR)$, $\Gamma:= \SL(d,\ZZ)$, and $\mu$ be the Haar measure on $G$ normalized so that the measure of a fundamental domain of the lattice $\Ga$ is equal to one.  In this paper, we introduce a generalization of the Farey sequence from number theory (\ref{eqn:DefnFareySeq}) which we call a \textit{translated Farey sequence} (\ref{eqn:DefnTranFareySeq}), prove geometric and dynamical properties of translated Farey sequences (Section~\ref{subsecTranFareySeq} and Theorem~\ref{thmGeneralMarklof}), and use these properties to prove shrinking target horospherical equidistribution for a certain diagonal flow (\ref{eqn:DefnOurDiaFlow}) on the space of unimodular lattices $\Gamma \backslash G$, extending the result for $d=2$, namely~\cite[Theorem~1.1]{Tse21}.  Like the Farey sequence, a translated Farey sequence has useful geometric and dynamical properties, namely that it is discrete and equidistributes on a distinguished section (\ref{eqn:DefnOurSection}) of $\Gamma \backslash G$, and these properties allow us to rescale in space and time, providing a {\em renormalization}, which, as we shall see, is a powerful tool in studying geometric and dynamical questions such as the type of equidistribution results considered in this paper and may have further applicability.  Even for $d=2$, the technique in this paper allows us to consider any horocycle, instead of just the periodic horocycles considered in~\cite{Tse21}.  The technique from~\cite{Tse21}, the double coset decomposition and the method of stationary phase, seems to only be useful for periodic horocycles.

\subsection{Shrinking target equidistribution and preliminary notation}\label{subsec:STHEandPrelimNotation}  On a finite-volume space with cusps and with a geometric object that, under a flow, equidistributes on that space, a natural question to consider is what happens to the equidistribution restricted to a target set shrinking into the cusps.  This question is thus about the relationship between two dynamics, the rate of the equidistribution of the object and the rate of the shrinking of the target set.  If the target shrinks too quickly relative to the equidistribution of the object on the whole space, then we expect that no equidistribution occurs on the shrinking set.  If the target shrinks slowly enough with respect to the equidistribution of the object on the whole space, then it may be possible to obtain equidistribution on the shrinking target, suitably normalized for the shrinking of the target.  For $d=2$ and the equidistribution of periodic horocycles, such shrinking target equidistribution is possible for not just the modular group but also any cofinite Fuchsian group with at least one cusp~\cite[Theorems~1.1 and~1.2]{Tse21}.  Moreover, this dichotomy is sharp, namely there is a {\em critical exponent of relative rate} $c_r$, under which the target shrinks slowly enough relative to the equidistribution so that, normalized by the rate of the shrinking of the target into a cusp raised to the {\em normalizing exponent} $c_e$, one obtains (a slower) equidistribution on a related fixed target, called the {\em renormalized target}.  In the notation (\ref{eqn:DefnOurDiaFlow}, \ref{eqnSectionForT}) of this paper, the results~\cite[Theorems~1.1 and~1.2]{Tse21} show that the pair $(c_r, c_e)$ is $(2, 1)$ and invariant over the aforementioned Fuchsian groups.  We note that our results, Theorems~\ref{thmShrinkCuspNeigh},~\ref{thmNANResultRankOneFlow},~\ref{thmNAKResultRankOneFlow},~\ref{thmThickenSubsetSectionSpherical},~\ref{thmTwoChartGeneralSpherical}, and~\ref{thmThickenSubsetSectionSphericalGeneral}, in this paper are in agreement and give this pair as $(d, d-1)$ for $\Gamma$.

The notion of shrinking target equidistribution for horocycles has been precisely formulated for $d=2$ and the aforementioned Fuchsian groups at~\cite[(1.2)]{Tse21}.  We now extend this definition to $d \geq 2$ and $\Gamma$, noting that for $d=2$ and $\Gamma$ the definition in~\cite{Tse21} and in this paper are essentially the same, with each being more general than the other in different aspects.  Let $I_\ell$ be the $\ell \times \ell$ identity matrix.  Define the diagonal flow \begin{align}\label{eqn:DefnOurDiaFlow}
 \Phi^t := \begin{pmatrix} e^{-t} I_{d-1} & \tensor*[^t]{\boldsymbol{0}}{}  \\ \boldsymbol{0} & e^{(d-1)t}\end{pmatrix} \end{align} and the unipotent elements \[n_+(\boldsymbol{\widetilde{x}}) := \begin{pmatrix} I_{d-1} & \tensor*[^t]{\boldsymbol{0}}{} \\ \boldsymbol{\widetilde{x}} & 1  \end{pmatrix} \quad \textrm{ and } \quad n_-(\boldsymbol{x}) := \begin{pmatrix} I_{d-1} & \tensor*[^t]{\boldsymbol{x}}{} \\ \boldsymbol{0} & 1  \end{pmatrix},\] where the abelian subgroups \[N^+:=\{n_+(\boldsymbol{\widetilde{x}}): \boldsymbol{\widetilde{x}}\in\RR^{d-1}\}  \quad \textrm{ and } \quad N^-:=\{n_-(\boldsymbol{x}): \boldsymbol{x}\in\RR^{d-1}\}\] are the stable and unstable horospherical subgroups, respectively, for $\Phi^t$.   Let $\wrt {\boldsymbol{x}}$ (and $\wrt {\widetilde{\boldsymbol{x}}}$) be the Lebesgue measure on $\RR^{d-1}$, normalized so that it is a probability measure on $\TT^{d-1}$.  For a set $\Bb \subset \RR^{d-1}$, let \[N^+(\Bb):=\{n_+(\boldsymbol{\widetilde{x}}): \boldsymbol{\widetilde{x}}\in\Bb\}  \quad \textrm{ and } \quad N^-(\Bb):=\{n_-(\boldsymbol{x}): \boldsymbol{x}\in\Bb\}.\] 
 
 Fix $L \in G$ and let $\Aa_L$ be defined as in Section~\ref{subsec:GammaDuplicates}.  Given a bounded subset $\Aa \subset \Aa_L$ with measure zero boundary with respect to $\wrt {\boldsymbol{x}}$, we refer to $\Ga L N^-(\Aa)$ as a {\em (left-)translated horosphere} and say it {\em equidistributes} on a $\mu$-measurable subset $\mU$ of $\Gamma \backslash G$ if \[ \int_\Aa \In_{\mU}\left(Ln_-(\boldsymbol{x}) \Phi^t\right)~\wrt {\boldsymbol{x}} \xrightarrow[]{t \rightarrow \infty}\mu(\mU)\left(\int_{\RR^{d-1}} \In_{\Aa}(\boldsymbol{x})~\wrt{\boldsymbol{x}}\right).\]  Here $\In_S$ is the indicator function for a set $S$, and the argument of $\In_{\mU}$ is considered as an element of $\Gamma \backslash G$.  When $L = I_d$, we also say the {\em horosphere} and the {\em horosphere equidistributes}.\footnote{There are some extra complications in the proofs for generic $L$ versus $L = I_d$ for which having terminology distinguishing these cases is convenient.}  Translated horosphere equidistribution has been studied in various setups in a number of papers such as, for example,~\cite[Theorem~1.2]{EM},~\cite[Proposition~2.4.8]{KM96},~\cite[Theorem~5.2]{MS10}, and~\cite[Theorem~1.1]{EMV09}.  Note that the horosphere and translated horospheres are all regarded as objects in the space $\Gamma \backslash G$, and their equidistribution is with respect to the right-action of $\Phi^t$.
 
For shrinking target equidistribution, we replace the fixed target $\mU$ with a shrinking target.  Let \[H := \left\{\begin{pmatrix} A & \tensor*[^t]{\boldsymbol{b}}{} \\ \boldsymbol{0} & 1  \end{pmatrix} : A \in \SL(d-1, \RR), \boldsymbol{b} \in \RR^{d-1} \right\} \quad \textrm{ and } \quad \Ga_H:= \Ga \cap H. \]  Let $\mu_0$ be the Haar measure on $\SL(d-1, \RR)$ normalized so that the measure of a fundamental domain of the lattice $\SL(d-1, \ZZ)$ is equal to one, and $\mu_H$ be the left-invariant Haar measure on $H$ normalized in the same way with respect to the lattice $\Ga_H$.  Explicitly, we have that $\wrt \mu_H = \wrt \mu_0 \wrt{ \boldsymbol{b}}$ (cf~\cite[(3.4)]{Mar10}).  We begin defining shrinking target equidistribution by considering a section \begin{align}\label{eqn:DefnOurSection}
 \Ss_1:=\Ga \backslash \Ga H \{ \Phi^{-s} : s \in \RR_{\geq 0}\}, \end{align} which is a closed embedded submanifold of $\Ga \backslash G$~\cite[Lemma~2]{Mar10}, and we are, in particular, interested in the subset \begin{align}\label{eqnSectionForT}
\Ss_T:=  \Ga \backslash \Ga H \{ \Phi^{-s} : s \in \RR_{\geq \frac 1 d \log T}\}  \end{align} for $T \geq 1$.  By thickening $\Ss_T$ or one of its subsets, we obtain a shrinking target $\mU_T$ to be defined precisely below.  Note that $\Ss_T$ and, thus, $\mU_T$  shrink into the cusp as $T \rightarrow \infty$.  For such a target $\mU_T$ shrinking into the cusp slowly enough with respect to the rate of the equidistribution of a translated horosphere, we expect, provided we normalize to account for the shrinking target, that the translated horosphere equidistributes as follows:  \begin{align}\label{eqn:STHEDefn} 
 {T^{d-1}}\int_{\Aa} \In_{\mU_T}({Ln_-(\boldsymbol{x}) \Phi^t})~\wrt {\boldsymbol{x}} \longrightarrow T_0^{d-1}\mu(\mU_{T_0})\left(\int_{\RR^{d-1}} \In_{\Aa}(\boldsymbol{x})~\wrt{\boldsymbol{x}}\right)\end{align} as $T \rightarrow \infty$ and $t \rightarrow \infty$.  We refer to (\ref{eqn:STHEDefn}) as {\em shrinking target horospherical equidistribution (STHE)}.  Here $T_0 \geq 1$ is a fixed constant and $\mU_{T_0}$ is a fixed target, which we refer to as the {\em renormalized target}.   The normalization to account for the shrinking target is to divide by $\left(\frac{T_0} {T}\right)^{d-1}$, and, since it is the same for all of our results in this paper, we refer to its exponent as the {\em normalizing exponent}.\footnote{This normalizing exponent is the same as that in~\cite{Tse21} even though the notation in that paper is slightly different.}  Note that STHE is closely related to Diophantine approximation in that Diophantine approximation can be regarded as the critical limit of STHE in a suitable sense, and this relationship and its implications for number theory will be elaborated upon in~\cite{Tse23}.

\subsection{Construction of shrinking targets}\label{secConstructShrinkTargets}  We construct the shrinking targets $\mU_T$ by thickening.  Fix a subset $\Bb \subset \TT^{d-1}$.  Since we are considering the equidistribution of a  translated horosphere, it is natural to consider the thickening in the stable horospherical directions, namely the set $\Ss_T N^+(\Bb)$.  (See Corollary~\ref{coroHaarMeaSTN} for $\mu\left(\Ss_T N^+(\Bb) \right)$.)  More generally, we will consider the analogous thickenings of subsets of $\Ss_T$, and these subsets are given in terms of (\ref{eqnFundDomCoord}), the Grenier coordinates.  These are the most convenient coordinates for our technique in this paper as the fundamental domain $\varphi(\Ff'_d)$, defined in Section~\ref{secCuspNeigh}, associated to these coordinates has a ``box shape'' in the cusp (see Section~\ref{secCuspNeigh} for details).  

Let us define a useful subgroup of \[K := \SO(d, \RR),\] namely \begin{align*} K':=\left\{\begin{pmatrix}  k &\tensor*[^t]{\boldsymbol{0}}{}  \\ \boldsymbol{0} & 1 \end{pmatrix} : k \in
\SO(d-1, \RR) \right\}. \end{align*}  Let $\wrt k$ denote the probability Haar measure on $K$ and $\wrt {\widetilde{k}}$ denote the induced Haar measure on $K'$.  By (\ref{eqnDefnSiegelSets}) and Remarks~\ref{rmkTheFunDoms} and~\ref{rmkOmittingVarphi}, we can form a subset on $\Ss_1$, which we refer to as a {\em Grenier box}, by choosing a measurable subset \begin{align}\label{eqnDefofTildeK}\widetilde{K} \subset K'.\end{align} and constants $\alpha_\ell$, $\gamma_\ell$, $\beta^-_{ij}$, $\beta^+_{ij}$ as follows:  \begin{align}\label{eqnDefnGrenierBox} \begin{cases}  \begin{cases}   0\leq \beta^-_{1j}\leq x_{1j}\leq \beta^+_{1j}\leq 1/2  &\textrm {for } j = 2,\cdots, d-1 \\ -1/2 \leq \beta^-_{1j} \leq  x_{1j}\leq \beta^+_{1j}\leq1/2 &\textrm {for } j =  d \\ -1/2\leq \beta^-_{ij} \leq x_{ij}\leq \beta^+_{ij} \leq 1/2 &\textrm {for } 2\leq i < j \leq  d \\  1\leq \alpha_\ell \leq y_\ell \leq \gamma_\ell \leq \infty &\textrm {for } \ell = 1,\cdots, d-1\end{cases} & \textrm{ if } d \textrm{ is even,} \\
\begin{cases}   0\leq \beta^-_{1j}\leq x_{1j}\leq \beta^+_{1j}\leq 1/2  &\textrm {for } j = 2,\cdots, d  \\ -1/2\leq \beta^-_{ij} \leq x_{ij}\leq \beta^+_{ij} \leq 1/2  &\textrm {for } 2\leq i < j \leq  d \\   1\leq \alpha_\ell \leq y_\ell \leq \gamma_\ell \leq \infty &\textrm {for } \ell = 1,\cdots, d-1\end{cases} & \textrm{ if } d \textrm{ is odd.}\end{cases}
  \end{align}  We will denote a Grenier box by $\Cc:=\Cc_{T_-, T_+}:=\Cc(\boldsymbol{\alpha}, \boldsymbol{\gamma}, \widetilde{K},  \beta^-_{ij}, \beta^+_{ij})$ where $\boldsymbol{\alpha} := (\alpha_1, \cdots, \alpha_{d-1})$, $\boldsymbol{\gamma} := (\gamma_1, \cdots, \gamma_{d-1})$, and \[T^d_-:=\prod^{d-1}_{k=1} \alpha_{d-k}^{2(d-k)}, \quad \quad T^d_+:=\prod^{d-1}_{k=1} \gamma_{d-k}^{2(d-k)}.\]  The constants $T_-$ and $T_+$ are the \textit{lower height} and \textit{upper height}, respectively, of the Grenier box.  Our shrinking targets $\mU_T$ thickened in the stable horospherical directions from subsets of $\Ss_T$ are then given by \[\widetilde{\Cc}_TN^+(\Bb):=\Cc \Phi^{-\frac 1 d \log(T/T_-^{d/2(d-1)})}N^+(\Bb).\]

Alternatively, in view of the Iwasawa decomposition of $G$, which gives a natural coordinate system for $G$, thickening in spherical directions is also natural and, perhaps, more useful than in stable directions.  Let $e_1, \cdots, e_\ell$ be the standard basis vectors of $\RR^\ell$ and $S^{\ell-1}$ be the sphere in $\RR^{\ell}$ of radius one centered at the origin.  If we thicken by a subset of $K$, we, in fact, only thicken in the directions given by the unit sphere $S^{d-1}$ because the directions corresponding to the subgroup $K'$ (whose elements fix, from the right, the vector $e_d$) are contained in the section $\Ss_1$.  Moreover, we have the identification \begin{align}\label{eqn:SphereIsKPrimeModK}
S^{d-1} = K' \backslash K  \end{align} as the $d$-th row of any element of an equivalence class is invariant under the class.  We will use similar notation to that in~\cite[Section~5.2]{MS10}.  Let $\Dd \subset \RR^{d-1}$ be an open bounded set.  Define a smooth mapping \begin{align}\label{eqn:DefnSmoothMapE} E: \Dd \rightarrow K  \end{align} such that  $\boldsymbol{0} \mapsto I_d$ and the induced mapping \begin{align}\label{eqnInducedDiffeoUsinged}
 \Dd \ni \boldsymbol{z} \mapsto e_d E(\boldsymbol{z})^{-1} \in S^{d-1} \end{align} is injective and has nonsingular differential at every point of $\Dd$ (cf the mapping into the $d-1$-dimensional unit sphere in~\cite[Corollary~5.4]{MS10}).  Define \begin{align}
\label{eqn:SphericalThickening} \Ee := \Ee(\Dd):= \left\{E(\boldsymbol{z})^{-1} : \boldsymbol{z} \in \Dd\right\}. \end{align}

Let us now write \begin{align}\label{eqnDefnEInv}
 E(\boldsymbol{z})^{-1} = \begin{pmatrix}  A &  \tensor*[^t]{\boldsymbol{w}}{} \\ \boldsymbol{v} & c  \end{pmatrix} := \begin{pmatrix}  A(\boldsymbol{z}) &  \tensor*[^t]{\boldsymbol{w}}{}(\boldsymbol{z}) \\ \boldsymbol{v}(\boldsymbol{z}) & c(\boldsymbol{z})  \end{pmatrix} \in \SO(d, \RR).  \end{align} Elementary properties of $K$ imply that \begin{align}\label{eqnElemPropSOd}
 \boldsymbol{v} \tensor*[^t]{\boldsymbol{v}}{} + c^2 =& 1 \\\nonumber \boldsymbol{v} A + c \boldsymbol{w} =& \boldsymbol{0}.  \end{align} We choose $\Dd$ so that \begin{align}\label{eqnRestrictionsOnD}
 c(\overline{\Dd})>0.\end{align}   A simple example of our shrinking targets $\mU_T$ thickened in spherical directions is $\Ss_T \Ee$ under the restriction (\ref{eqnRestrictionsOnD}), and this is target considered in Theorem~\ref{thmNAKResultRankOneFlow}.  More general spherical thickenings are considered in Section~\ref{secSdTGCatpothmNAK}, and these are obtained by removing the restriction (\ref{eqnRestrictionsOnD}) and by using certain Grenier boxes in their construction.  We will consider all of these thickenings in this paper.  Note that, in this paper, the renormalized target $\mU_{T_0}$ is constructed via the same method of thickening as the shrinking target $\mU_T$.

\subsection{Statement of main results}\label{secIntroAnyRankResults}  Our first result, Theorem~\ref{thmShrinkCuspNeigh}, is for shrinking targets constructed from Grenier boxes thickened in the stable horospherical directions.  These are the most general shrinking targets thickened in the stable directions that we consider.

\begin{theo}\label{thmShrinkCuspNeigh} Let $0<\eta <1$ be a fixed constant, $L \in G$,  $\Aa \subset \Aa_L$ be a bounded subset with measure zero boundary with respect to $\wrt {\boldsymbol{x}}$ and \[\widetilde{K} \subset K' \] be a measurable set with respect to $\wrt {\widetilde{k}}$.  Let $\boldsymbol{\alpha}$, $\boldsymbol{\gamma}$, $\beta^-_{ij}, \beta^+_{ij}$ be constants defining a neighborhood in $\Ss_1$ with lower height $\geq 1$ and define the following neighborhood of $\Ss_1$ \[\Cc:=\Cc_{T_-, T_+}:=\Cc(\boldsymbol{\alpha}, \boldsymbol{\gamma}, \widetilde{K},  \beta^-_{ij}, \beta^+_{ij})\] for constants $\infty \geq T_+ > T_- \geq 1$. Let  \[\Bb \subset \begin{cases}\left(T_- \right)\ZZ \backslash \RR & \textrm{ if } d =2 \\ \left(\frac 1 {\sqrt{d}}\left( \frac 3 4 \right)^{(d-1)/2} T_-^{d/2(d-1)} \right) \ZZ^{d-1} \bigg{\backslash} \RR^{d-1} & \textrm{ if }d \geq 3\end{cases}\] be a measurable set with respect to $\wrt {\widetilde{\boldsymbol{x}}}$.  For the family of neighborhoods of $\Gamma \backslash G$ \[\widetilde{\Cc}_TN^+(\Bb):=\Cc \Phi^{-\frac 1 d \log(T/T_-^{d/2(d-1)})}N^+(\Bb),\] then we have that \[T^{d-1} \int_{\Aa} \In_{\widetilde{\Cc}_T N^+(\Bb)}({Ln_-(\boldsymbol{x}) \Phi^t})~\wrt {\boldsymbol{x}} \xrightarrow[]{t \rightarrow \infty} T_-^{d/2}\mu\left(\Cc N^+(\Bb)\right)\int_{\RR^{d-1}} \In_{\Aa}(\boldsymbol{x})~\wrt{\boldsymbol{x}}\] uniformly for all $T \in [T_-^{d/2(d-1)}, e^{dt \eta}].$

\end{theo}

\begin{rema}  Since our result is for $t \rightarrow \infty$, we may assume, without loss of generality, that $e^{dt \eta} \geq T_-^{d/2(d-1)}$.   By Corollary~\ref{coroHaarMeaCTN}, the right-hand side is also equal to \[\lim_{t \rightarrow \infty} T^{d-1} \mu(\widetilde{\Cc}_T N^+(\Bb))\left(\int_{\TT^{d-1}} \In_{\Aa}(\boldsymbol{x})~\wrt{\boldsymbol{x}}\right).\]  Finally, note that one could obtain an explicit (and complicated) formula for $\mu\left(\Cc N^+(\Bb)\right)$ in terms of the variables $x_{ij}$, $y_\ell$ (recall that $\widetilde{K}$ is fixed) from (\ref{eqnFundDomCoord}) using Grenier coordinates.
\end{rema}

A special but useful case of Theorem~\ref{thmShrinkCuspNeigh}, which does not involve Grenier boxes, is Theorem~\ref{thmNANResultRankOneFlow} in which the shrinking target is a small thickening of the section $\Ss_T$.  Let $\varepsilon>0$ be small, $\boldsymbol{\widetilde{y}} \in \RR^{d-1}$, \[N^+_\varepsilon (\boldsymbol{\widetilde{y}}) := \left\{n_+(\boldsymbol{\widetilde{x}}) \in N^+ : \|\boldsymbol{\widetilde{x}} - \boldsymbol{\widetilde{y}}\|_\infty < \frac \varepsilon 2\right\} \quad \textrm{ and } \quad  N^+_\varepsilon:=N^+_\varepsilon(\boldsymbol{0}).\]  Translating $\Ss_T$ along the stable direction by $n_+(\boldsymbol{\widetilde{y}})$ and slightly thickening, we obtain our shrinking target $\Ss_TN^+_\varepsilon (\boldsymbol{\widetilde{y}})$.  Note that $\Ss_TN^+_\varepsilon (\boldsymbol{\widetilde{y}})$ has positive $\mu$-measure.

\begin{theo}\label{thmNANResultRankOneFlow}
Let $0<\eta <1$ and $T_0\geq1$ be fixed constants.  Let $L \in G$, $\boldsymbol{\widetilde{y}} \in \RR^{d-1}$, $\varepsilon>0$ be small, and $\Aa \subset \Aa_L$ be a bounded subset with measure zero boundary with respect to $\wrt {\boldsymbol{x}}$.  Then \[T^{d-1} \int_{\Aa} \In_{\Ss_T N^+_\varepsilon (\boldsymbol{\widetilde{y}})}({L n_-(\boldsymbol{x}) \Phi^t})~\wrt {\boldsymbol{x}} \xrightarrow[]{t \rightarrow \infty} \frac 1 {d \zeta(d)}  \varepsilon^{d-1} \int_{\RR^{d-1}} \In_{\Aa} ~\wrt {\boldsymbol{x}}\] uniformly for all $T \in [T_0,e^{dt \eta}].$
\end{theo}

\begin{rema} We may assume, without loss of generality, that $e^{dt \eta} \geq T_0$.  By Corollary~\ref{coroHaarMeaSTN}, the right-hand side is also equal to \[\lim_{t \rightarrow \infty} T^{d-1} \mu\left(\Ss_T N^+_\varepsilon (\boldsymbol{\widetilde{y}})\right)\left(  \int_{\TT^{d-1}} \In_{\Aa} ~\wrt {\boldsymbol{x}}\right).\]  We also note that $\varepsilon$ does not have to be very small.  By Theorem~\ref{thmShrinkCuspNeigh}, we can choose it to be any positive value strictly less than \[\begin{cases} \frac 1 2  T_0& \textrm{ if } d =2 \\  \frac 1 2 \left(\frac 1 {\sqrt{d}}\left( \frac 3 4 \right)^{(d-1)/2} T_0\right) & \textrm{ if } d \geq 3 \end{cases}.\]

\end{rema}

Our second result is for thickenings of $\Ss_T$ in spherical directions and is a special case of the generalizations for certain more general Grenier boxes, Theorems~\ref{thmThickenSubsetSectionSpherical},~\ref{thmTwoChartGeneralSpherical} and~\ref{thmThickenSubsetSectionSphericalGeneral}, that we give in Section~\ref{secSdTGCatpothmNAK}.  We can even remove the restriction (\ref{eqnRestrictionsOnD}).  Let \[\h_0 := \begin{cases} 1 & \textrm{ if } d =2 \\ \sqrt{d} \left(\frac 4 3\right)^{\frac{d-1}2} &\textrm{ if } d \geq3\end{cases}.\]

\begin{theo}\label{thmNAKResultRankOneFlow}  

Let $0<\eta <1$ and $T_0>\h_0$ be fixed constants. Let $L \in G$ and $\Aa \subset \Aa_L$ be a bounded subset with measure zero boundary with respect to $\wrt {\boldsymbol{x}}$ and let (\ref{eqnRestrictionsOnD}) hold.  Then \begin{align}\label{eqnthmNAKResultRankOneFlow}
 {T^{d-1}}\int_{\RR^{d-1}} \In_{\Aa \times\Ss_T  \Ee}(\boldsymbol{x}, {Ln_-(\boldsymbol{x}) \Phi^t})~\wrt {\boldsymbol{x}} \xrightarrow[]{t \rightarrow \infty}T_0^{d-1}\mu(\Ss_{T_0} \Ee)\left(\int_{\RR^{d-1}} \In_{\Aa}(\boldsymbol{x})~\wrt{\boldsymbol{x}}\right) \end{align} uniformly for all $T \in [T_0,e^{dt \eta}].$

\end{theo}

\begin{rema} We may assume, without loss of generality, that $e^{dt \eta} \geq T_0$. By Theorem~\ref{thnHaarMeaSTEvsST0E}, the right-hand side of (\ref{eqnthmNAKResultRankOneFlow}) is also equal to \[\lim_{t \rightarrow \infty} T^{d-1} \mu\left(\Ss_T \Ee\right)\left(  \int_{\TT^{d-1}} \In_{\Aa}(\boldsymbol{x})~\wrt {\boldsymbol{x}}\right).\]

\end{rema}

Our final result is the equidistribution of translated Farey sequences (defined in Section~\ref{subsecTranFareySeq}) and a generalization of Marklof's equidistribution result for the Farey sequence~\cite[Theorem~6]{Mar10}, which we quote as Theorem~\ref{thmMarklof}.\footnote{Even when $L=I_d$, the subsets of the translated Farey sequence in our result and the subsets of the Farey sequence in Theorem~\ref{thmMarklof}  are slightly different.  As Marklof's result and our result are key to the proofs of STHE for horospheres and translated horospheres, respectively, both results are stated for clarity.}  We use our final result as a geometrical and dynamical tool to prove STHE in our other results.  The subsets $\widehat{\If}_Q$ of the translated Farey sequence corresponding to $\Gamma L N^-$ are defined in (\ref{eqn:subsetTransFareySeq}).

\begin{theo}\label{thmGeneralMarklof}  Fix $\sigma \in \RR$.  Let $L \in G$, $\Aa$ be an open ball in $\RR^{d-1}$ of radius $r_0>0$, and \[f: \Aa \times \Ga \backslash G \rightarrow \RR\] be the product of a bounded continuous function with an indicator function of a subset $\Tt \subset \Aa \times \Ga \backslash G$ such that the set \begin{align}\label{eqnGeneralMarklofBound}
 \left\{(\boldsymbol{x}, M) : (\boldsymbol{x}, \tensor*[^t]{M}{^{-1}}) \in \Tt \right\} \end{align} has measure zero boundary with respect to $\wrt{\boldsymbol{x}}~\wrt \mu_H~\wrt s$.  Then, for $Q = e^{(d-1)(t-\sigma)}$, we have \begin{align*}
\lim_{t \rightarrow \infty} \frac {e^{-d(d-1)\sigma}} {\#(\widehat{\If}_Q)} & \sum_{\frac{\boldsymbol{\alpha}'}{\alpha_d} \in \widehat{\If}_{Q}} {f}\left(\frac{\boldsymbol{\alpha}'}{\alpha_d},  {L} n_-\left(\frac{\boldsymbol{\alpha}'}{\alpha_d}\right) \Phi^{t}\right) \\ \nonumber &= (d-1) \int_{\sigma}^{\infty}  \int_{\Aa \times \Gamma_H \backslash H} \widetilde{f} (\boldsymbol{x}, M \Phi^{-s})~\wrt{\boldsymbol{x}}~\wrt\mu_H(M)e^{-d(d-1)s}~\wrt s  \end{align*} with $\widetilde{f}(\boldsymbol{x}, M) := f(\boldsymbol{x}, \tensor*[^t]{M}{}^{-1} )$.

\end{theo}

\subsection{Outline of paper}  We prove all of our STHE results for the case $L = I_d$ first, deferring the case of generic $L \in G$ to after the definition and development of translated Farey sequences in Section~\ref{subsecTranFareySeq}.  Our renormalization technique and Marklof's result on the equidistribution of the Farey sequence (Theorem~\ref{thmMarklof}) are the two key ingredients in the proof of Theorem~\ref{thmNANResultRankOneFlow} given in Section~\ref{secNANResultRankOneFlow}.  Theorem~\ref{thmNAKResultRankOneFlow}, the first case of spherical thickenings, is proved in Section~\ref{secSdatpothmNAK}.  Cuspidal neighborhoods and Grenier coordinates (used to define Grenier boxes) are introduced and developed for our setting in Section~\ref{secCuspNeigh}, allowing for a proof of Theorem~\ref{thmShrinkCuspNeigh} in Section~\ref{subsecProofMainTheoremFullCase} under the assumption of disjointness (Lemma~\ref{lemmDisjointnessTransSections}).  Disjointness is shown for all of our setups for both stable horospherical and spherical thickenings in Section~\ref{secDisjTransNeigST} .  Three generalizations of our spherical thickening result are stated and proved in Section~\ref{secSdTGCatpothmNAK}, and the key ingredient in these proofs is a {\em reverse outer product Cholesky factorization.}  Theorem~\ref{thmGeneralMarklof}, the equidistribution of translated Farey sequences, is proved in Section~\ref{sec:proofGeneralMarklof}, and it is the key tool for proving our STHE results for generic $L \in G$ in Section~\ref{subsecUniEquiForTransHoro}.  Finally, we note that a key ingredient in the proof of Theorem~\ref{thmGeneralMarklof} (and also Theorem~\ref{thmMarklof}) is a version of Shah's theorem (Theorem~\ref{thm:TransShah:TransHoro} for the case of translated Farey sequences).  This version of Shah's theorem is a consequence of mixing and does not require Ratner's theorems.  We do, however, directly compute, in a simple way, orbit closures in Section~\ref{subsec:GammaDuplicates} for our (limited) setting.

\subsection*{Acknowledgements}  I wish to thank Jens Marklof for helpful discussions and the referee for helpful comments.

For the purpose of open access, the author has applied a CC BY public copyright licence (where permitted by UKRI, ‘Open Government Licence’ or ‘CC BY-ND public copyright licence’ may be stated instead) to any Author Accepted Manuscript version arising.  This study did not generate any new data.

\section{The geometry and equidistribution of the Farey sequence}\label{secEotFs}

A key ingredient in our proofs of STHE for the case $L = I_d$ is Marklof's result for the equidistribution of the Farey sequence.  In this section, we state Marklof's result and give a heuristic overview of the geometry of the Farey sequence, especially in relation to the horosphere.  We give a precise description of this geometry in Section~\ref{secTranslatedHorospheres}.

The key geometrical observation concerning the elements of the Farey sequence is that they, roughly speaking, anchor the horosphere to the section $\Ss_1$.  In particular, the local intersection of the horosphere with $\Ss_1$ is exactly one point (Theorem~\ref{thm:TranHoroInterSectS1AtAPoint}), and these intersection points correspond to, as shown in~\cite{Mar10}, the elements of the Farey sequence (see Theorem~\ref{thmTransFareyInterPointCorr} for a precise statement).  For conciseness, we identify these intersection points with their corresponding elements of the Farey sequence.  Under our flow $\Phi^t$ (for $t \geq 0$), these elements of the Farey sequence (i.e. these intersection points) remain in the section $\Ss_1$ and, moreover, move higher into the cusp within $\Ss_1$ (see Section~\ref{secNANResultRankOneFlow}).  Consequently, the dynamics of the horosphere under the flow is determined by the dynamics of the Farey sequence under the flow, the latter of which is governed by the powerful equidistribution result of Marklof.  This allows us to analyze the dynamics of the horosphere under the flow in a flexible way.  Note that, in the case $L = I_d$, we have that $ \Aa_{L}= \TT^{d-1}$ (Theorem~\ref{thm:HorosphereDuplicates}).

To state Marklof's result, let us introduce some notation.  Let $\wrt s$ be the Lebesgue measure on $\RR$.  Let $\widehat{\ZZ}^d := \{\boldsymbol{a} = (a_1, \cdots, a_d) \in \ZZ^d : \gcd(a_1, \cdots, a_d)=1\}$.  The elements of the {\em Farey sequence} \begin{align}\label{eqn:DefnFareySeq}
 \F_Q := \left\{\frac{\boldsymbol{p}} q \in [0, 1)^{d-1} : (\boldsymbol{p}, q) \in \widehat{\ZZ}^d, \quad 0 < q \leq Q \right\} \end{align} become equidistributed on $\Ss_T$ as $Q \rightarrow \infty$.  Precisely,

\begin{theo}[{\cite[Theorem~6]{Mar10}}]\label{thmMarklof}  Fix $\sigma \in \RR$.  Let $f: \TT^{d-1} \times \Ga \backslash G \rightarrow \RR$ be the product of a continuous bounded function with an indicator function of a subset $\Tt \subset \TT^{d-1} \times \Ga \backslash G$ such that the set \begin{align}
 \label{eqnOriginalMarklofBound}\left\{(\boldsymbol{x}, M) : (\boldsymbol{x}, \tensor*[^t]{M}{^{-1}}) \in \Tt \right\} \end{align} has measure zero boundary with respect to $\wrt{\boldsymbol{x}}~\wrt \mu_H~\wrt s$.  Then, for $Q = e^{(d-1)(t-\sigma)}$, we have \begin{align*}
 \lim_{t \rightarrow \infty}& \frac {e^{-d(d-1) \sigma}} {|\F_Q|} \sum_{\boldsymbol{r} \in \F_Q} f(\boldsymbol{r}, n_-(\boldsymbol{r}) \Phi^t) \\ = & d(d-1)  \int_\sigma^\infty \int_{\TT^{d-1} \times \Ga_H \backslash H} \widetilde{f}(\boldsymbol{x}, M \Phi^{-s})~\wrt{\boldsymbol{x}}~\wrt \mu_H(M) e^{-d(d-1)s}~\wrt s \end{align*} with $\widetilde{f}(\boldsymbol{x}, M) := f(\boldsymbol{x}, \tensor*[^t]{M}{}^{-1} )$.
 
\end{theo}  

\noindent Note that \begin{align} \label{eqnAsypFQOriginal}
 |\F_Q| \sim \frac {Q^d} {d \zeta(d)} \textrm{ as } Q \rightarrow \infty. \end{align}  We will always set $\sigma =0$ when using this theorem as this corresponds to our section $\Ss_1$.
 
 \begin{rema}  The generalization of Theorem~\ref{thmMarklof} to any translated Farey sequence is given in Theorem~\ref{thmGeneralMarklof}.

\end{rema}

\section{Renormalization:  proof of Theorem~\ref{thmNANResultRankOneFlow} for $L=I_d$}\label{secNANResultRankOneFlow}
  
In this section, we prove Theorem~\ref{thmNANResultRankOneFlow} for the case $L=I_d$.  Let $\wrt {\boldsymbol{\widetilde{x}}}$ be the Haar measure on $\RR^{d-1}$, induced by the Haar measure $\wrt {\widetilde{n}}$ on $N^+$, normalized so that the induced measure on $\TT^{d-1}$ is a probability measure. Let $Q := e^{(d-1)t}$.  Recall the definition of $N^+_\varepsilon (\boldsymbol{\widetilde{y}})$ from Section~\ref{secIntroAnyRankResults} and $\Aa$ from Section~\ref{subsec:STHEandPrelimNotation}.  Let us translate and thicken $\Ss_T$ in the stable directions with $N^+_\varepsilon (\boldsymbol{\widetilde{y}})$.  The notions of stable directions and of translating and thickening in these directions are well-defined in $G$ because we can take the coordinates given in (\ref{eqnParametrizationForG}) in Section~\ref{subsecVolumes}.  In particular, there is a unique point of intersection of a (piece of a) horosphere with the section $\Ss_1$.  We will estimate the volume of the intersection of a translate of $\Aa$ with the (translated and) thickened section $\Ss_T N^+_\varepsilon (\boldsymbol{\widetilde{y}})$.

Note first that, for $\boldsymbol{\widetilde{x}} \in \RR^{d-1}$, conjugation gives us:  \begin{align}\label{eqnConjy1}n_+(\boldsymbol{\widetilde{x}}) = \Phi^{\frac 1 d \log(T/T_0)} n_+\left(\frac {T_0} T \boldsymbol{\widetilde{x}}\right) \Phi^{-\frac 1 d \log(T/T_0)}.\end{align}

Also recall that, for $B \subset \Ga \backslash G$, $g, h \in G$, we have that  \begin{align}\label{eqnIndictorTranslation} \In_{B h}(g) = \In_{B}(g h^{-1}).
  \end{align}  (Note that we use the convention that the argument of the indicator function is considered as an element in the homogeneous space $\Ga \backslash G$.)

Recall that taking the transpose together with the inverse is an isometry and that (see~\cite[Equation(3.54)]{Mar10}) the point \[\Ga \tensor*[^t]{\left(n_-(\boldsymbol{r}) \Phi^t\right)}{}^{-1}  \in \Ss_1.\]  If the point $\Ga\tensor*[^t]{\left(n_-(\boldsymbol{r}) \Phi^t\right)}{}^{-1} $ intersects the section $\Ss_T \subset \Ss_1$, then the $N^+$-volume of the horosphere through this unique intersection point and meeting the thickening is given by \[f_{T, \varepsilon}(\boldsymbol{r}, n_-(\boldsymbol{r}) \Phi^t) := \int_{N^+} \In_{\Aa \times\Ss_T N^+_\varepsilon(\boldsymbol{\widetilde{y}})}\left(\boldsymbol{r}, \tensor*[^t]{\left(n_-(\boldsymbol{r}) \Phi^t\right)}{}^{-1} n_+(\boldsymbol{\widetilde{x}})\right) ~\wrt {\widetilde{n}(\boldsymbol{\widetilde{x}})}.\]  Recall that we are given $T \in [T_0,e^{dt \eta}]$ and hence, in particular, we have that \begin{align}\label{eqnChangeOfRate}
 \lim_{t \rightarrow \infty}\left(t-\frac 1 d \log(T/T_0)\right)= \infty, \end{align} which also allows for equidistribution when the Farey points are pushed with this slower time.

\begin{lemm}\label{lemmUnstableVolumes}  Let \[\widetilde{Q} = e^{(d-1)\left(t-\frac 1 d \log(T/T_0)\right)}.\]  
We have that 
 \[\frac {T^{d-1}} {Q^d} \sum_{\boldsymbol{r} \in \F_Q} f_{T, \varepsilon}(\boldsymbol{r}, n_-(\boldsymbol{r}) \Phi^t) = \frac {T_0^{d-1}} {\widetilde{Q}^{d}}\sum_{\boldsymbol{r} \in \F_{\widetilde{Q}}} f_{T_0, \varepsilon}\left(\boldsymbol{r}, n_-(\boldsymbol{r}) \Phi^{t- \frac 1 d \log(T/T_0)}\right).\] 
\end{lemm}

\begin{proof}

 Since \[ \In_{\Ss_T n_+(\boldsymbol{\widetilde{y}})}\left(\tensor*[^t]{\left(n_-(\boldsymbol{r}) \Phi^t\right)}{}^{-1} n_+(\boldsymbol{\widetilde{y}})\right) = \In_{\Ss_T }\left(\tensor*[^t]{\left(n_-(\boldsymbol{r}) \Phi^t\right)}{}^{-1} \right)\] holds by (\ref{eqnIndictorTranslation}), we have that multiplying on the right by $n_+(\boldsymbol{\widetilde{y}})$ and using the right invariance of the measure $\wrt{\widetilde{n}}$ gives \begin{align*}f_{T, \varepsilon}(\boldsymbol{r}, & n_-(\boldsymbol{r}) \Phi^t) =\int_{N^+} \In_{\Aa \times\Ss_T  N^+_\varepsilon(\boldsymbol{0})}\left(\boldsymbol{r}, \tensor*[^t]{\left(n_-(\boldsymbol{r}) \Phi^t\right)}{}^{-1} n_+(\boldsymbol{\widetilde{x}})\right) ~\wrt {\widetilde{n}(\boldsymbol{\widetilde{x}})}.
  \end{align*}  We note that, using (\ref{eqnIndictorTranslation}), the multiplication right-translates the center of the thickening to the point $\Ga \tensor*[^t]{\left(n_-(\boldsymbol{r}) \Phi^t\right)}{}^{-1}$.

Now consider slowing the speed of the flow pushing the Farey points from $t$ to $t - \frac 1 d \log(T/T_0)$.  Using the observation that \begin{align}\label{eqnPullingBackOfSection}
\Ss_T = \Ss_{T_0} \Phi^{-\frac 1 d \log(T/T_0)}  \end{align}holds, we see that $\Ga\tensor*[^t]{\left(n_-(\boldsymbol{r}) \Phi^t\right)}{}^{-1}  $ meets $\Ss_T$ if and only if $\Ga\tensor*[^t]{\left(n_-(\boldsymbol{r}) \Phi^{t - \frac 1 d \log(T/T_0)}\right)}{}^{-1} $ meets $\Ss_{T_0}$.  Let denote these intersections as follows:  \begin{align*}
\I(\boldsymbol{r}, t, T) &:= \Ga\tensor*[^t]{\left(n_-(\boldsymbol{r}) \Phi^{t}\right)}{}^{-1}  \cap \Ss_T \\ \J\left(\boldsymbol{r}, t ,T\right) &:=\Ga \tensor*[^t]{\left(n_-(\boldsymbol{r}) \Phi^{t - \frac 1 d \log(T/T_0)}\right)}{}^{-1}  \cap \Ss_{T_0}.\end{align*}  Thus, in particular, we have that the mapping \[  z \mapsto z \Phi^{-\frac 1 d \log(T/T_0)}\] is a bijection \[\varphi:\J(\boldsymbol{r}, t, T) \longrightarrow \I(\boldsymbol{r}, t, T)\] whenever either set is nonempty.  Since $\Ga \tensor*[^t]{\left(n_-(\boldsymbol{r}) \Phi^t\right)}{}^{-1}  \in \Ss_1$ for $t$ large enough, applying the flow $\Phi^{\widetilde{t}}$ to $n_-(\boldsymbol{r})$  (i.e. considering $n_-(\boldsymbol{r}) \Phi^{t+\widetilde{t}} $) for $\widetilde{t} \geq 0$ keeps it in $\Ss_1$ and, when $\widetilde{t}$ is large enough, it will be (and stay) in $\Ss_{T_0}$ and hence these sets will be nonempty and continue to be nonempty for increasing $\widetilde{t}$, which is guaranteed by (\ref{eqnChangeOfRate}).

Recalling that $\Ss_T \subset \Ss_{T_0}$, we have that the (piece of the) horosphere $z N^+_\delta(\boldsymbol{0})$ for $z \in \J(\boldsymbol{r}, t, T)$ (whenever the set is nonempty) corresponds to the (piece of the) horosphere through $\varphi(z)$, which is namely the set \[z N^+_\delta(\boldsymbol{0}) \Phi^{ -\frac 1 d \log(T/T_0)} = \varphi(z) N^+_{\delta T/T_0}(\boldsymbol{0}),\] where the equality follows by an application of (\ref{eqnConjy1}).   We want the pieces to exactly match (i.e. $\varphi(z)N^+_\varepsilon(\boldsymbol{0}) = \varphi(z) N^+_{\delta T/T_0}(\boldsymbol{0})$), and, hence, we have that \[\varepsilon = \delta \frac{T}{T_0}.\]

  Thus we have
  
\begin{align}\label{eqnCompareExpandingHaarVol}f_{T_0, \delta}(\boldsymbol{r},  n_-(\boldsymbol{r}) \Phi^{t-\frac 1 d \log(T/T_0)}) &= \int_{N^+} \In_{\Aa \times\Ss_{T_0}  N^+_{\varepsilon T_0/T}(\boldsymbol{0})}\left(\boldsymbol{r},\tensor*[^t]{\left(n_-(\boldsymbol{r}) \Phi^{t - \frac 1 d \log(T/T_0)}\right)}{}^{-1}n_+(\boldsymbol{\widetilde{x}})\right) ~\wrt {\widetilde{n}(\boldsymbol{\widetilde{x}})}\\ \nonumber &=\In_{\Aa \times\Ss_{T_0}}\left(\boldsymbol{r}, \tensor*[^t]{\left(n_-(\boldsymbol{r}) \Phi^{t - \frac 1 d \log(T/T_0)}\right)}{}^{-1}\right) \int_{N^+} \In_{N^+_{\varepsilon T_0/T}(\boldsymbol{0})}\left(n_+(\boldsymbol{\widetilde{x}})\right) ~\wrt {\widetilde{n}(\boldsymbol{\widetilde{x}})}\\ \nonumber&=\In_{\Aa \times\Ss_{T_0}}\left(\boldsymbol{r}, \tensor*[^t]{\left(n_-(\boldsymbol{r}) \Phi^{t - \frac 1 d \log(T/T_0)}\right)}{}^{-1}\right) \int_{\RR^{d-1}} \In_{\left[-\frac{\varepsilon T_0}{2T},\frac{\varepsilon T_0}{2T}\right]^{d-1} }\left(\boldsymbol{\widetilde{x}}\right) ~\wrt {\boldsymbol{\widetilde{x}}}\\ \nonumber&=\varepsilon^{d-1}\frac {T_0^{d-1}}{T^{d-1}}\In_{\Aa \times\Ss_{T_0}}\left(\boldsymbol{r}, \tensor*[^t]{\left(n_-(\boldsymbol{r}) \Phi^{t - \frac 1 d \log(T/T_0)}\right)}{}^{-1}\right) \\ \nonumber&= \frac {T_0^{d-1}}{T^{d-1}}f_{T_0, \varepsilon}\left(\boldsymbol{r}, n_-(\boldsymbol{r}) \Phi^{t-\frac 1 d \log(T/T_0)}\right).
  \end{align} Here we have used the fact that \[\int_{N^+} \In_{N^+_{\widetilde{\varepsilon}}(\boldsymbol{0})}\left(n_+(\boldsymbol{\widetilde{x}})\right) ~\wrt {\widetilde{n}(\boldsymbol{\widetilde{x}})} = \int_{\RR^{d-1}} \In_{\left[-\frac{\widetilde{\varepsilon}} 2, \frac {\widetilde{\varepsilon}}2\right]^{d-1}}(\boldsymbol{\widetilde{x}})~\wrt {\boldsymbol{\widetilde{x}}}. \]

Let $\widehat{Q}:= Q T^{-\frac {d-1}d}$.  Now, at time $t$, we have that $\boldsymbol{r} \in \F_Q$, but, in particular, if it determines a $z \in \I(\boldsymbol{r}, t, T)$, then $\boldsymbol{r} \in\F_{\widehat{Q}}$ because all $\boldsymbol{r} \in \F_Q \backslash\F_{\widehat{Q}}$ will correspond to points on the horosphere that lie in $\Ss_1 \backslash \Ss_{T}$ and hence the following three sums, which gives the sum of the Haar measures of the intersections, are equal: \[ \sum_{\boldsymbol{r} \in \F_Q} f_{T, \varepsilon}(\boldsymbol{r}, n_-(\boldsymbol{r}) \Phi^t) =  \sum_{\boldsymbol{r} \in \F_{\widetilde{Q}}} f_{T, \varepsilon}(\boldsymbol{r}, n_-(\boldsymbol{r}) \Phi^t) = \sum_{\boldsymbol{r} \in \F_{\widehat{Q}}} f_{T, \varepsilon}(\boldsymbol{r}, n_-(\boldsymbol{r}) \Phi^t).\]  We note that horospheres through different right-translated Farey points  $\Ga \tensor*[^t]{\left(n_-(\boldsymbol{r}) \Phi^t\right)}{}^{-1}$ are disjoint and thus the sum is over disjoint pieces, giving the total measure of the pieces with respect to the Haar measure $~\wrt {\widetilde{n}}$.

Note that $Q^d$ is, by direct computation, the Haar measure of $ \Phi^{-t}N_-(\TT^{d-1})\Phi^t$ on $N^-$  (this is the Haar measure normalized so to induce the probability Haar measure $\wrt {\boldsymbol{x}}$) and, likewise, $\widetilde{Q}^d$ is the the Haar measure of $ \Phi^{-t + \frac 1 d \log(T/T_0)}N_-(\TT^{d-1})\Phi^{t - \frac 1 d \log(T/T_0)}$ on $N^-$.  Applying $\varphi$ to the horosphere on $N^+$ changes the Haar measure by the factor $\frac {T^{d-1}}{T_0^{d-1}}$ and hence preserves the ratio of Haar measures:  \begin{align}\label{eqnRatioHaarMeasuresNANCase}
 \frac {1} {Q^d} \sum_{\boldsymbol{r} \in \F_Q} f_{T, \varepsilon}(\boldsymbol{r}, n_-(\boldsymbol{r}) \Phi^t) = \frac 1 {\widetilde{Q}^{d}}\sum_{\boldsymbol{r} \in \F_{\widetilde{Q}}} f_{T_0, \delta}\left(\boldsymbol{r}, n_-(\boldsymbol{r}) \Phi^{t- \frac 1 d \log(T/T_0)}\right). \end{align}  

Using (\ref{eqnCompareExpandingHaarVol}), we obtain the following:  \[\frac {T^{d-1}} {Q^d} \sum_{\boldsymbol{r} \in \F_Q} f_{T, \varepsilon}(\boldsymbol{r}, n_-(\boldsymbol{r}) \Phi^t) = \frac {T_0^{d-1}} {\widetilde{Q}^{d}}\sum_{\boldsymbol{r} \in \F_{\widetilde{Q}}} f_{T_0, \varepsilon}\left(\boldsymbol{r}, n_-(\boldsymbol{r}) \Phi^{t- \frac 1 d \log(T/T_0)}\right).\] 
 
 This gives the desired result.
\end{proof}

\begin{rema}
Note that, in the lemma, when we apply $\varphi$, we also pull down $S_T$ to $S_{T_0}$, and, thus, the number of intersection points of the expanding horosphere with $S_T$ is the same as with $S_{T_0}$.

\end{rema}

Lemma~\ref{lemmUnstableVolumes} allows us to apply Theorem~\ref{thmMarklof} to ${T^{d-1}}f_{T, \varepsilon}(\boldsymbol{r}, n_-(\boldsymbol{r}) \Phi^t)$ as follows.  Recall that $|\F_Q| \sim \frac {Q^d} {d \zeta(d)}$ as $Q \rightarrow \infty$.

We have that \begin{align}\label{eqnEquiOnSectTransA}\lim_{t \rightarrow \infty} &\frac {T^{d-1}} {|\F_Q|} \sum_{\boldsymbol{r} \in \F_Q}  f_{T, \varepsilon}(\boldsymbol{r}, n_-(\boldsymbol{r}) \Phi^t)   \\ \nonumber&=\lim_{t \rightarrow \infty} \frac {T_0^{d-1}} {|\F_{\widetilde{Q}}|} \sum_{\boldsymbol{r} \in \F_{\widetilde{Q}}}  f_{T_0, \varepsilon}(\boldsymbol{r}, n_-(\boldsymbol{r}) \Phi^{t- \frac 1 d \log(T/T_0)})\\ \nonumber&=  \lim_{t \rightarrow \infty} \frac {T_0^{d-1}} {|\F_{\widetilde{Q}}|} \sum_{\boldsymbol{r} \in \F_{\widetilde{Q}}} \varepsilon^{d-1}\In_{\Aa \times\Ss_{T_0}}\left(\boldsymbol{r}, \tensor*[^t]{\left(n_-(\boldsymbol{r}) \Phi^{t - \frac 1 d \log(T/T_0)}\right)}{}^{-1}\right) \\ \nonumber&=  \lim_{t \rightarrow \infty} \frac {T_0^{d-1}} {|\F_{\widetilde{Q}}|} \sum_{\boldsymbol{r} \in \F_{\widetilde{Q}}} \varepsilon^{d-1}\widetilde{\In}_{\Aa \times\Ss_{T_0}}\left(\boldsymbol{r},n_-(\boldsymbol{r}) \Phi^{t - \frac 1 d \log(T/T_0)}\right)\\ \nonumber&
  = d(d-1)T_0^{d-1}\varepsilon^{d-1} \int_0^\infty \int_{\TT^{d-1} \times \Ga_H \backslash H} \In_{\Aa \times \Ss_{T_0}}(\boldsymbol{x},  M \Phi^{-s})~\wrt{\boldsymbol{x}}~\wrt \mu_H(M) e^{-d(d-1)s}~\wrt s 
  \\ \nonumber&
  = d(d-1)T_0^{d-1}\varepsilon^{d-1} \int_{\frac 1 d \log(T_0)}^\infty \int_{\TT^{d-1}} \In_{\Aa}(\boldsymbol{x})~\wrt{\boldsymbol{x}} \ e^{-d(d-1)s}~\wrt s 
    \\\nonumber & = \varepsilon^{d-1}  \int_{\TT^{d-1}} \In_{\Aa}(\boldsymbol{x})~\wrt{\boldsymbol{x}}.
   \end{align}  Note, in particular, that, using the isometry given by taking transpose and inverse, we have applied Theorem~\ref{thmMarklof} in the third equality to a product of a continuous bounded function with an indicator function, as required.

Since $Q^d = e^{d(d-1)t}$ is the Haar measure of  $\Phi^{-t}N_-(\TT^{d-1})\Phi^t$ on $N^-$ and we have that\[\Ga \tensor*[^t]{\left(n_-(\boldsymbol{r}) \Phi^t\right)}{}^{-1}  \in \Ss_1,\] the proportion of the translates of $\Aa$ meeting $\Ss_T N^+_\varepsilon (\boldsymbol{\widetilde{y}})$ for all $T\geq T_0$ is uniformly (as long as $\lim_{t \rightarrow \infty} \frac {\log(T^{1/d})}  {t} < 1$) the following:

\[{T^{d-1}}\int_{\TT^{d-1}} \In_{\Aa \times\Ss_T N^+_\varepsilon (\boldsymbol{\widetilde{y}})}(\boldsymbol{x}, {n_-(\boldsymbol{x}) \Phi^t})~\wrt {\boldsymbol{x}} \xrightarrow[]{t \rightarrow \infty} \frac 1 {d \zeta(d)} \varepsilon^{d-1}  \int_{\TT^{d-1}} \In_{\Aa}(\boldsymbol{x}) ~\wrt {\boldsymbol{x}}.\]  Here the indicator function is preserved under the isometry given by taking transpose and inverse.  This proves Theorem~\ref{thmNANResultRankOneFlow}.

\section{Spherical directions:  proof of Theorem~\ref{thmNAKResultRankOneFlow} for $L=I_d$} \label{secSdatpothmNAK}

We now consider, in a simplified setting, thickenings in spherical directions or, equivalently (see (\ref{eqn:SphereIsKPrimeModK})), by subsets of $\SO(d, \RR)$.  The simplified setting is as follows.  First, we require that thickenings be given by a chart contained in an open hemisphere because, as we will prove in Section~\ref{subsecSpherDirectionsDisjoint}, such thickenings do not self-intersect.  Second, we restrict the set we thicken to all of $\Ss_T$.  In particular, in this section, we will prove Theorem~\ref{thmNAKResultRankOneFlow} in the case $L=I_d$.  We defer thickenings by more general sets, including the whole sphere, and thickenings of more general Grenier boxes in $\Ss_T$ to Section~\ref{secSdTGCatpothmNAK}.

We adapt the proof of Theorem~\ref{thmNANResultRankOneFlow} to give a proof of Theorem~\ref{thmNAKResultRankOneFlow}.  Recall the definitions of $\Dd, \Ee, \boldsymbol{v},$ and $c$ from Section~\ref{secConstructShrinkTargets}.   Let \begin{align}\label{eqnzprime}
 \boldsymbol{z'} := c(\boldsymbol{z})^{-1} \boldsymbol{v}(\boldsymbol{z}),\end{align} \[c_{\min}:= \min(\left\{c(\boldsymbol{z}) : \boldsymbol{z} \in\overline{\Dd} \right\}).\]  When (\ref{eqnRestrictionsOnD}) holds, we can change coordinates from coordinates involving the sphere (via (\ref{eqn:SphereIsKPrimeModK})) to coordinates involving the expanding horospherical subgroup $N^+$ (for the precise definition of the latter coordinates, see (\ref{eqnParametrizationForG}) and the beginning of Section~\ref{subsecVolumes}):  \begin{align}\label{eqnDecompSOintoN}
\Phi^{-s} E(\boldsymbol{z})^{-1} n_+(-\boldsymbol{z'}) = &\begin{pmatrix}  e^s(A - \tensor*[^t]{\boldsymbol{w}}{}\boldsymbol{z'}) & e^s ( \tensor*[^t]{\boldsymbol{w}}{})\\  0 & e^{-(d-1)s} c\end{pmatrix} \\\nonumber = &\begin{pmatrix}  c^{1/(d-1)}(A - \tensor*[^t]{\boldsymbol{w}}{}\boldsymbol{z'}) & c^{-1}  e^{ds} ( \tensor*[^t]{\boldsymbol{w}}{})\\  0 & 1\end{pmatrix} \begin{pmatrix} e^{s} c^{-1/(d-1)}I_{d-1} & \tensor*[^t]{\boldsymbol{0}}{}  \\ 0 & e^{-(d-1)s} c\end{pmatrix} \\\nonumber = &\begin{pmatrix}  c^{1/(d-1)}(A - \tensor*[^t]{\boldsymbol{w}}{}\boldsymbol{z'}) & c^{-1}  e^{ds} ( \tensor*[^t]{\boldsymbol{w}}{})\\  0 & 1\end{pmatrix} \Phi^{-s + \log c/ (d-1)}.
  \end{align}  Here the first equality follows by (\ref{eqnzprime}).  It is immediate that \begin{align}\label{eqnSOMinNDet1}
 \det (A - \tensor*[^t]{\boldsymbol{w}}{}\boldsymbol{z'}) = c^{-1}. \end{align} 
 
Also, setting $s=0$ in (\ref{eqnDecompSOintoN}), taking the inverse, and rewriting gives the following: 
 \begin{align}\label{eqnDecompNintoSO}
\Phi^{-s} n_+(\boldsymbol{z'}) E(\boldsymbol{z}) & =  \Phi^{-s-\log c/ (d-1)}\begin{pmatrix}  c^{1/(d-1)}(A - \tensor*[^t]{\boldsymbol{w}}{}\boldsymbol{z'}) & c^{-1}  ( \tensor*[^t]{\boldsymbol{w}}{})\\  0 & 1\end{pmatrix}^{-1}  \\\nonumber & =  \Phi^{-s-\log c/ (d-1)}\begin{pmatrix}  c^{-1/(d-1)}(A - \tensor*[^t]{\boldsymbol{w}}{}\boldsymbol{z'})^{-1} & -c^{-d/(d-1)} (A - \tensor*[^t]{\boldsymbol{w}}{}\boldsymbol{z'})^{-1}(\tensor*[^t]{\boldsymbol{w}}{})\\  0 & 1\end{pmatrix} \\\nonumber & = \begin{pmatrix}  c^{-1/(d-1)}(A - \tensor*[^t]{\boldsymbol{w}}{}\boldsymbol{z'})^{-1} & -e^{ds} (A - \tensor*[^t]{\boldsymbol{w}}{}\boldsymbol{z'})^{-1}(\tensor*[^t]{\boldsymbol{w}}{})\\  0 & 1\end{pmatrix} \Phi^{-s-\log c/ (d-1)}.
  \end{align}

 Consider the family of thickenings in spherical directions: \begin{align}\label{eqnDistortedSphericalFam}
 \{\Ss_{T} \Ee : T \geq T_0\}. \end{align}  Observe that, unlike for $N^+$, the action of the flow $\Phi^{-\frac 1 d \log (T/T_0)}$ is not uniform on $\Ee$ and, thus, the ``shape'' of these neighborhoods will change under different $T$.  

Using (\ref{eqnDecompSOintoN}), let us express $\Ss_{T} \Ee$ in terms of the coordinates involving $N^+$:

\begin{lemm}\label{lemmMMyCoordinatesForSE}
We have that \[\Ss_{T} \Ee = \left\{\Ss_{Tc^{-d/(d-1)}} n_+(\boldsymbol{z'}):  \boldsymbol{z} \in \Dd \right\}\] 
\end{lemm}

\begin{proof}  Applying (\ref{eqnDecompSOintoN}), we have that \begin{align*} \Ss_{T}& \Ee \\ =& \Ga \backslash \Ga H \left\{\begin{pmatrix}  c^{1/(d-1)}(A - \tensor*[^t]{\boldsymbol{w}}{}\boldsymbol{z'}) & c^{-1}  e^{ds} ( \tensor*[^t]{\boldsymbol{w}}{})\\  0 & 1\end{pmatrix} \Phi^{-s + \log c/ (d-1)}  n_+(\boldsymbol{z'}) :  \boldsymbol{z} \in \Dd \textrm{ and } s \geq {\frac 1 d \log T}\right\}.
  \end{align*}
Fix $\boldsymbol{z}$ and $s$ temporarily.  By (\ref{eqnSOMinNDet1}), have that \[H \begin{pmatrix}  c^{1/(d-1)}(A - \tensor*[^t]{\boldsymbol{w}}{}\boldsymbol{z'}) & c^{-1}  e^{ds} ( \tensor*[^t]{\boldsymbol{w}}{})\\  0 & 1\end{pmatrix}  = H\] and, consequently,\[\Ga \backslash \Ga H \begin{pmatrix}  c^{1/(d-1)}(A - \tensor*[^t]{\boldsymbol{w}}{}\boldsymbol{z'}) & c^{-1}  e^{ds} ( \tensor*[^t]{\boldsymbol{w}}{})\\  0 & 1\end{pmatrix} \Phi^{-s + \log c/ (d-1)}  n_+(\boldsymbol{z'}) = \Ga \backslash \Ga H  \Phi^{-s + \log c/ (d-1)}  n_+(\boldsymbol{z'}).\]  Taking the union over all $\boldsymbol{z} \in \Dd$ and $s \geq {\frac 1 d \log T}$ yields the desired result.
\end{proof}

 As in the proof of Theorem~\ref{thmNANResultRankOneFlow}, we have that, if the point $\Ga\tensor*[^t]{\left(n_-(\boldsymbol{r}) \Phi^t\right)}{}^{-1} $ intersects the section $\Ss_1$, then, using Lemma~\ref{lemmMMyCoordinatesForSE}, the $N^+$-volume of the horosphere through this unique intersection point and meeting the neighborhood $\Ss_{T} \Ee$ is given by \[\widehat{f}_{T, \Dd}(\boldsymbol{r}, n_-(\boldsymbol{r}) \Phi^t) := \int_{N^+} \In_{\Aa \times\left\{\Ss_{Tc^{-d/(d-1)}} n_+(\boldsymbol{z'}):  \boldsymbol{z} \in \Dd \right\}}\left(\boldsymbol{r}, \tensor*[^t]{\left(n_-(\boldsymbol{r}) \Phi^t\right)}{}^{-1} n_+(\boldsymbol{\widetilde{x}})\right) ~\wrt {\widetilde{n}(\boldsymbol{\widetilde{x}})}.\]  More generally, for any $\alpha >0$, we have that the $N^+$-volume of the horosphere through the unique intersection point and meeting the neighborhood \[\left\{\Ss_{Tc^{-d/(d-1)}} n_+(\alpha \boldsymbol{z'}):  \boldsymbol{z} \in \Dd \right\}\] is given by \begin{align*}
\widehat{f}_{T, \Dd, \alpha}(\boldsymbol{r},  n_-(\boldsymbol{r}) \Phi^t)  = \int_{N^+} \In_{\Aa \times\left\{\Ss_{T c^{-d/(d-1)}} n_+\left(\alpha \boldsymbol{z'}\right):  \boldsymbol{z} \in \Dd \right\}}\left(\boldsymbol{r}, \tensor*[^t]{\left(n_-(\boldsymbol{r}) \Phi^t\right)}{}^{-1} n_+(\boldsymbol{\widetilde{x}})\right) ~\wrt {\widetilde{n}(\boldsymbol{\widetilde{x}})}.\end{align*} 

Now note that $\Ss_{T c^{-d/(d-1)}} \subset \Ss_T$ and $\Ss_{T_0 c^{-d/(d-1)}} \subset \Ss_{T_0}$ because $c \leq1$.  Moreover, $c$ can be $1$ because $\boldsymbol{0} \in \Dd$.  Thus, we have that \[\bigcup_{\boldsymbol{z} \in \Dd} \Ss_{T c^{-d/(d-1)}} = \Ss_T \textrm{ and }  \bigcup_{\boldsymbol{z} \in \Dd} \Ss_{T_0 c^{-d/(d-1)}} = \Ss_{T_0}.\]

Let $\widetilde{u} \in \Ss_{T}$.  Then the intersection of the horosphere through $\widetilde{u}$ with the neighborhood \[\left\{\Ss_{T c^{-d/(d-1)}} n_+\left(\alpha \boldsymbol{z'}\right):  \boldsymbol{z} \in \Dd \right\}\] is the set \[\widetilde{u}N_{T, \alpha, \Dd}:=\widetilde{u}N_{\widetilde{u}, T, \alpha,\Dd}:=\widetilde{u}N^+ \cap \left\{\Ss_{T c^{-d/(d-1)}} n_+\left(\alpha\boldsymbol{z'}\right):  \boldsymbol{z} \in \Dd \right\}.\] Let ${u} \in \Ss_{T_0}$.  By (\ref{eqnConjy1}), we have that \[uN_{T_0,\alpha,\Dd}\Phi^{-\frac 1 d \log (T/T_0)}=u\Phi^{-\frac 1 d \log (T/T_0)} \left\{n_+\left(\frac T {T_0}\boldsymbol{v}\right) : n_+(\boldsymbol{v}) \in N_{T_0,\alpha,\Dd}\right\}.\]  Recalling that the intersection of any expanding horosphere with the section $\Ss_1$ is a unique point, we have that \begin{align}\label{eqnSphereHoroEquivalence1}u\Phi^{-\frac 1 d \log (T/T_0)} N_{T,\alpha T/T_0,\Dd} = u\Phi^{-\frac 1 d \log (T/T_0)} \left\{n_+\left(\frac T {T_0}\boldsymbol{v}\right) : n_+(\boldsymbol{v}) \in N_{T_0,\alpha,\Dd}\right\}.
  \end{align}  Consequently, we have shown that \begin{align}\label{eqnSphereHoroEquivalence2}u\Phi^{-\frac 1 d \log (T/T_0)} N_{T,\alpha T/T_0,\Dd}=uN_{T_0,\alpha,\Dd}\Phi^{-\frac 1 d \log (T/T_0)}.
  \end{align}

We now show the analog of Lemma~\ref{lemmUnstableVolumes}.

\begin{lemm}\label{lemmUnstableVolumesNAKCase}  Let \[\widetilde{Q} = e^{(d-1)\left(t-\frac 1 d \log(T/T_0)\right)}.\]  
We have that \[\frac {T^{d-1}} {Q^d} \sum_{\boldsymbol{r} \in \F_Q} \widehat{f}_{T, \Dd}(\boldsymbol{r}, n_-(\boldsymbol{r}) \Phi^t)  = \frac {T_0^{d-1}} {\widetilde{Q}^{d}}\sum_{\boldsymbol{r} \in \F_{\widetilde{Q}}}\widehat{f}_{T_0, \Dd}(\boldsymbol{r}, n_-(\boldsymbol{r}) \Phi^{t-\frac 1 d \log (T/T_0)}).\] \end{lemm}
\begin{proof}
Recall the definitions of \[\I(\boldsymbol{r}, t, T), \quad \J(\boldsymbol{r}, t, T), \quad \varphi\] from the proof of Lemma~\ref{lemmUnstableVolumes}.  Let $w \in \J(\boldsymbol{r}, t, T)$.  We have that the intersection of the horosphere through $w$ with $\left\{\Ss_{T_0c^{-d/(d-1)}} n_+(\beta \boldsymbol{z'}):  \boldsymbol{z} \in \Dd \right\}$, namely the set $wN_{T_0,\beta,\Dd}$, corresponds to the (piece of the) horosphere through $\varphi(w)$, namely the set \begin{align}\label{eqnSphereHoroEquivalence3}wN_{ T_0,\beta,\Dd}\Phi^{-\frac 1 d \log (T/T_0)} = \varphi(w)N_{T,\beta T/T_0,\Dd}
  \end{align} where the equality follows from (\ref{eqnSphereHoroEquivalence2}).  We want these pieces to exactly match, and, hence we must have that \begin{align*}
 \beta = \frac {T_0}{T}. \end{align*}

Consider \begin{align*}\widehat{f}_{T_0, \Dd, \beta}(\boldsymbol{r}, n_-(\boldsymbol{r})& \Phi^{t-\frac 1 d \log (T/T_0)}) \\ \nonumber & = \int_{N^+} \In_{\Aa \times\left\{\Ss_{T_0c^{-d/(d-1)}} n_+(\beta \boldsymbol{z'}):  \boldsymbol{z} \in \Dd \right\}}\left(\boldsymbol{r}, \tensor*[^t]{\left(n_-(\boldsymbol{r}) \Phi^{t-\frac 1 d \log (T/T_0)}\right)}{}^{-1} n_+(\boldsymbol{\widetilde{x}})\right) ~\wrt {\widetilde{n}(\boldsymbol{\widetilde{x}})} 
\\ \nonumber & = \In_{\Aa \times\Ss_{T_0c^{-d/(d-1)}}}\left(\boldsymbol{r}, \tensor*[^t]{\left(n_-(\boldsymbol{r}) \Phi^{t-\frac 1 d \log (T/T_0)}\right)}{}^{-1}\right)\int_{N^+} \In_{N_{ T_0,\beta,\Dd}} \left( n_+(\boldsymbol{\widetilde{x}})\right) ~\wrt {\widetilde{n}(\boldsymbol{\widetilde{x}})}. 
  \end{align*}
  
Note that, we have \begin{align*}\int_{N^+} \In_{N_{ T_0,\beta,\Dd}} \left( n_+(\boldsymbol{\widetilde{x}})\right) ~\wrt {\widetilde{n}(\boldsymbol{\widetilde{x}})}  = &\int_{N^+} \In_{N_{ T_0,\beta T/T_0,\Dd}} \left( n_+\left(\frac {T}{T_0}\boldsymbol{\widetilde{x}}\right)\right) ~\wrt {\widetilde{n}(\boldsymbol{\widetilde{x}})} \\ \nonumber = &\frac{T_0^{d-1}}{T^{d-1}}\int_{N^+} \In_{N_{ T_0,1,\Dd}} \left( n_+\left(\boldsymbol{\widetilde{y}}\right)\right) ~\wrt {\widetilde{n}(\boldsymbol{\widetilde{y}})}\end{align*} by changing variables $\widetilde{\boldsymbol{y}} = \frac {T}{T_0} \widetilde{\boldsymbol{x}}$, which implies that $\wrt {\boldsymbol{\widetilde{y}}} = \frac{T^{d-1}}{T_0^{d-1}}~\wrt {\boldsymbol{\widetilde{x}}}$ and, thus, $\wrt{\widetilde{n}(\boldsymbol{\widetilde{y}})} = \frac{T^{d-1}}{T_0^{d-1}}~\wrt {\widetilde{n}(\boldsymbol{\widetilde{x}})}$.

Consequently, we have \begin{align}\label{eqnCompareExpandingHaarVolNAKCase}\widehat{f}_{T_0, \Dd, \beta}(\boldsymbol{r}, n_-(\boldsymbol{r}) \Phi^{t-\frac 1 d \log (T/T_0)})  = \frac{T_0^{d-1}}{T^{d-1}}\widehat{f}_{T_0, \Dd}(\boldsymbol{r}, n_-(\boldsymbol{r}) \Phi^{t-\frac 1 d \log (T/T_0)}).
  \end{align}
  
Using the same proof as in Lemma~\ref{lemmUnstableVolumes} to obtain (\ref{eqnRatioHaarMeasuresNANCase}), we obtain \begin{align}\label{eqnRatioHaarMeasuresNAKCase}
 \frac {1} {Q^d} \sum_{\boldsymbol{r} \in \F_Q} \widehat{f}_{T, \Dd}(\boldsymbol{r}, n_-(\boldsymbol{r}) \Phi^t) = \frac 1 {\widetilde{Q}^{d}}\sum_{\boldsymbol{r} \in \F_{\widetilde{Q}}} \widehat{f}_{T_0, \Dd, \beta}\left(\boldsymbol{r}, n_-(\boldsymbol{r}) \Phi^{t- \frac 1 d \log(T/T_0)}\right). \end{align}
 
Using (\ref{eqnCompareExpandingHaarVolNAKCase}), we obtain the following:  \[\frac {T^{d-1}} {Q^d} \sum_{\boldsymbol{r} \in \F_Q} \widehat{f}_{T, \Dd}(\boldsymbol{r}, n_-(\boldsymbol{r}) \Phi^t)  = \frac {T_0^{d-1}} {\widetilde{Q}^{d}}\sum_{\boldsymbol{r} \in \F_{\widetilde{Q}}}\widehat{f}_{T_0, \Dd}(\boldsymbol{r}, n_-(\boldsymbol{r}) \Phi^{t-\frac 1 d \log (T/T_0)}).\] 

 This gives the desired result.

 \end{proof}
  
As in the proof of Theorem~\ref{thmNANResultRankOneFlow}, we use Lemma~\ref{lemmUnstableVolumesNAKCase} to apply Theorem~\ref{thmMarklof}:

 \begin{align}\label{eqnMainDerivUsingFrobNAKCase}\lim_{t \rightarrow \infty} &\frac {T^{d-1}} {|\F_Q|} \sum_{\boldsymbol{r} \in \F_Q} \widehat{f}_{T, \Dd}(\boldsymbol{r}, n_-(\boldsymbol{r}) \Phi^t)   \\ \nonumber&=\lim_{t \rightarrow \infty} \frac {T_0^{d-1}} {|\F_{\widetilde{Q}}|} \sum_{\boldsymbol{r} \in \F_{\widetilde{Q}}}  \widehat{f}_{T_0, \Dd}(\boldsymbol{r}, n_-(\boldsymbol{r}) \Phi^{t-\frac 1 d \log (T/T_0)}) 
 \\ \nonumber&=  \lim_{t \rightarrow \infty} \frac {T_0^{d-1}} {|\F_{\widetilde{Q}}|} \sum_{\boldsymbol{r} \in \F_{\widetilde{Q}}}\In_{\Aa \times\Ss_{T_0c^{-d/(d-1)}}}\left(\boldsymbol{r}, \tensor*[^t]{\left(n_-(\boldsymbol{r}) \Phi^{t-\frac 1 d \log (T/T_0)}\right)}{}^{-1}\right)\int_{N^+} \In_{N_{ T_0,1,\Dd}} \left( n_+\left(\boldsymbol{\widetilde{x}}\right)\right) ~\wrt {\widetilde{n}(\boldsymbol{\widetilde{x}})}
  \\ \nonumber&=  \lim_{t \rightarrow \infty} \frac {T_0^{d-1}} {|\F_{\widetilde{Q}}|} \sum_{\boldsymbol{r} \in \F_{\widetilde{Q}}} \widetilde{\In}_{\Aa \times\Ss_{T_0c^{-d/(d-1)}}}\left(\boldsymbol{r}, \left(n_-(\boldsymbol{r}) \Phi^{t-\frac 1 d \log (T/T_0)}\right)\right)\int_{N^+} \In_{N_{ T_0,1,\Dd}} \left( n_+\left(\boldsymbol{\widetilde{x}}\right)\right) ~\wrt {\widetilde{n}(\boldsymbol{\widetilde{x}})}
  \\ \nonumber&
  = d(d-1)T_0^{d-1} \int_0^\infty \int_{\TT^{d-1} \times \Ga_H \backslash H}\int_{N^+} \In_{\Aa \times \Ss_{T_0c^{-d/(d-1)}}}(\boldsymbol{x},  M \Phi^{-s})\\ \nonumber &\quad \quad \quad \quad \quad \quad \quad \quad \quad \quad \quad \quad \quad \quad \quad \quad \quad \quad  \In_{N_{ T_0,1,\Dd}} \left( n_+\left(\boldsymbol{\widetilde{x}}\right)\right) ~\wrt {\widetilde{n}(\boldsymbol{\widetilde{x}})}~\wrt{\boldsymbol{x}}~\wrt \mu_H(M) e^{-d(d-1)s}~\wrt s 
  \\ \nonumber&
  = d(d-1)T_0^{d-1} \int_0^\infty \int_{\TT^{d-1} \times \Ga_H \backslash H}\int_{N^+} \In_{\Aa \times \left\{\Ss_{T_0 c^{-d/(d-1)}} n_+\left(\boldsymbol{z'}\right):  \boldsymbol{z} \in \Dd \right\}}(\boldsymbol{x},  M \Phi^{-s}n_+\left(\boldsymbol{\widetilde{x}}\right))\\ \nonumber &\quad \quad \quad \quad \quad\quad \quad \quad\quad \quad \quad\quad \quad \quad\quad \quad \quad \quad \quad \quad \quad \quad \quad \quad \quad \quad \quad ~\wrt {\widetilde{n}(\boldsymbol{\widetilde{x}})}~\wrt{\boldsymbol{x}}~\wrt \mu_H(M) e^{-d(d-1)s}~\wrt s 
  \\ \nonumber&
  = d(d-1)T_0^{d-1}\left(\int_{\TT^{d-1}} \In_{\Aa}(\boldsymbol{x})~\wrt{\boldsymbol{x}}\right)\\ \nonumber &\quad \quad \quad  \left(\int_0^\infty \int_{\Ga_H \backslash H}\int_{N^+} \In_{\left\{\Ss_{T_0 c^{-d/(d-1)}} n_+\left(\boldsymbol{z'}\right):  \boldsymbol{z} \in \Dd \right\}}(M \Phi^{-s}n_+\left(\boldsymbol{\widetilde{x}}\right))~\wrt {\widetilde{n}(\boldsymbol{\widetilde{x}})}~\wrt \mu_H(M) e^{-d(d-1)s}~\wrt s \right)
  \\ \nonumber&
  = d\zeta(d) T_0^{d-1}\mu(\Ss_{T_0} \Ee)\left(\int_{\TT^{d-1}} \In_{\Aa}(\boldsymbol{x})~\wrt{\boldsymbol{x}}\right).
        \end{align}  

The last equality follows by Theorem~\ref{thmHaarMeaSE}.  Now, as in the proof of Theorem~\ref{thmNANResultRankOneFlow}, we have that the proportion of the translates of $\Aa$ meeting $\Ss_T \Ee$ for all $T\geq T_0$ is uniformly (as long as $\lim_{t \rightarrow \infty} \frac {\log(T^{1/d})}  {t} < 1$) the following:

\begin{align}\label{eqnNAKMainResultConclusion}
{T^{d-1}}\int_{\TT^{d-1}} \In_{\Aa \times\Ss_T  \Ee}(\boldsymbol{x}, {n_-(\boldsymbol{x}) \Phi^t})~\wrt {\boldsymbol{x}} \xrightarrow[]{t \rightarrow \infty}T_0^{d-1}\mu(\Ss_{T_0} \Ee)\left(\int_{\TT^{d-1}} \In_{\Aa}(\boldsymbol{x})~\wrt{\boldsymbol{x}}\right).
 \end{align}

 This proves Theorem~\ref{thmNAKResultRankOneFlow}.

 \section{Cuspidal neighborhoods}\label{secCuspNeigh}  In this section, we define the family of cuspidal neighborhoods that we consider in Theorem~\ref{thmShrinkCuspNeigh}.  We use the Grenier fundamental domain~\cite{Gre93, Gre88} (first considered by Hermite and also by Korkine-Zolotareff~\cite{KZ1873}) because this domain is well-suited for the geometric constructions that we need to make as it has a ``box shape'' in the cusp (see {\em exact box shape} in \cite[Section~1.4.3]{Ter16}).  For similar reasons, the Grenier fundamental domain is also used in other settings such as in the construction of an Eisenstein series in harmonic analysis (see, for example,~\cite[Chapter~6]{JL05} or~\cite[Section~1.5]{Ter16}) or working with the Epstein Zeta function~\cite{SaSt06}.  For an introduction to this domain, a discussion of its history, and a comparison with the best known fundamental domain, the Minkowski fundamental domain, see~\cite[Section~1.4]{Ter16}.  The facts about the Grenier fundamental domain that we use have been conveniently collated in~\cite[Chapter~1]{JL05}.

Let $\widetilde{G}:= \widetilde{G}_d := \GL(d,\RR)$ and $\widetilde{\Ga}:=\widetilde{\Ga}_d:= \GL(d, \ZZ)$.  Here, $\widetilde{\Ga}$ is a discrete subgroup of $\widetilde{G}$.  Let $G$ and $\Ga$ be as before.  Let us denote the determinant one hypersurface of the space of positive symmetric matrices by \[\SP:= \SP_d:=\{g \tensor*[^t]{g}{}  : g \in \GL(d,\RR), |\det(g)| =1 \}.\]  The actions of $G$ and $\widetilde{\Ga}$ on $\SP$ are by \begin{align} \label{defnofSquareBraket}
 [g]\widetilde{Z} :=  g\widetilde{Z} \tensor*[^t]{g}{} \end{align} for $\widetilde{Z} \in \SP$ and $g \in G$ or $g \in \widetilde{\Ga}$.\footnote{This gives us a left action as in~\cite[Chapter~1]{JL05}, not the right action of Grenier.}  Let us consider a recursively-defined family of coordinates, called {\em partial Iwasawa coordinates} (see~\cite[Page 14]{Ter16} or~\cite[Pages~14-16]{JL05} where ``coordinates'' are replaced by ``decomposition''). The first member of this family is the {\em first-order Iwasawa coordinates} \begin{align}\label{eqn1OrdIwaCoord}
  \widetilde{Z}=\begin{bmatrix}  I_{d-1}&\tensor*[^t]{\boldsymbol{x}}{}^{(d-1)} \\ \boldsymbol{0} & 1\end{bmatrix}\begin{pmatrix}a_d^{1/(d-1)} \widetilde{Z}^{(d-1)}&\tensor*[^t]{\boldsymbol{0}}{}  \\ \boldsymbol{0} & a_d^{-1}   \end{pmatrix} \end{align}where $\widetilde{Z}^{(d-1)} \in \SP_{d-1}$ and $\tensor*[^t]{\boldsymbol{x}}{}^{(d-1)}$ is a column vector with $d-1$ entries.  The element $\widetilde{Z}^{(d-1)}$ can also be expressed in first-order Iwasawa coordinates \[\widetilde{Z}^{(d-1)}=\begin{bmatrix}  I_{d-2}&\tensor*[^t]{\boldsymbol{x}}{}^{(d-2)} \\ \boldsymbol{0} & 1\end{bmatrix}\begin{pmatrix}a_{d-1}^{1/(d-2)} \widetilde{Z}^{(d-2)}&\tensor*[^t]{\boldsymbol{0}}{}  \\ \boldsymbol{0} & a_{d-1}^{-1}   \end{pmatrix} \] where $ \widetilde{Z}^{(d-2)}\in \SP_{d-2}$ and $\tensor*[^t]{\boldsymbol{x}}{}^{(d-2)}$ is a column vector with $d-2$ entries.  Directly computing, we obtain the {\em second-order Iwasawa coordinates} \[\widetilde{Z}=\begin{bmatrix}  I_{d-2}&\tensor*[^t]{\boldsymbol{x}}{}^{(d-2)} & \tensor*[^t]{\boldsymbol{x}}{}^{(d-1)}\\ \boldsymbol{0} & 1 & \\ \boldsymbol{0} & 0 &1\end{bmatrix}\begin{pmatrix}a_d^{1/(d-1)} a_{d-1}^{1/(d-2)}\widetilde{Z}^{(d-2)}& &  \\ &a_d^{1/(d-1)}a_{d-1}^{-1}   &\\& & a_d^{-1}   \end{pmatrix}.\]  Recursively repeating this decomposition into higher order Iwasawa coordinates, we obtain, at the final step, the {\em full Iwasawa coordinates} for $\widetilde{Z}$~\cite[Page~16, Equation (5)]{JL05}:  \begin{align}\label{eqnFullIwasawa} \begin{bmatrix}  1&x_{ij} &\\ &\ddots& \\& & 1\end{bmatrix}y^{-1} \begin{pmatrix}(y_{d-1}y_{d-2} \cdots y_{1})^2&&&& \\ &\ddots&&&&\\ && (y_{d-1} y_{d-2})^2&&\\ &&&y_{d-1}^2&\\ &&&& 1   \end{pmatrix}
  \end{align} where $y:= a_d$, $y_\ell^2:=a_{\ell +1}^{(\ell +1)/\ell} a_\ell^{-1}$ for an integer $1\leq \ell \leq d-1$, and $x_{ij}$ is the entry in the $i$-th row and $j$-th column of \[\begin{bmatrix}  1&x_{ij} &\\ &\ddots& \\& & 1\end{bmatrix} := \begin{bmatrix}  1&\tensor*[^t]{\boldsymbol{x}}{}^{(1)}  &  \cdots &\tensor*[^t]{\boldsymbol{x}}{}^{(d-2)} & \tensor*[^t]{\boldsymbol{x}}{}^{(d-1)}\\ 0 & 1 & \\ 0 &0 & \ddots & & \\ \vdots & \vdots  & &1\\ 0 & 0 & \cdots & 0&1\end{bmatrix}.\] Note that the partial Iwasawa coordinates allow us to embed lower rank spaces $\SP_m$ into $\SP_d$, which we regard as an identification~\cite{Gre93}.

The subgroup $\widetilde{\Ga}$ acts properly continuously on $\SP$ and a fundamental domain $\Ff_d$ (the Grenier fundamental domain) is recursively defined as the set of all $\widetilde{Z} \in \SP$ satisfying for all primitive vectors $(\boldsymbol{p}, q) \in \ZZ^d$ the following three properties~\cite[Page~17-18]{JL05}:
\begin{description}
\item[Gren 1] $a_d(\boldsymbol{p}, q)\widetilde{Z} \begin{pmatrix}\tensor*[^t]{ \boldsymbol{p}}{} \\ q  \end{pmatrix} \geq 1$.
\item[Gren 2] $\widetilde{Z}^{(d-1)} \in \Ff_{d-1}$.
\item[Gren 3] $\begin{cases}  0\leq \boldsymbol{x}^{(d-1)}_j\leq1/2  &\textrm{ if } j=1  \\ |\boldsymbol{x}^{(d-1)}_j|\leq1/2 &\textrm{ if }  j=2, \cdots, d-1.\end{cases}$
\end{description}  Note that the notation $\boldsymbol{x}_j$ denotes the $j$-th entry of a vector $\boldsymbol{x}$.  The fundamental domain $\Ff_d$ can be identified with the space $\widetilde{\Ga} \backslash \SP$.  For $d$ odd, we note that $\widetilde{\Ga}   \backslash \SP= \Ga  \backslash \SP$ and thus note that $\Ff'_d := \Ff_d$ is a fundamental domain for $\Ga \backslash \SP$.  For $d$ even, the domain of $\widetilde{\Ga}  \backslash \SP$ is strictly smaller.  Since we have that $\widetilde{\Ga} = \{ \Ga, \Ga\gamma\}$ where $\gamma:= \diag(1, \cdots, 1, -1)$, it follows that $\Ff'_d := \Ff_d \cup  [\gamma]\Ff_d$ is a fundamental domain for $\Ga \backslash \SP$.\footnote{Our choice of $\gamma$ is for convenience.  Let $\Delta_d \subset \widetilde{\Ga}$ be the subgroup of diagonal matrices with $\pm1$ as diagonal elements.  Then replacing $\gamma$ with another element of determinant $-1$ in $\{\pm I_d\} \backslash \Delta_d$ will yield another fundamental domain for $\Ga \backslash \SP$ when $d$ is even.  Our results are not affected by this choice except for the explicit construction of Siegel sets in (\ref{eqnDefnSiegelSets}) and of Grenier boxes (defined in Section~\ref{secConstructShrinkTargets}).  The explicit constructions for any of these choices can be easily derived if desired.}

Because our targets are shrinking into the cusp, we do not need to work with the fundamental domain $\Ff'_d$ but may instead work with geometrically simpler sets (Lemma~\ref{lemm:SiegelSetApprox}).  These geometrically simpler sets are the {\em Siegel sets for $\Ff'_d$}, denoted by $\widetilde{\Sie}^{(d)}_T$ for a real number $T > 0$ and defined as follows (notation as in (\ref{eqnFullIwasawa})):

\begin{align}\label{eqnDefnSiegelSets}\widetilde{\Sie}^{(d)}_T= \begin{cases} \left\{\widetilde{Z} \in \SP: \begin{cases}   0\leq x_{1j}\leq 1/2  &\textrm {for } j = 2,\cdots, d-1 \\ |x_{1j}|\leq 1/2 &\textrm {for } j =  d \\ |x_{ij}|\leq 1/2 &\textrm {for } 2\leq i < j \leq  d \\  y^2_\ell \geq T &\textrm {for } \ell = 1,\cdots, d-1\end{cases}\right\} & \textrm{ if } d \textrm{ is even,} \\
\left\{\widetilde{Z} \in \SP: \begin{cases}   0\leq x_{1j}\leq 1/2  &\textrm {for } j = 2,\cdots, d  \\ |x_{ij}|\leq 1/2 &\textrm {for } 2\leq i < j \leq  d \\  y^2_\ell \geq T &\textrm {for } \ell = 1,\cdots, d-1\end{cases}\right\} & \textrm{ if } d \textrm{ is odd.}\end{cases}
  \end{align}  The importance of Siegel sets is that they bound the Grenier fundamental domain:
  
  \begin{lemm}\label{lemm:SiegelSetApprox}$\widetilde{\Sie}^{(d)}_1 \subset \Ff'_d \subset \widetilde{\Sie}^{(d)}_{3/4}$.
 
\end{lemm}
\begin{proof}
 The proof is analogous to the proof of~\cite[Chapter~1, Theorem~5.1]{JL05} with our Siegel set $\widetilde{\Sie}^{(d)}_T$ and domain $\Ff'_d$ replacing, respectively, the Siegel set $\Sie^{(d)}_T$ and domain $F_d^{\pm}$ defined in~\cite[Page~20-21]{JL05}. 
\end{proof}

To identify $\SP$ with the group $NA$ where $A \subset G$ is the subgroup of diagonal matrices with all positive entries and $N \subset G$ is the unipotent upper triangular subgroup, we apply the diffeomorphism $\varphi: \SP \rightarrow NA$ in which $[n]a^2 \mapsto na$.  (Our diffeomorphism $\varphi$ is the inverse of the diffeomorphism $\varphi^+$ from~\cite[Chapter~2, Proposition~2.2]{JL05}.)  This is well-defined because our Lie groups are defined over the real numbers.

Now recall the definitions of $K$ and $K'$ in Section~\ref{secConstructShrinkTargets}.  Consider the following \begin{align}\label{eqnDecompsitionOfEleOfH}
  \tensor*[]{\begin{pmatrix} B &  \tensor*[^t]{\boldsymbol{b}}{} \\ \boldsymbol{0} & 1\end{pmatrix}}{}  = \begin{pmatrix}  I_{d-1} &  \tensor*[^t]{\boldsymbol{b}}{} \\ \boldsymbol{0} & 1\end{pmatrix}\begin{pmatrix}  \tensor*[]{Z}{}  &  \tensor*[^t]{\boldsymbol{0}}{} \\ \boldsymbol{0} & 1\end{pmatrix}\begin{pmatrix} k  &  \tensor*[^t]{\boldsymbol{0}}{} \\ \boldsymbol{0} & 1\end{pmatrix} \end{align} where $B \in \SL(d-1, \RR)$ and can be uniquely written in Iwasawa coordinates as $B = Zk$ where $k \in K'$ and $\varphi^{-1}(Z) \in \SP_{d-1}$.  (Here we are using the identification of $\SP_{d-1}$ with the subspace of $\SP$ given by the partial Iwasawa coordinates.)  Let us apply the diagonal flow $\Phi^{-\frac{\log y}{2(d-1)}}$ to (\ref{eqnDecompsitionOfEleOfH}) on the right:  \begin{lemm}\label{lemmeqnFundDomCoord}  We have that
 \begin{align}\label{eqnFundDomCoord} \tensor*[]{\begin{pmatrix} B &  \tensor*[^t]{\boldsymbol{b}}{} \\ \boldsymbol{0} & 1\end{pmatrix}}{}&\begin{pmatrix}y^{1/2(d-1)} I_{d-1}&\tensor*[^t]{\boldsymbol{0}}{}  \\ \boldsymbol{0} & y^{-1/2}   \end{pmatrix}   \\ \nonumber &= \begin{pmatrix}  1&x_{ij} &\\ &\ddots& \\& & 1\end{pmatrix}y^{-1/2} \begin{pmatrix}y_{d-1}y_{d-2} \cdots y_{1}&&&& \\ &\ddots&&&&\\ && y_{d-1} y_{d-2}&&\\ &&&y_{d-1}&\\ &&&& 1   \end{pmatrix}\begin{pmatrix} k  &  \tensor*[^t]{\boldsymbol{0}}{} \\ \boldsymbol{0} & 1\end{pmatrix}.\end{align}  
\end{lemm} 
\begin{proof}  We will use the diffeomorphism \[(\varphi, \textrm{id}):  \SP \times K \rightarrow G \quad \quad ([na], k) \mapsto nak.\]  Equation (\ref{eqnDecompsitionOfEleOfH}) yields \begin{align*} \tensor*[]{\begin{pmatrix} B &  \tensor*[^t]{\boldsymbol{b}}{} \\ \boldsymbol{0} & 1\end{pmatrix}}{}\begin{pmatrix}y^{1/2(d-1)} I_{d-1}&\tensor*[^t]{\boldsymbol{0}}{}  \\ \boldsymbol{0} & y^{-1/2}   \end{pmatrix}  = \begin{pmatrix}  I_{d-1} &  \tensor*[^t]{\boldsymbol{b}}{} \\ \boldsymbol{0} & 1\end{pmatrix}\begin{pmatrix}  \tensor*[]{Z}{}  &  \tensor*[^t]{\boldsymbol{0}}{} \\ \boldsymbol{0} & 1\end{pmatrix}\begin{pmatrix}y^{1/2(d-1)} I_{d-1}&\tensor*[^t]{\boldsymbol{0}}{}  \\ \boldsymbol{0} & y^{-1/2}   \end{pmatrix} \begin{pmatrix} k  &  \tensor*[^t]{\boldsymbol{0}}{} \\ \boldsymbol{0} & 1\end{pmatrix},
  \end{align*} to the right-hand side of which we apply the diffeomorphism $(\varphi, \textrm{id})^{-1}$ to obtain \begin{align*}\left(\begin{bmatrix}  I_{d-1} &  \tensor*[^t]{\boldsymbol{b}}{} \\ \boldsymbol{0} & 1\end{bmatrix}\begin{bmatrix}  \tensor*[]{Z}{}  &  \tensor*[^t]{\boldsymbol{0}}{} \\ \boldsymbol{0} & 1\end{bmatrix}\begin{bmatrix}y^{1/2(d-1)} I_{d-1}&\tensor*[^t]{\boldsymbol{0}}{}  \\ \boldsymbol{0} & y^{-1/2}   \end{bmatrix}, \begin{pmatrix} k  &  \tensor*[^t]{\boldsymbol{0}}{} \\ \boldsymbol{0} & 1\end{pmatrix}\right) \\ = \left(\begin{bmatrix}  I_{d-1} &  \tensor*[^t]{\boldsymbol{b}}{} \\ \boldsymbol{0} & 1\end{bmatrix}\begin{pmatrix}  y^{1/(d-1)}\varphi^{-1}(Z)  &  \tensor*[^t]{\boldsymbol{0}}{} \\ \boldsymbol{0} & y^{-1}\end{pmatrix}, \begin{pmatrix} k  &  \tensor*[^t]{\boldsymbol{0}}{} \\ \boldsymbol{0} & 1\end{pmatrix}\right).
  \end{align*}  Applying first (\ref{eqnFullIwasawa}) and then the diffeomorphism $(\varphi, \textrm{id})$ to the right-hand side of this equation yields the desired result.
 
\end{proof} 
\noindent Note that any point $z$ in $H \{\Phi^{-s} :  s \in \RR\}$ can be uniquely expressed in the form (\ref{eqnFundDomCoord}), which we refer to as the {\em Grenier coordinates}, and the value of $y$ in this expression is called the \textit{height} of this point, a well-defined notion because of the uniqueness of the expression.  Moreover, in Grenier coordinates, the height can be expressed as \begin{align}\label{eqnHeightinGrenierCoords}
 y^d = \prod^{d-1}_{k=1} y_{d-k}^{2(d-k)}. \end{align} Let us denote the height of $z$ by $\h(z)$.  
 
 \begin{rema}\label{rmkTheFunDoms} The Grenier coordinates yield a fundamental domain $\varphi(\Ff'_d) K'$ for the left-action of $\Gamma$ on $\Gamma H \{\Phi^{-s} :  s \in \RR\}$, and, therefore, allows us to identify $\varphi(\Ff'_d) K'$ with $\Ss_0$ and $\varphi(\widetilde{\Sie}^{(d)}_1) K'$ with $\Ss_1$.\footnote{One can also give a fundamental domain for the left-action of $\Gamma$ on $G$ using the Grenier coordinates.  As the kernel of the representation $\gamma \mapsto [\gamma]$ of $\widetilde{\Ga}$ in $\Aut(\SP)$ is $\pm I_d$~\cite[Page~6]{JL05} and under $\varphi$ the action given by $[\cdot]$ is mapped to the left-action on $G$, we have a fundamental domain for the left-action by $\Gamma$ on $G$ to be $\varphi(\Ff'_d) K$ when $d$ is odd.  On the other hand, we have $-I_d \in \Gamma$ for $d$ even and thus $\varphi(\Ff'_d)K_+$ is a fundamental domain where $K_+$ is a set of coset representatives of $K\backslash{\pm I_d}$.  While many choices of $K_+$ exist, we require that every element of $K_+$ have nonnegative $(d,d)$-th entry.  Note that, as the condition that the $(d,d)$-th entry is equal to zero yields a measure zero set, we have that  $\varphi(\Ff'_d)K_+$ is a Borel set.  Also note that, in particular, $K' \subset K_+$, and, thus, our Grenier boxes (defined in Section~\ref{secConstructShrinkTargets}) lie within this fundamental domain for the left $\Gamma$-action on $G$.}
\end{rema}  
 
\begin{rema}\label{rmkOmittingVarphi}
For conciseness, let us identify $\varphi(\Ff'_d)$ with $\Ff'_d$,  $\varphi(\widetilde{\Sie}^{(d)}_T)$ with $\widetilde{\Sie}^{(d)}_T$, and, henceforth, omit $\varphi$ from the notation. 
\end{rema}

Recall from Section~\ref{secConstructShrinkTargets} the definition of a Grenier box $\Cc:=\Cc_{T_-, T_+}:=\Cc(\boldsymbol{\alpha}, \boldsymbol{\gamma}, \widetilde{K},  \beta^-_{ij}, \beta^+_{ij})$ and its lower height $T_-$ and upper height $T_+$.  Since (\ref{eqnHeightinGrenierCoords}) holds, we note that the height of any point in $\Cc$ lies between $T_-$ and $T_+$.  Note that $\Cc \subset \Ss_1$.\footnote{Here and below, we identify $\Cc$ with $\Ga \backslash \Ga \Cc$ using context to distinguish between them if needed.  This allows us to more quickly develop our main results.  Since the lower heights of interest to us will always be large enough so that $\Cc$ lies in a fundamental domain for the left $\Gamma$-action, this identification is justified.}  Also note that (\ref{eqnFundDomCoord}) relates two sets of coordinates for the points of $\Ss_1$, namely the left-hand side, which we refer to as {\em first-order section coordinates} and the right-hand side, which we refer to as {\em full section coordinates}.

By considering full section coordinates, we see that the only way to approach the cusp is to let at least one of the $y_k \rightarrow \infty$; consequently, should at least one of the components of $\boldsymbol{\gamma}$ be $\infty$, then our neighborhood meets the cusp and conversely.  This equivalence can be reformulated in terms of the upper height:

\begin{lemm}\label{lemmBndHeightIFFRelCpct} The set $\Cc_{T_-, T_+}$ is relatively compact if and only if $T_+ < \infty$.
 
\end{lemm}

\begin{proof}
We first prove the direct implication.  Every point in the set $\Cc_{T_-, T_+}$ can be written in coordinates given by (\ref{eqnFundDomCoord}).  Using the definition of the Grenier domain and the fact that $\widetilde{K}$ is relatively compact implies that each $y_k$ coordinate must be bounded.  Consequently, $T_+ < \infty$.

We now prove the inverse implication.  Since the height of any point in the neighborhood is less than or equal to $T_+$ and since all of the $y_k$'s are bounded from below for any point in the Grenier domain, the condition (\ref{eqnHeightinGrenierCoords}) implies that all of the $y_k$'s for this neighborhood are bounded from above, which gives the desired result.
 
\end{proof}

\subsection{Applying the flow to $\Cc$}  Let $T\geq 0$.  Using (\ref{eqnFundDomCoord}) (or, equivalently, the standard coordinates~\cite[Page 15]{JL05} under $\varphi$) and that \begin{align}\label{eqnTFlow}
 \begin{pmatrix}
T^{1/(d-1)} I_{d-1}&\tensor*[^t]{\boldsymbol{0}}{}  \\ \boldsymbol{0} & T^{-1}   \end{pmatrix}  \end{align}commutes with $K'$, we have the following relationship:  \begin{align}\label{eqnRelCTAndC}
\Cc^T:=\Cc\left((\alpha_1, \cdots, \alpha_{d-2}, T^{d/(d-1)}\alpha_{d-1}), (\gamma_1, \cdots, \gamma_{d-2}, T^{d/(d-1)}\gamma_{d-1}), \widetilde{K},  \beta^-_{ij}, \beta^+_{ij}\right) & \\ \nonumber =  \Cc\left((\alpha_1, \cdots, \alpha_{d-1}), (\gamma_1, \cdots, \gamma_{d-1}), \widetilde{K},  \beta^-_{ij}, \beta^+_{ij}\right) & \begin{pmatrix}
T^{1/(d-1)} I_{d-1}&\tensor*[^t]{\boldsymbol{0}}{}  \\ \boldsymbol{0} & T^{-1}   \end{pmatrix}.
  \end{align}  In particular, note that applying (on the right) the flow given by (\ref{eqnTFlow}) to $\Cc$ changes only the $y_{d-1}$-coordinate and that  letting $T \rightarrow \infty$ flows the neighborhood $\Cc$ towards the cusp along only the $y_{d-1}$-coordinate, while leaving invariant all other coordinates.

\subsection{Volumes}\label{subsecVolumes} We must compute the volumes of the thickenings of $\Cc$, $\Cc^T$.  To do this, we parametrize $G$ following~\cite{Mar10} to obtain a convenient normalization for the Haar measure on $\Ga \backslash G$.

For $d\geq 3$ and odd, let \begin{align} \label{eqnDefTildeIdMin1}
 \widetilde{I}_{d-1} := \begin{pmatrix}  I_{d-2} &  \tensor*[^t]{\boldsymbol{0}}{} \\ \boldsymbol{0} & -1  \end{pmatrix}. \end{align}  Now given $\boldsymbol {y} :=(y_1, \cdots, y_d) \in \RR^d$ where $y_d \neq 0$, set $\boldsymbol{y}':=(y_1, \cdots, y_{d-1})$ and form the matrix \[M_{\boldsymbol{y}} :=  \begin{cases}  \begin{pmatrix}
y_d^{-1/(d-1)} I_{d-1}&\tensor*[^t]{\boldsymbol{0}}{}  \\ \boldsymbol{y'} & y_d   \end{pmatrix} &\text{ if } y_d >0 \\ \begin{pmatrix}  (-y_d)^{-1/(d-1)} I_{d-1} & \tensor*[^t]{\boldsymbol{0}}{} \\ -\boldsymbol{y'} & -y_d  \end{pmatrix} &\text{ if } y_d <0\end{cases}\] and note that the mapping \begin{align}\label{eqnParametrizationForG}
H \times \RR^d \backslash (\RR^{d-1} \times \{0\} ) \rightarrow G \quad \begin{cases}  (M, \boldsymbol{y}) \mapsto MM_{\boldsymbol{y}}  &\text{ if } y_d >0 \\ (M, \boldsymbol{y}) \mapsto -I_d MM_{\boldsymbol{y}}  &\text{ if } y_d <0 \text{ and } d \text{ is even} \\  (M, \boldsymbol{y}) \mapsto \begin{pmatrix}  \widetilde{I}_{d-1} &  \tensor*[^t]{\boldsymbol{0}}{} \\ \boldsymbol{0} & -1  \end{pmatrix}MM_{\boldsymbol{y}}  &\text{ if } y_d <0 \text{ and } d \text{ is odd}\end{cases}\end{align} provides a parametrization of $G$ (except for the null set $\{A \in G: (\boldsymbol{0}, 1) \  A \ \tensor*[^t]{(\boldsymbol{0}, 1)}{}= 0\}$).  Also note that the use of \[\begin{pmatrix}  \widetilde{I}_{d-1} &  \tensor*[^t]{\boldsymbol{0}}{} \\ \boldsymbol{0} & -1  \end{pmatrix}\] is a natural choice, but not the only one.  Other choices would lead to different parameterizations, but would not change the results in this paper.

 Using Iwasawa coordinates, we have that \[M_{\boldsymbol{y}} =  \begin{cases}  \begin{pmatrix}
y_d^{-1/(d-1)} I_{d-1}&\tensor*[^t]{\boldsymbol{0}}{}  \\ \boldsymbol{0} & y_d   \end{pmatrix} \begin{pmatrix}
 I_{d-1}&\tensor*[^t]{\boldsymbol{0}}{}  \\ y_d^{-1}\boldsymbol{y'} & 1   \end{pmatrix} & \text{ for } y_d >0 \\ \begin{pmatrix}
(-y_d)^{-1/(d-1)} I_{d-1}&\tensor*[^t]{\boldsymbol{0}}{}  \\ \boldsymbol{0} & -y_d   \end{pmatrix} \begin{pmatrix}
 I_{d-1}&\tensor*[^t]{\boldsymbol{0}}{}  \\ y_d^{-1}\boldsymbol{y'} & 1   \end{pmatrix} & \text{ for } y_d <0\end{cases},\] which allows us to derive the following explicit formula for the right Haar measure on \[\left\{M_{\boldsymbol{y}} \in G: {\boldsymbol{y}}\in \RR^d \backslash (\RR^{d-1} \times \{0\} )\right\},\] namely \begin{align}\label{eqnFormulaDy} \wrt{\boldsymbol{y}} = |y_d|^{d} \frac {\wrt y_d}{|y_d|}\frac{\wrt y_1 \cdots~\wrt y_{d-1}}{|y_d|^{d-1}} = \wrt y_1 \cdots~\wrt y_{d}.
\end{align}  Note that the leading factor of $|y_d|^{d}$ is the modular function.  Recall that $\wrt \mu_H =~\wrt \mu_0~\wrt{\boldsymbol{b}}$ is the left Haar measure on $H$, normalized so that $\mu_H\left(\Gamma_H \backslash H\right)=1$, which, in turn, implies that the Haar measure on $G$ in this parametrization is \begin{align}\label{eqnFormulaMuEqualsHy}
\wrt \mu_H~\wrt{\boldsymbol{y}}. \end{align}  It, thus, must be a constant multiple of the Haar measure $\mu$ on $G$, normalized so that $\mu(\Gamma \backslash G)=1$.  Now, a well-known result of Siegel shows that the volume of $\Gamma \backslash G$ is $\zeta(d) \cdots \zeta(2)$, which implies that\begin{align}\label{eqnHaarMeasAsMultOFLebMeas}
  \wrt \mu \frac{\wrt t} {t} =\left(\zeta(d) \cdots \zeta(2)\right)^{-1} \det(X)^{-d}\prod_{i,j=1}^d~\wrt X_{ij}, \end{align}  where $X = (X_{ij}) = t^{1/d}A \in \GL^+(d, \RR)$ and $A \in G$.  Thus, in view of our parametrization, we have that (cf~\cite[Equation (3.39)]{Mar10}) \begin{align}\label{eqnHaarInMMyCoord}
\wrt \mu = \zeta(d)^{-1}~\wrt \mu_H~\wrt{\boldsymbol{y}}. \end{align}  See~\cite[Section~4.1]{APT16} for a similar construction.

Finally, we have 

\begin{lemm}
We have \[\wrt\mu_{\tensor*[^t]{H}{^{-1}}} = \wrt \mu_H.\] 
\end{lemm}

\begin{proof}  Note first that the measure $\wrt \mu_0$ on $\SL(d-1, \RR)$ is invariant under taking the inverse, which is an involutive anti-automorphism, because $\SL(d-1, \RR)$ is unimodular (\cite[Proposition~1.3]{Bum13}).
 
As \[\begin{pmatrix} A & \tensor*[^t]{\boldsymbol{b}}{} \\ \boldsymbol{0} & 1  \end{pmatrix}^{-1} = \begin{pmatrix} A^{-1} & -A^{-1}\tensor*[^t]{\boldsymbol{b}}{} \\ \boldsymbol{0} & 1  \end{pmatrix}\] and $\det(A^{-1})=1$, we have the right-invariant Haar measure $\wrt\mu_{H{^{-1}}} = \wrt \mu_0 \wrt {\boldsymbol{b}}.$  As taking the transpose is another involutive anti-automorphism on $\SL(d-1, \RR)$, we have the left-invariant Haar measure $\wrt\mu_{\tensor*[^t]{H}{^{-1}}} = \wrt \mu_0 \wrt {\boldsymbol{b}}= \wrt \mu_H$.
 
\end{proof}

\subsubsection{Volume ratios of $\Cc$ and $\Cc^T$ and section coordinates}  Let \[\widetilde{\Phi} := \left\{\begin{pmatrix}y^{1/(d-1)} I_{d-1}&\tensor*[^t]{\boldsymbol{0}}{}  \\ \boldsymbol{0} & y^{-1}   \end{pmatrix}: y >0\right\}.\]  In this section, we express the left Haar measure of the group $H \widetilde{\Phi}$ in terms of first-order and full section coordinates and use the latter to compute the ratio of the volumes of $\Cc$ and $\Cc^T$.  First note that $H$ is normal in $H \widetilde{\Phi}$, which leads to the modular function $y^{-d}$ and thus a left Haar measure \begin{align}\label{eqn:leftHaarHTildePhi}
 y^{-d} \wrt \mu_H \frac{\wrt y}{y} \end{align} for $H \widetilde{\Phi}$ in first-order section coordinates.

We now give a left Haar measure in full section coordinates.  Let $\wrt n$ be the Haar measure on the subgroup $N$ of upper unipotent matrices in $G$ (under the usual identification~\cite[Chapter~V, Lemma~2.2]{BM00}) and $\wrt a$ denote the Haar measure on the subgroup $A$ of diagonal matrices of $G$ with all positive entries on the diagonal (under the usual identification~\cite[Chapter~V, Lemma~2.3]{BM00}).  With the modular function \begin{align}\label{eqnModularFct}
\rho(a) = \prod_{i<j} \frac{a_j}{a_i}, \end{align} the measure $\rho(a)~\wrt n~\wrt a$ is a left-Haar measure on $NA$~\cite[Chapter~V, Lemma~2.4]{BM00}.  Recall that $\wrt{\widetilde{k}}$ is a Haar measure on $K'$.  

\begin{lemm}\label{lemm:FullSecCoordLeftHaarHPhi}  We have that $\rho(a)~\wrt n~\wrt a~\wrt {\widetilde{k}}$ is a left Haar measure on $H \widetilde{\Phi} = NAK'.$
 
\end{lemm}

\begin{proof}
 By (\ref{eqnFundDomCoord}), we have that $H \widetilde{\Phi} = NAK'.$  Let $\rho_{NAK'}$ and $\rho_{K'}$ be the modular functions of $NAK'$ and $K'$, respectively.  As $K'$ is compact, we have that $\rho_{NAK'} \vert_{K'}=1$ and $\rho_{K'} =1$.  Consequently, it follows that there exists a (left) $NAK'$-invariant regular Borel measure $\nu$, unique up to scalar multiple, on the homogeneous space $NAK'/K'$ such that $\wrt \nu~\wrt {\widetilde{k}}$ is a left Haar measure on $NAK'$~\cite[Corollary~B.1.7]{BHV08}.  Pushing forward $\nu$ via the homeomorphism $NAK'/K' \cong NA$ yields a left Haar measure on $NA$.  The desired result now follows.
\end{proof}

Thus by (\ref{eqn:leftHaarHTildePhi}) and the lemma, we have that \begin{align}\label{eqnDecompOfMeasNAtildek}
 y^{-d}~\wrt \mu_H \frac{\wrt y}{y} = C_0 \rho(a)~\wrt n~\wrt a~\wrt {\widetilde{k}}\end{align} for a fixed positive constant $C_0$.  Using this, we could compute the measure of $\Cc^T$ explicitly, but we only need a ratio:

\begin{lemm}\label{lemmRatioHaarMeasOnSections} We have that 
\[\int\int\int \In_{\Cc^T}(n a \widetilde{k}) \rho(a)~\wrt n~\wrt a~\wrt {\widetilde{k}}= \frac 1 {T^d} \int\int\int \In_{\Cc}(n a \widetilde{k})\rho(a)~\wrt n~\wrt a~\wrt {\widetilde{k}}.\]
\end{lemm}

\begin{proof}
Recall from (\ref{eqnRelCTAndC}) that $\Cc^T$ and $\Cc$ have the same $\widetilde{K}$ and, hence, the integral over $\wrt {\widetilde{k}}$ cancels.  Recall that the flow \[ \begin{pmatrix}
T^{1/(d-1)} I_{d-1}&\tensor*[^t]{\boldsymbol{0}}{}  \\ \boldsymbol{0} & T^{-1}   \end{pmatrix}\] commutes with $K'$ and hence we can change coordinates \[\widetilde{a} = a  \begin{pmatrix}
T^{1/(d-1)} I_{d-1}&\tensor*[^t]{\boldsymbol{0}}{}  \\ \boldsymbol{0} & T^{-1}   \end{pmatrix}.\]  Note that \[\In_{\Cc^T}(n \widetilde{a} \widetilde{k}) = \In_{\Cc}(n a \widetilde{k}).\]

 Using (\ref{eqnModularFct}) and the coordinates given by (\ref{eqnFundDomCoord}), we have that \[\rho(\widetilde{a}) = \frac 1 {T^d} \rho (a)\] and, hence, we have that \[\int  \In_{\Cc^T}(n \widetilde{a} \widetilde{k}) \rho(\widetilde{a})~\wrt{\widetilde{a}} = \int \In_{\Cc}(n a \widetilde{k}) \rho(\widetilde{a})~\wrt {\widetilde{a}} = \frac 1 {T^d}  \int \In_{\Cc}(n a \widetilde{k})\rho(a)~\wrt a.\]  Finally, recall that the $x$-coordinates of $\Cc^T$ and $\Cc$ are the same.  Consequently, we have our desired result.

\end{proof}

\subsubsection{Volumes of thickened neighborhoods}

Let $\Bb \subset \TT^{d-1}$.   Let us now thicken our neighborhoods:  \[\Cc N^+(\Bb) \quad \textrm{and} \quad \Cc^T N^+(\Bb).\]

\begin{theo}\label{thmHaarMeaCN}
We have that \begin{align*} \mu(\Cc N^+(\Bb)) =  \frac 1 {\zeta(d)} \left( \int_{\RR^{d-1}} \In_{\Bb}(\boldsymbol{y'})~\wrt{\boldsymbol{y'}}\right) \left(\int \int \In_{\Cc}\left(M \begin{pmatrix}
y^{1/(d-1)} I_{d-1}&\tensor*[^t]{\boldsymbol{0}}{}  \\ \boldsymbol{0} & y^{-1}   \end{pmatrix}\right) y^{-d}~\wrt \mu_H \frac{\wrt y}{y} \right).
\end{align*}
\end{theo}

\begin{proof}
Using (\ref{eqnFormulaDy}) and (\ref{eqnHaarInMMyCoord}) and changing variables $y_d \mapsto y^{-1}$, we have \begin{align*} \mu(\Cc N^+(\Bb)) & \\ =\zeta(d)^{-1} &\int \int \int \In_{\Cc N^+(\Bb)}\left(M \begin{pmatrix}
y^{1/(d-1)} I_{d-1}&\tensor*[^t]{\boldsymbol{0}}{}  \\ \boldsymbol{0} & y^{-1}   \end{pmatrix}\begin{pmatrix}
 I_{d-1}&\tensor*[^t]{\boldsymbol{0}}{}  \\ y\boldsymbol{y'}/y & 1   \end{pmatrix}\right) y^{-d}~\wrt \mu_H \frac{\wrt y}{y}y^{d-1}\frac{\wrt{\boldsymbol{y'}}}{y^{d-1}} \\ =\zeta(d)^{-1} &\int \int \In_{\Cc}\left(M \begin{pmatrix}
y^{1/(d-1)} I_{d-1}&\tensor*[^t]{\boldsymbol{0}}{}  \\ \boldsymbol{0} & y^{-1}   \end{pmatrix}\right) \left(\int \In_{N^+(\Bb)}\left(\begin{pmatrix}
 I_{d-1}&\tensor*[^t]{\boldsymbol{0}}{}  \\ y\boldsymbol{y'}/y & 1   \end{pmatrix}\right)y^{d-1}\frac{\wrt{\boldsymbol{y'}}}{y^{d-1}} \right)y^{-d}~\wrt \mu_H \frac{\wrt y}{y}.
  \end{align*}  Now we have that \begin{align*}\int \In_{N^+(\Bb)}\left(\begin{pmatrix}
 I_{d-1}&\tensor*[^t]{\boldsymbol{0}}{}  \\ y\boldsymbol{y'}/y & 1   \end{pmatrix}\right) y^{d-1}\frac{\wrt{\boldsymbol{y'}}}{y^{d-1}} &= \int \In_{N^+(\Bb)}\left(\begin{pmatrix}
 I_{d-1}&\tensor*[^t]{\boldsymbol{0}}{}  \\ \boldsymbol{y'} & 1   \end{pmatrix}\right)~\wrt{\boldsymbol{y'}} = \int_{\RR^{d-1}} \In_{\Bb}(\boldsymbol{y'})~\wrt{\boldsymbol{y'}},
\end{align*} giving the desired result.

\end{proof}

\begin{coro}\label{coroHaarMeaCTN}
We have that \begin{align*}\mu(\Cc^T N^+(\Bb))& = \frac 1 {T^d} \mu(\Cc N^+(\Bb)) \\ &=  \frac 1 {\zeta(d) T^d} \left( \int_{\RR^{d-1}} \In_{\Bb}(\boldsymbol{y'})~\wrt{\boldsymbol{y'}}\right) \left(\int \int \In_{\Cc}\left(M \begin{pmatrix}
y^{1/(d-1)} I_{d-1}&\tensor*[^t]{\boldsymbol{0}}{}  \\ \boldsymbol{0} & y^{-1}   \end{pmatrix}\right) y^{-d}~\wrt \mu_H \frac{\wrt y}{y} \right).
\end{align*}
\end{coro}

\begin{proof} Conjugation gives \[N^+(\Bb)=\begin{pmatrix}
T^{1/(d-1)} I_{d-1}&\tensor*[^t]{\boldsymbol{0}}{}  \\ \boldsymbol{0} & T^{-1}   \end{pmatrix}^{-1} N^+(T^{-d/(d-1)}\Bb) \begin{pmatrix}
T^{1/(d-1)} I_{d-1}&\tensor*[^t]{\boldsymbol{0}}{}  \\ \boldsymbol{0} & T^{-1}   \end{pmatrix}.\]  Since the diagonal flow preserves the Haar measure, we have that \begin{align*} \mu(\Cc^T N^+(\Bb)) =  \mu\left(\Cc N^+(T^{-d/(d-1)}\Bb)\right),
\end{align*} which, using Theorem~\ref{thmHaarMeaCN}, gives the desired result.
 
\end{proof}

Using full section coordinates, we could explicitly compute volumes for $\Cc^T$ to obtain a formula in terms of $\boldsymbol{\alpha}, \boldsymbol{\gamma}, \widetilde{K}, \beta_{ij}^+, \beta_{ij}^-$, and $T$.  However, we will do this only in the case of $\Ss_T$:

\begin{coro}\label{coroHaarMeaSTN}
We that that \[\mu(\Ss_T N^+(\Bb)) = \frac 1 {d \zeta(d)} \frac 1{T^{d-1}} \int_{\RR^{d-1}} \In_{\Bb}(\boldsymbol{y'})~\wrt{\boldsymbol{y'}}\]
\end{coro}

\begin{proof}
Changing variables $y = e^{(d-1)s}$ in Theorem~\ref{thmHaarMeaCN}, we obtain\begin{align*}\mu(\Ss_T N^+(\Bb)) &=   \frac {d-1} {\zeta(d)} \left( \int_{\RR^{d-1}} \In_{\Bb}(\boldsymbol{y'})~\wrt{\boldsymbol{y'}}\right) \left(\int \int \In_{\Ss_T}\left(M \Phi^{-s} \right)~\wrt \mu_H \ e^{-d(d-1)s}~\wrt s \right)  \\ \nonumber &=   \frac {d-1} {\zeta(d)} \left( \int_{\RR^{d-1}} \In_{\Bb}(\boldsymbol{y'})~\wrt{\boldsymbol{y'}}\right) \left(\int_{\frac 1 d \log(T)}^\infty  e^{-d(d-1)s}~\wrt s \right)   \\ \nonumber &= \frac 1 {d \zeta(d)} \frac 1{T^{d-1}}\int_{\RR^{d-1}} \In_{\Bb}(\boldsymbol{y'})~\wrt{\boldsymbol{y'}}.
\end{align*}
\end{proof}

Finally, recall the definition of $\Ee$ from Section~\ref{secConstructShrinkTargets}.  We now compute the volume of $\Ss_{T} \Ee$ (also see Theorem~\ref{thmHaarMeaWidetildeSE} for a generalization).

\begin{theo}\label{thmHaarMeaSE}  We have that  \begin{align*}\mu(\Ss_{T} \Ee) =\frac{d-1}{\zeta(d)} \int \int \int \In_{\left\{\Ss_{Tc^{-d/(d-1)}} n_+(\boldsymbol{z'}):  \boldsymbol{z} \in \Dd \right\}}\left(M \Phi^{-s}n_+(\widetilde{\boldsymbol{x}})\right)~\wrt \mu_H(M) e^{-d(d-1)s}~\wrt s~\wrt{\widetilde{\boldsymbol{x}}}.
\end{align*}
\end{theo}

\begin{proof}
Using (\ref{eqnFormulaDy}), (\ref{eqnHaarInMMyCoord}) and Lemma~\ref{lemmMMyCoordinatesForSE} and changing variables $y_d \mapsto y^{-1}$, we have \begin{align*} \mu(\Ss_{T} \Ee) =\zeta(d)^{-1} \int \int \int \In_{\left\{\Ss_{Tc^{-d/(d-1)}} n_+(\boldsymbol{z'}):  \boldsymbol{z} \in \Dd \right\}}\left(M \begin{pmatrix}
y^{1/(d-1)} I_{d-1}&\tensor*[^t]{\boldsymbol{0}}{}  \\ \boldsymbol{0} & y^{-1}   \end{pmatrix}\begin{pmatrix}
 I_{d-1}&\tensor*[^t]{\boldsymbol{0}}{}  \\ y\boldsymbol{y'}/y & 1   \end{pmatrix}\right) & \\  y^{-d}~\wrt \mu_H \frac{\wrt y}{y}y^{d-1}&\frac{\wrt{\boldsymbol{y'}}}{y^{d-1}}.
   \end{align*}

  Changing variables $y = e^{(d-1)s},$ we obtain  \begin{align*} \mu(\Ss_{T} \Ee) &  =\frac{d-1}{\zeta(d)} &\int \int \int \In_{\left\{\Ss_{Tc^{-d/(d-1)}} n_+(\boldsymbol{z'}):  \boldsymbol{z} \in \Dd \right\}}\left(M \Phi^{-s}\begin{pmatrix}
 I_{d-1}&\tensor*[^t]{\boldsymbol{0}}{}  \\ \boldsymbol{y'} & 1   \end{pmatrix}\right)~\wrt \mu_H e^{-d(d-1)s}~\wrt s~\wrt{\boldsymbol{y'}}, 
   \end{align*} giving the desired result.
  \end{proof}
  
We also have an analog of Corollary~\ref{coroHaarMeaCTN} for $\Ss_{T} \Ee$ (also see Theorem~\ref{thnHaarMeaSTEvsST0ECuspidalVersion} for a generalization).  First note that we can choose a convenient parametrization of $G$ using full section coordinates (see (\ref{eqnFundDomCoord})) and the Iwasawa decomposition of $G$.  Namely, any element of $G$ can be expressed as \begin{align}\label{eqnFundDomCoordForG}  \begin{pmatrix}  1&x_{ij} &\\ &\ddots& \\& & 1\end{pmatrix}y^{-1/2} \begin{pmatrix}y_{d-1}y_{d-2} \cdots y_{1}&&&& \\ &\ddots&&&&\\ && y_{d-1} y_{d-2}&&\\ &&&y_{d-1}&\\ &&&& 1   \end{pmatrix}\begin{pmatrix} k  &  \tensor*[^t]{\boldsymbol{0}}{} \\ \boldsymbol{0} & 1\end{pmatrix} \widehat{k}\end{align} where $\widehat{k} \in  K' \backslash \SO(d, \RR).$  Recall $\wrt k$ is the probability Haar measure on $\SO(d, \RR)$.  It is well-known that $\rho(a)~\wrt n~\wrt a~\wrt k$ is a Haar measure on $G$~\cite[Chapter V, Lemma~2.5]{BM00}.  Since $ \SO(d, \RR)$ and $K'$ are both unimodular, there exists (\cite[Corollary~B.1.7]{BHV08}) a right $K$-invariant regular Borel measure $\wrt k'$ (unique up to positive scalar multiple) on $K' \backslash \SO(d, \RR)$ such that \begin{align}\label{eqnComponentdk}
\wrt k =\wrt {\widetilde{k}}~\wrt k'. \end{align}

\begin{theo}\label{thnHaarMeaSTEvsST0E}
We have that \begin{align*}\mu(\Ss_{T} \Ee)& = \frac {T_0^{d-1}} {T^{d-1}} \mu(\Ss_{T_0} \Ee).\end{align*}
 
\end{theo}
\begin{proof}
We use full section coordinates and, in particular, Lemma~\ref{lemmRatioHaarMeasOnSections}.  Let $\Cc := \Ss_{T_0}$ and \[\widetilde{T}:= \left(\frac{T}{T_0} \right)^{\frac{d-1}d}.\]  Then, by (\ref{eqnRelCTAndC}), we have that $\Cc^{\widetilde{T}}= \Ss_T.$  Since $\Ss_{T_0}$ contains all of $\Ga \backslash \Ga H$, then, using (\ref{eqnDecompsitionOfEleOfH}), we have that $\widetilde{K}$ is all of $K'$ where $\widetilde{K}$ is the set defined in (\ref{eqnDefofTildeK}).  

For $g \in \SO(d, \RR)$, let us denote its image under the natural projection map $\SO(d, \RR) \rightarrow K' \backslash \SO(d, \RR)$ by $[g]$ and, similarly for subsets $\widetilde{\Ee} \subset \SO(d, \RR)$, let us denote its image by $[\widetilde{\Ee}] := \{[g] : g \in \widetilde{\Ee}\}.$  We note that $[\widetilde{K} \Ee] = [\Ee]$.   Using (\ref{eqnComponentdk}) and Lemma~\ref{lemm:FullSecCoordLeftHaarHPhi}, we have \begin{align*}\int\int\int \In_{\Cc^{\widetilde{T}}\Ee}(n a k) \rho(a)&~\wrt n~\wrt a~\wrt {{k}}  \\\nonumber & =  
 \int\int\int\int \In_{\Cc^{\widetilde{T}}[\Ee]}(n a \widetilde{k}k') \rho(a)~\wrt n~\wrt a~\wrt {\widetilde{k}}~\wrt {{k'}}  \\\nonumber & =  \left(\int\int\int \In_{\Cc^{\widetilde{T}}}(n a \widetilde{k}) \rho(a)~\wrt n~\wrt a~\wrt {\widetilde{k}}\right) \left(\int \In_{[\Ee]}(k')~\wrt {{k'}}\right)\\\nonumber &= \frac 1 {\widetilde{T}^d}\left(\int\int\int \In_{\Cc}(n a \widetilde{k})\rho(a)~\wrt n~\wrt a~\wrt {\widetilde{k}}\right)\left(\int \In_{[\Ee]}(k')~\wrt {{k'}}\right)  \\\nonumber & =  \frac{T_0^{d-1}}{T^{d-1}}\int\int\int\int \In_{\Cc[\Ee]}(n a \widetilde{k}k') \rho(a)~\wrt n~\wrt a~\wrt {\widetilde{k}}~\wrt {{k'}}\\\nonumber & =  \frac{T_0^{d-1}}{T^{d-1}}\int\int\int \In_{\Cc\Ee}(n a k) \rho(a)~\wrt n~\wrt a~\wrt {{k}}, \end{align*} where the third equality follows by an application of Lemma~\ref{lemmRatioHaarMeasOnSections} with $\Cc$ and $\Cc^{\widetilde{T}}$.  Since $\rho(a)~\wrt n~\wrt a~\wrt {{k}}$ is a positive constant multiple of the Haar measure $\mu$, the desired result is immediate. 
 \end{proof}

\subsection{Proof of Theorem~\ref{thmShrinkCuspNeigh} for $L=I_d$}\label{subsecProofMainTheoremFullCase}  Recall the definition of $\Cc$ from Section~\ref{secConstructShrinkTargets} and that, in particular, it has lower height $T_-$ (which is a fixed constant).  Set \begin{align*} T_0:= T_-^{d/2(d-1)}. 
\end{align*}  Now consider the neighborhood $\Cc^{(TT_0^{-1})^{(d-1)/d}}$, which, by (\ref{eqnRelCTAndC}), is \begin{align}\label{eqnPullingBackOfSectionalNeigh}
\widetilde{\Cc}_T:=\Cc^{(TT_0^{-1})^{(d-1)/d}} = \Cc \Phi^{-\frac 1 d \log(T/T_0)} \end{align} and, using (\ref{eqnFundDomCoord}), has lower height $T^{2(d-1)/d}$.  Note that the lower height of $\Ss_{T_0}$ is the same as that of $\Cc$ and the lower height of $S_T$ is the same as that of $\widetilde{\Cc}_T$.

Recall that $\widetilde{\Cc}_T \subset \Ss_1$ and we can (translate and) thicken this neighborhood by $N^+_\varepsilon (\boldsymbol{\widetilde{y}})$, namely we consider the set $\widetilde{\Cc}_T N^+_\varepsilon(\boldsymbol{\widetilde{y}})$.  As in the proof of Theorem~\ref{thmNANResultRankOneFlow}, if the point $\Ga\tensor*[^t]{\left(n_-(\boldsymbol{r}) \Phi^t\right)}{}^{-1} $ intersects the neighborhood $\widetilde{\Cc}_T$, then the $N^+$-volume of the horosphere through this unique intersection point and meeting the thickening is given by \[\widehat{f}_{T, \varepsilon}(\boldsymbol{r}, n_-(\boldsymbol{r}) \Phi^t) := \int_{N^+} \In_{\Aa \times\widetilde{\Cc}_T N^+_\varepsilon(\boldsymbol{\widetilde{y}})}\left(\boldsymbol{r}, \tensor*[^t]{\left(n_-(\boldsymbol{r}) \Phi^t\right)}{}^{-1} n_+(\boldsymbol{\widetilde{x}})\right) ~\wrt {\widetilde{n}(\boldsymbol{\widetilde{x}})}.\] 

Recall that we are given $T \in [T_-^{d/2(d-1)},e^{dt \eta}]$.  We have the following lemma:

 \begin{lemm}\label{lemmUnstableVolumes2}  Let \[\widetilde{Q} = e^{(d-1)\left(t-\frac 1 d \log(T/T_0)\right)}.\]  
We have that 
 \[\frac {T^{d-1}} {Q^d} \sum_{\boldsymbol{r} \in \F_Q} \widehat{f}_{T, \varepsilon}(\boldsymbol{r}, n_-(\boldsymbol{r}) \Phi^t) = \frac {T_0^{d-1}} {\widetilde{Q}^{d}}\sum_{\boldsymbol{r} \in \F_{\widetilde{Q}}} \widehat{f}_{T_0, \varepsilon}\left(\boldsymbol{r}, n_-(\boldsymbol{r}) \Phi^{t- \frac 1 d \log(T/T_0)}\right).\] 
\end{lemm}

\begin{proof}
Replace $\Ss_{T_0}$ with $\Cc$ and $\Ss_T$ with $\widetilde{C}_T$ and use (\ref{eqnPullingBackOfSectionalNeigh}) in place of (\ref{eqnPullingBackOfSection}) in the proof of Lemma~\ref{lemmUnstableVolumes}.  We note that the set \[\Ga \tensor*[^t]{\left(n_-(\boldsymbol{r}) \Phi^{t - \frac 1 d \log(T/T_0)}\right)}{}^{-1}  \cap \Cc\] is nonempty for all $t$ large enough by (\ref{eqnChangeOfRate}) and Theorem~\ref{thmMarklof}.  This yields the desired result.
\end{proof}

Using Lemma~\ref{lemmUnstableVolumes2} in place of  Lemma~\ref{lemmUnstableVolumes} in (\ref{eqnEquiOnSectTransA}) yields 

 \begin{align}\label{eqnEquiOnSectTransB}\lim_{t \rightarrow \infty} &\frac {T^{d-1}} {|\F_Q|} \sum_{\boldsymbol{r} \in \F_Q}  \widehat{f}_{T, \varepsilon}(\boldsymbol{r}, n_-(\boldsymbol{r}) \Phi^t)  \\ \nonumber&
  = d(d-1)T_0^{d-1}\varepsilon^{d-1} \int_0^\infty \int_{\TT^{d-1} \times \Ga_H \backslash H} \In_{\Aa \times \Cc}(\boldsymbol{x},  M \Phi^{-s})~\wrt{\boldsymbol{x}}~\wrt \mu_H(M) e^{-d(d-1)s}~\wrt s 
  \\ \nonumber&
  = T_0^{d-1} d \zeta(d) \mu\left(\Cc N^+_\varepsilon(\boldsymbol{\widetilde{y}})\right)\int_{\TT^{d-1}} \In_{\Aa}(\boldsymbol{x})~\wrt{\boldsymbol{x}}.
  \end{align}  

Since \[\varepsilon^{d-1} =  \int_{\RR^{d-1}} \In_{N^+_\varepsilon(\boldsymbol{\widetilde{y}})}(\boldsymbol{y'})~\wrt{\boldsymbol{y'}},\] the second equality follows by changing variables $y = e^{(d-1)s}$ in Theorem~\ref{thmHaarMeaCN}.  Consequently, we have that \begin{align*}\lim_{t \rightarrow \infty} \frac {T^{d-1}} {Q^d} \sum_{\boldsymbol{r} \in \F_Q}  \widehat{f}_{T, \varepsilon}(\boldsymbol{r}, n_-(\boldsymbol{r}) \Phi^t)    
  = T_0^{d-1}\mu\left(\Cc N^+_\varepsilon(\boldsymbol{\widetilde{y}})\right)\int_{\TT^{d-1}} \In_{\Aa}(\boldsymbol{x})~\wrt{\boldsymbol{x}}.
\end{align*}

Now following the rest of the proof of Theorem~\ref{thmNANResultRankOneFlow}, we have that

\[T^{d-1} \int_{\Aa} \In_{\widetilde{\Cc}_T N^+_\varepsilon (\boldsymbol{\widetilde{y}})}({n_-(\boldsymbol{x}) \Phi^t})~\wrt {\boldsymbol{x}} \xrightarrow[]{t \rightarrow \infty} T_0^{d-1}\mu\left(\Cc N^+_\varepsilon(\boldsymbol{\widetilde{y}})\right)\int_{\TT^{d-1}} \In_{\Aa}(\boldsymbol{x})~\wrt{\boldsymbol{x}},\] which completes the proof of Theorem~\ref{thmShrinkCuspNeigh} (when $L=I_d$) for $\Bb = N^+_\varepsilon(\boldsymbol{\widetilde{y}})$.

Finally, note that if we replace $N^+_\varepsilon(\boldsymbol{\widetilde{y}})$ with $N^+(\Bb)$, the proof works provided that the thickening given by $\widetilde{\Cc}_T N^+(\Bb)$ does not self-intersect on $\Ga \backslash G$.  By Lemma~\ref{lemmDisjointnessTransSections}, the latter condition is satisfied by choosing $\Bb \subset \left(\frac 1 {\sqrt{d}}\left( \frac 3 4 \right)^{(d-1)/2} T \right) \ZZ^{d-1} \backslash \RR^{d-1}$.  This proves Theorem~\ref{thmShrinkCuspNeigh} for $L=I_d$.

 \section{Disjointness of translates of neighborhoods of $\Ss_T$}\label{secDisjTransNeigST}  In this section, we give conditions for the disjointness of translates for stable directions and for spherical directions.  Our aim is to understand how much we can thicken without overlaps.  Pick a $T \geq 1$. Recall that  the lower height of $\widetilde{\Cc}_T$ is the same as that of $S_T$ and, thus, $\widetilde{\Cc}_T \subset S_T$.  Consequently, it suffices to consider the disjointness of translates of $\Ss_T$.  
 
\subsection{Stable directions} 
Consider elements of the set \[\Ga \backslash \Ga H \{\Phi^{-s}: s \in \RR_{\geq \frac 1 d \log T}\} n_+(\widetilde{\boldsymbol{x}}_1) \bigcap \Ga \backslash \Ga H \{\Phi^{-s}: s \in \RR_{\geq \frac 1 d \log T}\}  n_+(\widetilde{\boldsymbol{x}}_2).\]  Without loss of generality, by setting $\widetilde{\boldsymbol{x}} := \widetilde{\boldsymbol{x}}_1 - \widetilde{\boldsymbol{x}}_2$, we may consider instead elements of the set \[\Ga \backslash \Ga H \{\Phi^{-s}: s \in \RR_{\geq \frac 1 d \log T}\} n_+(\widetilde{\boldsymbol{x}}) \bigcap \Ga \backslash \Ga H \{\Phi^{-s}: s \in \RR_{\geq \frac 1 d \log T}\}.\] Such an element can be expressed as the following equality \begin{align}\label{eqnThickeningRelation}
 \gamma M \Phi^{-s} n_+(\widetilde{x}) = \widehat{\gamma} \widehat{M} \Phi^{-\widehat{s}}, \end{align} for some $\gamma, \widehat{\gamma} \in \Ga$, $M, \widehat{M} \in H$ and $s, \widehat{s} \geq \log T/d$.  Let \begin{align}\label{eqnDefGammaMMHat}
 \gamma^{-1} \widehat{\gamma}  := \begin{pmatrix}  N &  \tensor*[^t]{\boldsymbol{m}}{} \\ \boldsymbol{p} & q \end{pmatrix} \quad M := \begin{pmatrix} B &  \tensor*[^t]{\boldsymbol{b}}{} \\ \boldsymbol{0} & 1\end{pmatrix} \quad  \widehat{M} := \begin{pmatrix} \widehat{B} &  \tensor*[^t]{\widehat{\boldsymbol{b}}}{} \\ \boldsymbol{0} & 1\end{pmatrix}
  \end{align}
 
 Computing using (\ref{eqnThickeningRelation}), we obtain the following:  \begin{align}\label{eqnxtildeBhat}\widetilde{\boldsymbol{x}} = e^{(d-1)s + \widehat{s}}\boldsymbol{p} \widehat{B}.
  \end{align}

\noindent Note that, by (\ref{eqnFundDomCoord}), $\widehat{M}$ is an element in the Grenier domain of height 1.

\begin{lemm}\label{lemmDisjointnessTransSections}  We have that \[\Ga \backslash \Ga H \{\Phi^{-s}: s \in \RR_{\geq \frac 1 d \log T}\} n_+(\widetilde{\boldsymbol{x}}_1) \bigcap \Ga \backslash \Ga H \{\Phi^{-s}: s \in \RR_{\geq \frac 1 d \log T}\}  n_+(\widetilde{\boldsymbol{x}}_2) = \emptyset\] for all $\widetilde{\boldsymbol{x}}_1 \neq \widetilde{\boldsymbol{x}}_2 \mod (C_d T) \ZZ^{d-1},$ where $C_d$ is a constant depending only on $\Gamma$ and $d$ and, for which, we have the following bounds:   

\begin{center}\begin{tabular}{ll} $C_d = 1$ & $\textrm{ for } d=2$ \\ $C_d \geq \frac 1 {\sqrt{d}}\left( \frac 3 4 \right)^{(d-1)/2}$ & $ \textrm{ for } d>2.$
\end{tabular}\end{center}

\end{lemm}

\begin{proof}

 For $d=2$, we have that $\widehat{B} =1$ by the definition of $H$, and, using this in (\ref{eqnxtildeBhat}), the desired result is immediate.
 
 For $d>2$, let $\boldsymbol{n} \in \ZZ^{d-1}\backslash \{\boldsymbol{0}\}$ and $\alpha \in \RR$.  Using (\ref{eqnFundDomCoord}) and the observation that  $\widehat{M}$ has height 1, we have that \begin{align*}\begin{pmatrix}  \boldsymbol{n}, \alpha\end{pmatrix}& \widehat{M} \\ \nonumber &=\begin{pmatrix}  \boldsymbol{n}, \alpha \end{pmatrix} \begin{pmatrix}  1&x_{ij} &\\ &\ddots& \\& & 1\end{pmatrix}\begin{pmatrix}y_{d-1}y_{d-2} \cdots y_{1}&&&& \\ &\ddots&&&&\\ && y_{d-1} y_{d-2}&&\\ &&&y_{d-1}&\\ &&&& 1   \end{pmatrix}\begin{pmatrix} k  &  \tensor*[^t]{\boldsymbol{0}}{} \\ \boldsymbol{0} & 1\end{pmatrix}.
  \end{align*}
  
Since $\boldsymbol{n} \in \ZZ^{d-1}\backslash \{\boldsymbol{0}\}$, multiplying the matrices gives us that \[\left\| \begin{pmatrix}  \boldsymbol{n}, \alpha \end{pmatrix} \begin{pmatrix}  1&x_{ij} &\\ &\ddots& \\& & 1\end{pmatrix} \right\|_\infty \geq 1.\]  Recalling that all the $y_{d-k}^2 \geq 3/4$ (\cite[Chapter 1,~Lemma~4.1]{JL05}) and that the rotation \[\begin{pmatrix} k  &  \tensor*[^t]{\boldsymbol{0}}{} \\ \boldsymbol{0} & 1\end{pmatrix}\] is invariant under $\|\cdot\|_2$, we obtain our desired result when we set $\alpha =  -\boldsymbol{n} \tensor*[^t]{\widehat{\boldsymbol{b}}}{} $, set $\boldsymbol{p} = \boldsymbol{n}$ in (\ref{eqnxtildeBhat}), and note that \[\begin{pmatrix}  \boldsymbol{p}, -\boldsymbol{p} \tensor*[^t]{\widehat{\boldsymbol{b}}}{} \end{pmatrix} \widehat{M}  = \begin{pmatrix}  \boldsymbol{p}  \widehat{B}, \boldsymbol{0} \end{pmatrix}.\]

\end{proof}

\begin{rema}
Note that $C_d \leq 1$ in the lemma because we have that \[\Phi^{-\frac 1 d \log T} n_+(\boldsymbol{e}_j T) =  \begin{pmatrix}  I_{d-1} &  \tensor*[^t]{\boldsymbol{0}}{} \\ \boldsymbol{e}_j & 1 \end{pmatrix}\Phi^{-\frac 1 d \log T}\] for all $j$ where $\{\boldsymbol{e}_j \}$ is the set of standard basis vectors of $\ZZ^{d-1}$.

\end{rema}

\subsection{Spherical directions}\label{subsecSpherDirectionsDisjoint}   Recall the definitions of $\boldsymbol{z}, E, \Dd,$ and $\Ee$ from Section~\ref{secConstructShrinkTargets}.  Recall that \[\h_0 := \begin{cases} 1 & \textrm{ if } d =2 \\ \sqrt{d} \left(\frac 4 3\right)^{\frac{d-1}2} &\textrm{ if } d \geq3\end{cases}.\]  Pick $T > \h_0$.  Using Iwasawa decomposition and (\ref{eqnFundDomCoord}), we have that every element of \begin{align}\label{eqnHPhiSO}
 H \{\Phi^{-s}: s \in \RR_{\geq \frac 1 d \log T}\}  \SO(d, \RR) \end{align} can be uniquely written as  \begin{align}\label{eqnFundDomCoord2} \begin{pmatrix}  1&x_{ij} &\\ &\ddots& \\& & 1\end{pmatrix}y^{-1/2} \begin{pmatrix}y_{d-1}y_{d-2} \cdots y_{1}&&&& \\ &\ddots&&&&\\ && y_{d-1} y_{d-2}&&\\ &&&y_{d-1}&\\ &&&& 1   \end{pmatrix}\begin{pmatrix} k  &  \tensor*[^t]{\boldsymbol{0}}{} \\ \boldsymbol{0} & 1\end{pmatrix} \widehat{k}\end{align} where $k \in \SO(d-1, \RR)$ and $\widehat{k} \in K' \backslash \SO(d, \RR)$.  Now (\ref{eqnInducedDiffeoUsinged}) yields a diffeomorphism, call it $\psi$, between $\Dd$ and a submanifold of $S^{d-1}$ because it yields an injective immersion immediately and is, in fact, a proper map.  To see that $\psi$ is a proper map, one notes that any compact subset $\Kk$ of $\psi(\Dd)$ is closed and, hence, so is $\psi^{-1}(\Kk)$, which because it also bounded must be compact. Consequently, we have made the first part of following observation:  
 
 \begin{lemm}\label{lemmInducedDiffeoDeterByvAndc}  
Given $(\widetilde{\boldsymbol{v}}, \widetilde{c}) \in e_dE(\Dd)^{-1}$, there exists a unique $\boldsymbol{z} \in \Dd$ and its corresponding element \[E(\boldsymbol{z})^{-1} := \begin{pmatrix}  A(\boldsymbol{z}) &  \tensor*[^t]{\boldsymbol{w}}{}(\boldsymbol{z}) \\ \boldsymbol{v}(\boldsymbol{z}) & c(\boldsymbol{z})\end{pmatrix}\] such that \[(\widetilde{\boldsymbol{v}}, \widetilde{c}) = (\boldsymbol{v}(\boldsymbol{z}), c(\boldsymbol{z})).\]  Moreover, for each $\boldsymbol{z} \in \Dd$, $E(\boldsymbol{z})^{-1}$ lies in a unique equivalence class of $K' \backslash \SO(d, \RR)$.
\end{lemm}
\begin{proof}
For the second part, note that if $E(\boldsymbol{z})$ and $E(\widetilde{\boldsymbol{z}})$ lie in the same equivalence class, then \[E(\boldsymbol{z}) = \begin{pmatrix} k  &  \tensor*[^t]{\boldsymbol{0}}{} \\ \boldsymbol{0} & 1\end{pmatrix}E(\widetilde{\boldsymbol{z}})\] for some  $k \in \SO(d-1, \RR)$.  Applying $e_d$ on the left yields the desired result.
\end{proof}

\begin{rema}

Hence, the lemma implies that for any element of $\Ee \subset \SO(d, \RR)$, we may identify the equivalence class $\widehat{E(\boldsymbol{z})^{-1}}$ with the element $E(\boldsymbol{z})^{-1}$ in (\ref{eqnFundDomCoord2}).  This continues to hold true even if, in (\ref{eqnHPhiSO}), we replace $H$ by a subset of $H$ comprised of elements expressed uniquely as in (\ref{eqnDecompsitionOfEleOfH}) for a proper subset $\widetilde{K}$ from (\ref{eqnDefofTildeK}).  In particular, if we replace $\widehat{k}$ by $E(\boldsymbol{z})^{-1}$, the decomposition (\ref{eqnFundDomCoord2}) yields unique coordinates for the elements of \[H \{\Phi^{-s}: s \in \RR_{\geq \frac 1 d \log T}\}  \Ee.\]

\end{rema}

\subsubsection{The hemispherical case}

We will first consider the case in which \begin{align} \label{eqnUpperHemiDiffCond} c(\Dd)>0 
  \end{align} hold.  We note that in this case the induced diffeomorphism (\ref{eqnInducedDiffeoUsinged}) is really into a hemisphere.

\begin{theo}\label{thmUpperHemiThickening} Let (\ref{eqnUpperHemiDiffCond}) hold and $T > \h_0$.  For $\boldsymbol{z}_1, \boldsymbol{z}_2 \in \Dd$, we have that \[\Ga \backslash \Ga H \{\Phi^{-s}: s \in \RR_{\geq \frac 1 d \log T}\} E(\boldsymbol{z}_1)^{-1} \bigcap \Ga \backslash \Ga H \{\Phi^{-s}: s \in \RR_{\geq \frac 1 d \log T}\}  E(\boldsymbol{z}_2)^{-1} = \emptyset\] whenever $\boldsymbol{z}_1\neq \boldsymbol{z}_2$.
 
\end{theo}

\begin{proof}

The set \[\Ga \backslash \Ga H \{\Phi^{-s}: s \in \RR_{\geq \frac 1 d \log T}\} E(\boldsymbol{z}_1)^{-1} \bigcap \Ga \backslash \Ga H \{\Phi^{-s}: s \in \RR_{\geq \frac 1 d \log T}\}  E(\boldsymbol{z}_2)^{-1}\]  is equivalent to the set \[\Ga \backslash \Ga H \{\Phi^{-s}: s \in \RR_{\geq \frac 1 d \log T}\} \widehat{k} \bigcap \Ga \backslash \Ga H \{\Phi^{-s}: s \in \RR_{\geq \frac 1 d \log T}\}\] where \begin{align}\label{eqnkEqualsE1InvE2}
 k =   E(\boldsymbol{z}_1)^{-1}  E(\boldsymbol{z}_2). \end{align}  Recalling (\ref{eqnDefGammaMMHat}), we note that such an element can be expressed as the following equality \begin{align}\label{eqnThickeningRelation2}
 \gamma M \Phi^{-s} k = \widehat{\gamma} \widehat{M} \Phi^{-\widehat{s}}, \end{align} for some $\gamma, \widehat{\gamma} \in \Ga$, $M, \widehat{M} \in H$ and $s, \widehat{s} \geq \log T/d$. 
 
Setting \[k = \begin{pmatrix} A &  \tensor*[^t]{\boldsymbol{w}}{} \\ \boldsymbol{v} & c\end{pmatrix}\] (where $A$ is a $d-1 \times d-1$ matrix and $c$ is a real number) and using the same computation for the derivation of (\ref{eqnxtildeBhat}), we obtain its analog for spherical directions:   \begin{align}\label{eqnxtildeBhatSphere}{\boldsymbol{v}} = e^{(d-1)s + \widehat{s}}\boldsymbol{p} \widehat{B}.
  \end{align}
  
 Since $T > \h_0$, we have that $ e^{(d-1)s + \widehat{s}} > \h_0$ and the proof of Lemma~\ref{lemmDisjointnessTransSections} implies that any nonzero entry of $\boldsymbol{v}$ is strictly greater than $1$, which is not possible for $k \in \SO(d, \RR)$.  Consequently, we have that \[\boldsymbol{v}  = \boldsymbol{p} = \boldsymbol{w} =\boldsymbol{0}\] and \[c =  q = \det(N) =\pm1.\]

 Let us write \begin{align*}
&E(\boldsymbol{z}_1)^{-1} = \begin{pmatrix}  A_1 &  \tensor*[^t]{\boldsymbol{w}}{}_1 \\ \boldsymbol{v}_1 & c_1  \end{pmatrix} := \begin{pmatrix}  A(\boldsymbol{z}_1) &  \tensor*[^t]{\boldsymbol{w}}{}(\boldsymbol{z}_1) \\ \boldsymbol{v}(\boldsymbol{z}_1) & c(\boldsymbol{z}_1)  \end{pmatrix} \\  &E(\boldsymbol{z}_2)^{-1} = \begin{pmatrix}  A_2 &  \tensor*[^t]{\boldsymbol{w}}{}_2 \\ \boldsymbol{v}_2 & c_2  \end{pmatrix} := \begin{pmatrix}  A(\boldsymbol{z}_2) &  \tensor*[^t]{\boldsymbol{w}}{}(\boldsymbol{z}_2) \\ \boldsymbol{v}(\boldsymbol{z}_2) & c(\boldsymbol{z}_2)  \end{pmatrix}.  \end{align*}  Using (\ref{eqnkEqualsE1InvE2}) and the fact that $E(\boldsymbol{z}_2) = \tensor*[^t]{\left(E(\boldsymbol{z}_2)\right)}{^{-1}}$, we have that \[\boldsymbol{v}_1 \tensor*[^t]{\boldsymbol{v}}{}_2 + c_1 c_2 = 1,\] which implies that \begin{align}\label{eqnEqualSphereVectorsHemiCase}
 (\boldsymbol{v}_1, c_1)= (\boldsymbol{v}_2, c_2), \end{align}  or \[\boldsymbol{v}_1 \tensor*[^t]{\boldsymbol{v}}{}_2 + c_1 c_2 = -1,\] which implies that \begin{align}\label{eqnEqualSphereVectorsAntiHemiCase}
 (\boldsymbol{v}_1, c_1)= (-\boldsymbol{v}_2, -c_2).  \end{align} Both implications follow from the fact that the vectors are unit.  The second implication is, however, not allowed by  (\ref{eqnUpperHemiDiffCond}).  Consequently, only (\ref{eqnEqualSphereVectorsHemiCase}) holds and we, moreover, must have that \begin{align}\label{eqncqdetNIs1}
 c =  q = \det(N) =1 \end{align}
 Lemma~\ref{lemmInducedDiffeoDeterByvAndc}  now implies that \begin{align*} \boldsymbol{z}_1 &= \boldsymbol{z}_2 \\ E(\boldsymbol{z}_1)^{-1} &= E(\boldsymbol{z}_2)^{-1},
  \end{align*} from which the desired result is immediate.

\end{proof}

\begin{rema}
Theorem~\ref{thmUpperHemiThickening} will remain true if we replace the condition (\ref{eqnUpperHemiDiffCond}) by the condition  $c(\Dd) <0$ because the only time we used (\ref{eqnUpperHemiDiffCond}) is to exclude the second implication from the proof of the theorem.  

\end{rema}
\begin{rema}

 Let \begin{align*}
\partial_{NS} \Dd := \{\boldsymbol{v} \in \overline{\Dd}\backslash\Dd : &\textrm{ there exists an open neighborhood of } \boldsymbol{v} \\ &\textrm{ on which the induced diffeomorphism  (\ref{eqnInducedDiffeoUsinged}) has nonsingular differential}\}.
  \end{align*}

If, in Theorem~\ref{thmUpperHemiThickening}, we replace the condition (\ref{eqnUpperHemiDiffCond}) by the condition  \begin{align*} c(\Dd)>0  \textrm{ or } c(\partial_{NS} \Dd)\geq0,
  \end{align*} then the proof shows that the only elements that may intersect are those for which $c(\partial_{NS} \Dd) = 0$ holds.  Note that, by making this change, we are allowing some points on the boundary of the hemisphere.  Moreover, for each such element $\boldsymbol{z} \in \partial_{NS} \Dd)$ for which $c(\boldsymbol{z})=0$, there is at most one other element $\boldsymbol{z'} \in \partial_{NS} \Dd)$ that it intersects with.  This one other element is its antipodal point (see \ref{defn:antipodalpoint}) and also lies on the boundary of the hemisphere.  This observation about antipodal points will generalize to the complete spherical case.
 \end{rema}

  \subsubsection{The spherical case}\label{subsubsecSphericalCase}  This case is the general case.  There are two natural ways to handle the general case and, since they are both potentially useful for further applications, we give both.  In either way, we handle even and odd $d$ differently.  In particular, for even $d$, there is a canonical choice, which will become apparent later, and, for odd $d$, there is a number of natural choices, which will also become apparent later.  (Our way of handling the odd case will work for the even case too, but the canonical choice for the even case simplifies the proofs and may be useful for later applications.)  We choose one of the natural choices, namely using \[\begin{pmatrix}  \widetilde{I}_{d-1} &  \tensor*[^t]{\boldsymbol{0}}{} \\ \boldsymbol{0} & -1  \end{pmatrix}\] for $d\geq 3$ and odd.  Recall that $\widetilde{I}_{d-1}$ is defined in (\ref{eqnDefTildeIdMin1}).   Other choices would give results analogous to ours.

Let us assume that we still have the mapping $E$ from the hemispherical case; in particular, it satisfies the condition (\ref{eqnUpperHemiDiffCond}).  The first way to handle the general case is to use $E$ to define two related charts on $S^{d-1}$, which allows us to define the following subset of $S^{d-1}$: \begin{align}\label{eqnTwoChartsUpperLowerHemi}
  \widetilde{\Ee} := 
  \begin{cases} 
   \{E(\boldsymbol{z})^{-1} : \boldsymbol{z} \in \Dd\} \sqcup  \{- E(\boldsymbol{z})^{-1} : \boldsymbol{z} \in \Dd\} & \text{if } d \textrm{ is even,}\\
    \{E(\boldsymbol{z})^{-1} : \boldsymbol{z} \in \Dd\} \sqcup  \left\{\begin{pmatrix}  \widetilde{I}_{d-1} &  \tensor*[^t]{\boldsymbol{0}}{} \\ \boldsymbol{0} & -1  \end{pmatrix}E(\boldsymbol{z})^{-1} : \boldsymbol{z} \in \Dd\right\} & \text{if } d \textrm{ is odd.}
  \end{cases} \end{align}
  The fact that the two unions are disjoint follows by Lemma~\ref{lemmInducedDiffeoDeterByvAndc}.  Note that we will not need to consider points for which $c=0$ in this first way, as such points form a null set and does not affect any of our other results in this paper.  For $\boldsymbol{z} \in \Dd$, let us also define \[\widetilde{\Ee} (\boldsymbol{z}) :=  \begin{cases}  \left\{E(\boldsymbol{z})^{-1},- E(\boldsymbol{z})^{-1}\right\} & \text{if } d \textrm{ is even,}\\
\left\{E(\boldsymbol{z})^{-1}, \begin{pmatrix}  \widetilde{I}_{d-1} &  \tensor*[^t]{\boldsymbol{0}}{} \\ \boldsymbol{0} & -1  \end{pmatrix}E(\boldsymbol{z})^{-1}\right\} & \text{if } d \textrm{ is odd.}  
\end{cases}\]

\begin{lemm}\label{lemmSkewIdConjH}
We have that \[\begin{pmatrix}  \widetilde{I}_{d-1} &  \tensor*[^t]{\boldsymbol{0}}{} \\ \boldsymbol{0} & -1  \end{pmatrix} H \begin{pmatrix}  \widetilde{I}_{d-1} &  \tensor*[^t]{\boldsymbol{0}}{} \\ \boldsymbol{0} & -1  \end{pmatrix} = H.\]  In particular, given \[\begin{pmatrix} A &  \tensor*[^t]{\boldsymbol{b}}{} \\ \boldsymbol{0} & 1  \end{pmatrix} \in H,\] we have that \[\begin{pmatrix}  \widetilde{I}_{d-1} &  \tensor*[^t]{\boldsymbol{0}}{} \\ \boldsymbol{0} & -1  \end{pmatrix}  \begin{pmatrix} A &  \tensor*[^t]{\boldsymbol{b}}{} \\ \boldsymbol{0} & 1  \end{pmatrix} \begin{pmatrix}  \widetilde{I}_{d-1} &  \tensor*[^t]{\boldsymbol{0}}{} \\ \boldsymbol{0} & -1  \end{pmatrix} = \begin{pmatrix} \widetilde{I}_{d-1} A \widetilde{I}_{d-1}&  -\widetilde{I}_{d-1}\tensor*[^t]{\boldsymbol{b}}{} \\ \boldsymbol{0} & 1  \end{pmatrix}.\]
\end{lemm}

\begin{proof}
 Compute.
\end{proof}

\begin{theo}\label{thmSphericalThickeningConjCharts} Let $T > \h_0$.  For $\boldsymbol{z}_1, \boldsymbol{z}_2 \in \Dd$ and \[E_1^{-1} \in \widetilde{\Ee} (\boldsymbol{z}_1)\quad E_2^{-1} \in \widetilde{\Ee} (\boldsymbol{z}_2),\] we have that \[\Ga \backslash \Ga H \{\Phi^{-s}: s \in \RR_{\geq \frac 1 d \log T}\} E_1^{-1} \bigcap \Ga \backslash \Ga H \{\Phi^{-s}: s \in \RR_{\geq \frac 1 d \log T}\}  E_2^{-1} = \emptyset\] if and only if $\boldsymbol{z}_1\neq \boldsymbol{z}_2$.  

Moreover, if $\boldsymbol{z}_1= \boldsymbol{z}_2=:\boldsymbol{z}$, then \[\Ga \backslash \Ga H \{\Phi^{-s}: s \in \RR_{\geq \frac 1 d \log T}\} E_1^{-1} = \Ga \backslash \Ga H \{\Phi^{-s}: s \in \RR_{\geq \frac 1 d \log T}\}  E_2^{-1}.\]  If $d$ is even, we have the even stronger assertion that, for any point \[\frak{p} \in \Ga \backslash \Ga H \{\Phi^{-s}: s \in \RR_{\geq \frac 1 d \log T}\},\]  the set \[\frak{p} \widetilde{\Ee} (\boldsymbol{z})\] consists of exactly one point in $\Gamma \backslash G$.

\end{theo}

\begin{proof}
 Write  \begin{align*}
&E_1^{-1} = \begin{pmatrix}  A_1 &  \tensor*[^t]{\boldsymbol{w}}{}_1 \\ \boldsymbol{v}_1 & c_1  \end{pmatrix} := \begin{pmatrix}  A(\boldsymbol{z}_1) &  \tensor*[^t]{\boldsymbol{w}}{}(\boldsymbol{z}_1) \\ \boldsymbol{v}(\boldsymbol{z}_1) & c(\boldsymbol{z}_1)  \end{pmatrix} \\  &E_2^{-1} = \begin{pmatrix}  A_2 &  \tensor*[^t]{\boldsymbol{w}}{}_2 \\ \boldsymbol{v}_2 & c_2  \end{pmatrix} := \begin{pmatrix}  A(\boldsymbol{z}_2) &  \tensor*[^t]{\boldsymbol{w}}{}(\boldsymbol{z}_2) \\ \boldsymbol{v}(\boldsymbol{z}_2) & c(\boldsymbol{z}_2)  \end{pmatrix}.  \end{align*}   Following the proof of Theorem~\ref{thmUpperHemiThickening}, we have either (\ref{eqnEqualSphereVectorsHemiCase}) or (\ref{eqnEqualSphereVectorsAntiHemiCase}).  

The case (\ref{eqnEqualSphereVectorsHemiCase}) yields that $c_1=c_2$ (and does not equal zero by assumption).  If $c_1 >0$, then we are exactly in the same case as in Theorem~\ref{thmUpperHemiThickening} and we obtain the same answer, namely $E_1 = E_2$.  If $c_1 < 0$, then, by construction, we have that $-E_1^{-1}, -E_2^{-1}$, when $d$ is even, or  \[\begin{pmatrix}  \widetilde{I}_{d-1} &  \tensor*[^t]{\boldsymbol{0}}{} \\ \boldsymbol{0} & -1  \end{pmatrix}E_2^{-1} \quad \begin{pmatrix}  \widetilde{I}_{d-1} &  \tensor*[^t]{\boldsymbol{0}}{} \\ \boldsymbol{0} & -1  \end{pmatrix}E_2^{-1},\] when $d$ is odd, are in the chart defined by $c >0$.  Lemma~\ref{lemmInducedDiffeoDeterByvAndc} now implies that $\boldsymbol{z}_1 = \boldsymbol{z}_2$ and $E_1 = E_2$.

The case (\ref{eqnEqualSphereVectorsAntiHemiCase}) yields that $c_1 = - c_2$ (and does not equal zero by assumption).  Without loss of generality we may assume that $c_1 >0$.  This implies that $E_1$ and $E_2$ are in different charts.  Moreover, by construction, we have that $E^{-1}_1$ and $-E_2^{-1}$, when $d$ is even, or  \[\begin{pmatrix}  \widetilde{I}_{d-1} &  \tensor*[^t]{\boldsymbol{0}}{} \\ \boldsymbol{0} & -1  \end{pmatrix}E_2^{-1},\] when $d$ is odd, are in the chart defined by $c>0$.  Lemma~\ref{lemmInducedDiffeoDeterByvAndc} now implies that $ \boldsymbol{z}_1 = \boldsymbol{z}_2 $ and that \begin{align*}E_1^{-1} = \begin{cases} - E_2^{-1} & \textrm{ if } d \textrm{ is even}, \\ \begin{pmatrix}  \widetilde{I}_{d-1} &  \tensor*[^t]{\boldsymbol{0}}{} \\ \boldsymbol{0} & -1  \end{pmatrix}E_2^{-1} & \textrm{ if } d \textrm{ is odd}.\end{cases}
  \end{align*}  The first assertion is now immediate.

We now show the second assertion.  We have that $E_1^{-1}, E_2^{-1} \in \widetilde{\Ee} (\boldsymbol{z})$.  When $d$ is odd, we may, without loss of generality, assume that \[E_1^{-1} =\begin{pmatrix}  \widetilde{I}_{d-1} &  \tensor*[^t]{\boldsymbol{0}}{} \\ \boldsymbol{0} & -1  \end{pmatrix}E_2^{-1}.\]  Since \[\begin{pmatrix}  \widetilde{I}_{d-1} &  \tensor*[^t]{\boldsymbol{0}}{} \\ \boldsymbol{0} & -1  \end{pmatrix}\] commutes with $\Phi^{-s}$, is an element of $\Gamma$, and is its own inverse, the desired result follows by an application of Lemma~\ref{lemmSkewIdConjH}.

When $d$ is even, we may, without loss of generality, assume that $E_1^{-1} =-E_2^{-1}$.  This case of the second assertion and the third assertion follow because $-I_d$ commutes with all elements of $\SL(d, \RR)$, which implies that $\frak{p} E(\boldsymbol{z})^{-1}$ and $\frak{p} \left(-E(\boldsymbol{z})^{-1}\right)$ lie in the same equivalence class in $\Gamma \backslash G$.
\end{proof}

Let us now assume that our mapping $E$ is \textit{no longer} constrained by the condition (\ref{eqnUpperHemiDiffCond}).  The second way to handle the general case is to only use the original chart, but we note that a function (defined in Proposition~\ref{propConjAntipodalPoints}) with values in $\SO(d-1, \RR)$ depending on the smooth mapping $E$ arises.  In the special case of the exponential map, which allows us to use one chart for all of the sphere but a single point, the function is particularly simple and, potentially, useful for later applications.

Let us introduce some notation.  Given a point $\boldsymbol{z} \in \Dd$, its \textit{antipodal point} $\boldsymbol{z}_{ap}$ is the unique point in $\Dd$, should it exist, for which \begin{align}\label{defn:antipodalpoint}
 (\boldsymbol{v}(\boldsymbol{z}_{ap}), c(\boldsymbol{z}_{ap})) = (-\boldsymbol{v}(\boldsymbol{z}), -c(\boldsymbol{z})). \end{align}  The uniqueness follows immediately from Lemma~\ref{lemmInducedDiffeoDeterByvAndc}.  Clearly, we also have that $(\boldsymbol{z}_{ap})_{ap} = \boldsymbol{z}$, and hence we say $\boldsymbol{z},\boldsymbol{z}_{ap}$ are \textit{antipodal}.  Using Lemma~\ref{lemmInducedDiffeoDeterByvAndc}, we, likewise, refer to $E(\boldsymbol{z})^{-1}$ and $E(\boldsymbol{z}_{ap})^{-1}$ as \textit{antipodal}.  Let us define  \[\widehat{\Ee} (\boldsymbol{z}) :=  \begin{cases}  \left\{E(\boldsymbol{z})^{-1},E(\boldsymbol{z}_{ap})^{-1}\right\} & \text{if the antipode of } \boldsymbol{z}  \textrm{ exists in } \Dd,
\\\left\{E(\boldsymbol{z})^{-1}\right\} & \text{if the antipode of } \boldsymbol{z}  \textrm{ does not exist in } \Dd.\\
\end{cases}\] 

\begin{prop}\label{propConjAntipodalPoints}
For any antipodal pair $\boldsymbol{z}, \boldsymbol{z}_{ap} \in \Dd$, the mapping \begin{align*} E^{-1}(\boldsymbol{z}) E(\boldsymbol{z}_{ap}): \Ga \backslash \Ga H \{\Phi^{-s}: s \in \RR_{\geq \frac 1 d \log T}\} &\rightarrow \Ga \backslash \Ga H \{\Phi^{-s}: s \in \RR_{\geq \frac 1 d \log T}\} \\ \frak{p} \mapsto \frak{p}E^{-1}(\boldsymbol{z}) E(\boldsymbol{z}_{ap}) =: \frak{q}
  \end{align*} is a smooth automorphism.  This automorphism is solely a function of $\boldsymbol{z}.$  Moreover, the heights of $\frak{p}$ and $\frak{q}$ as elements of the Grenier fundamental domain are both equal.

\end{prop}
\begin{proof}
Since $k:=E^{-1}(\boldsymbol{z}) E(\boldsymbol{z}_{ap})$ is invertible, it suffices, for the first assertion, to show that $\frak{q}$ is in \[\Ga \backslash \Ga H \{\Phi^{-s}: s \in \RR_{\geq \frac 1 d \log T}\}.\]  Write  \begin{align*}
&E(\boldsymbol{z})^{-1} = \begin{pmatrix}  A(\boldsymbol{z}) &  \tensor*[^t]{\boldsymbol{w}}{}(\boldsymbol{z}) \\ \boldsymbol{v}(\boldsymbol{z}) & c(\boldsymbol{z})  \end{pmatrix} \quad E(\boldsymbol{z}_{ap})^{-1} = \begin{pmatrix}  A_{ap}(\boldsymbol{z}) &  \tensor*[^t]{\boldsymbol{w}}{_{ap}}(\boldsymbol{z}) \\ -\boldsymbol{v}(\boldsymbol{z}) & -c(\boldsymbol{z})  \end{pmatrix}.\end{align*}  Since $\boldsymbol{v}(\boldsymbol{z})  \tensor*[^t]{\boldsymbol{v}}{} (\boldsymbol{z})+ \left(c(\boldsymbol{z})\right)^2 =1$, we have that \begin{align*}
k &= \begin{pmatrix}  A(\boldsymbol{z})  \tensor*[^t]{A}{_{ap}} (\boldsymbol{z}) +  \tensor*[^t]{\boldsymbol{w}}{}(\boldsymbol{z}) \boldsymbol{w}_{ap}(\boldsymbol{z}) &  \tensor*[^t]{\boldsymbol{0}}{}\\ \boldsymbol{0} & -1  \end{pmatrix} \\ &=\begin{cases}-\begin{pmatrix}  -\left(A(\boldsymbol{z})  \tensor*[^t]{A}{_{ap}} (\boldsymbol{z}) +  \tensor*[^t]{\boldsymbol{w}}{}(\boldsymbol{z}) \boldsymbol{w}_{ap}(\boldsymbol{z}) \right) &  \tensor*[^t]{\boldsymbol{0}}{}\\ \boldsymbol{0} & 1  \end{pmatrix} & \textrm{ if } d \textrm{ is even}, \\ \begin{pmatrix}  \widetilde{I}_{d-1} &  \tensor*[^t]{\boldsymbol{0}}{} \\ \boldsymbol{0} & -1  \end{pmatrix}  \begin{pmatrix}  \widetilde{I}_{d-1} \left(A(\boldsymbol{z})  \tensor*[^t]{A}{_{ap}} (\boldsymbol{z}) +  \tensor*[^t]{\boldsymbol{w}}{}(\boldsymbol{z}) \boldsymbol{w}_{ap}(\boldsymbol{z}) \right) &  \tensor*[^t]{\boldsymbol{0}}{}\\ \boldsymbol{0} & 1  \end{pmatrix} & \textrm{ if } d \textrm{ is odd},\end{cases} \end{align*}  where $-\left(A(\boldsymbol{z})  \tensor*[^t]{A}{_{ap}} (\boldsymbol{z}) +  \tensor*[^t]{\boldsymbol{w}}{}(\boldsymbol{z}) \boldsymbol{w}_{ap}(\boldsymbol{z}) \right) $, if $d$ is even, and $\widetilde{I}_{d-1} \left(A(\boldsymbol{z})  \tensor*[^t]{A}{_{ap}} (\boldsymbol{z}) +  \tensor*[^t]{\boldsymbol{w}}{}(\boldsymbol{z}) \boldsymbol{w}_{ap}(\boldsymbol{z}) \right)$ are in $\SO(d-1, \RR)$.  Since \[H \begin{pmatrix} R &  \tensor*[^t]{\boldsymbol{0}}{}\\ \boldsymbol{0} & 1  \end{pmatrix} = H\] for all $R \in \SO(d-1, \RR)$ and $k$ commutes with $\Phi^{-s}$, Lemma~\ref{lemmSkewIdConjH} gives the first assertion.

We now show the second assertion.  Let $\gamma M \Phi^{-s}$ be the element of the equivalence class $\frak{p}$ and $\widehat{\gamma} \widehat{M} \Phi^{-\widehat{s}}$ be the element of the equivalence class $\frak{q}$ in the Grenier fundamental domain.   Using (\ref{eqnThickeningRelation2}) with (\ref{eqnDefGammaMMHat}) and solving for $k$, we have the analog of (\ref{eqnxtildeBhatSphere}):  \begin{align*}{\boldsymbol{0}} = e^{(d-1)s + \widehat{s}}\boldsymbol{p} \widehat{B},
  \end{align*} from which we may deduce that $\boldsymbol{p} = \boldsymbol{0}$ as $\widehat{B}$ is invertible.  This implies that \begin{align*}
 q = \det(N) =-1, \end{align*} which further implies that $s = \widehat{s}$, which is the desired result.

Since $\boldsymbol{z}_{ap}$ depends only on $\boldsymbol{z}$, the automorphism depends only on $\boldsymbol{z}.$
\end{proof}

\begin{rema}
We will refer to the elements $\frak{p}, \frak{	q} \in \Ga \backslash \Ga H \{\Phi^{-s}: s \in \RR_{\geq \frac 1 d \log T}\}$ as \textit{associated elements} for the antipodal pair $\boldsymbol{z}, \boldsymbol{z}_{ap}$.

\end{rema}

\begin{theo}\label{thmSphericalThickeningGeneral} Let $T > \h_0$.  For $\boldsymbol{z}_1, \boldsymbol{z}_2 \in \Dd$ and \[E_1^{-1} \in \widehat{\Ee} (\boldsymbol{z}_1)\quad E_2^{-1} \in \widehat{\Ee} (\boldsymbol{z}_2),\] we have that \[\Ga \backslash \Ga H \{\Phi^{-s}: s \in \RR_{\geq \frac 1 d \log T}\} E_1^{-1} \bigcap \Ga \backslash \Ga H \{\Phi^{-s}: s \in \RR_{\geq \frac 1 d \log T}\}  E_2^{-1} = \emptyset\] if and only if $\boldsymbol{z}_1\neq \boldsymbol{z}_2$.  

Moreover, if $\boldsymbol{z}_1= \boldsymbol{z}_2=:\boldsymbol{z}$, then \[\Ga \backslash \Ga H \{\Phi^{-s}: s \in \RR_{\geq \frac 1 d \log T}\} E(\boldsymbol{z})^{-1} = \Ga \backslash \Ga H \{\Phi^{-s}: s \in \RR_{\geq \frac 1 d \log T}\}  E(\boldsymbol{z}_{ap})^{-1}.\]

\end{theo}
\begin{proof}
 Write  \begin{align*}
&E_1^{-1} = \begin{pmatrix}  A_1 &  \tensor*[^t]{\boldsymbol{w}}{}_1 \\ \boldsymbol{v}_1 & c_1  \end{pmatrix} := \begin{pmatrix}  A(\boldsymbol{z}_1) &  \tensor*[^t]{\boldsymbol{w}}{}(\boldsymbol{z}_1) \\ \boldsymbol{v}(\boldsymbol{z}_1) & c(\boldsymbol{z}_1)  \end{pmatrix} \\  &E_2^{-1} = \begin{pmatrix}  A_2 &  \tensor*[^t]{\boldsymbol{w}}{}_2 \\ \boldsymbol{v}_2 & c_2  \end{pmatrix} := \begin{pmatrix}  A(\boldsymbol{z}_2) &  \tensor*[^t]{\boldsymbol{w}}{}(\boldsymbol{z}_2) \\ \boldsymbol{v}(\boldsymbol{z}_2) & c(\boldsymbol{z}_2)  \end{pmatrix}.  \end{align*}  Following the proof of Theorem~\ref{thmUpperHemiThickening}, we have either (\ref{eqnEqualSphereVectorsHemiCase}) or (\ref{eqnEqualSphereVectorsAntiHemiCase}).  

The case (\ref{eqnEqualSphereVectorsHemiCase}) yields that $c_1=c_2$.  Here, we also allow $c_1 = c_2 =0$.  The first assertion for the case (\ref{eqnEqualSphereVectorsHemiCase}) follows as in Theorem~\ref{thmSphericalThickeningConjCharts}.  The case (\ref{eqnEqualSphereVectorsAntiHemiCase}) yields that $c_1 = - c_2$.  Here, we also allow $c_1 = c_2 =0$.  Likewise, the first assertion for the case (\ref{eqnEqualSphereVectorsAntiHemiCase}) follows as in Theorem~\ref{thmSphericalThickeningConjCharts}.

The second assertion follows from Proposition~\ref{propConjAntipodalPoints}.

\end{proof}

\section{Spherical directions:  the general case} \label{secSdTGCatpothmNAK}

The general case has two independent aspects.  We first consider thickenings of subsets of $\Ss_T$ defined by lower bounds for the $y_i$-coordinates on the Grenier domain.  Then we consider thickenings by general subsets of the sphere, including all of the sphere.  Recall the definitions of $\Dd, E, \Ee, \boldsymbol{v},$ and $c$ from Section~\ref{secConstructShrinkTargets} and $\boldsymbol{z'}$ and $c_{\min}$ from Section~\ref{secSdatpothmNAK}.  Before we show either aspect, we consider measurable sets in $\SO(d,\RR)$ induced by our smooth mapping $E$.

\subsection{Subsets of $\SO(d, \RR)$ induced by $E$}\label{subsec:SubsetsOfRotationsInduced}  Recall, from Section~\ref{subsecSpherDirectionsDisjoint}, that $E$ induces a diffeomorphism $\psi$ from $\Dd$ onto a submanifold of $S^{d-1} = K' \backslash \SO(d, \RR)$.  Let $\Pp: \SO(d, \RR) \rightarrow K' \backslash \SO(d, \RR)$ be the natural projection.  Then Lemma~\ref{lemmInducedDiffeoDeterByvAndc} implies that \[\Pp \vert_{\Ee}:\Ee \rightarrow \psi(\Dd)\] is a bijection.  We say a subset $U \subset \SO(d, \RR)$ is \textit{induced by $E$} if \[\Pp(U) = \psi(\Dd).\]  For any $U \subset \SO(d, \RR)$ induced by $E$, define the \textit{slice of $U$ through $E^{-1}(\boldsymbol{z})$} as the following:\[\widetilde{K}_U({\boldsymbol{z}}):=K'E^{-1}(\boldsymbol{z}) \cap U.\]

\begin{lemm}\label{lemmDisjKSlices} For any subset $U$ induced by $E$, we have that \[\bigsqcup_{\boldsymbol{z} \in \Dd} \widetilde{K}_U({\boldsymbol{z}}) = U\] 
\end{lemm}
\begin{proof}
 The disjointness follows by Lemma~\ref{lemmInducedDiffeoDeterByvAndc}.  Let $g \in U$.  Then $\Pp(g) = \psi(\boldsymbol{z})$ for some $\boldsymbol{z} \in \Dd$.  Now $\Pp^{-1}(\psi(\boldsymbol{z})) = K' h$ for some $h \in \SO(d, \RR)$.  Consequently, both $g$ and $E^{-1}(\boldsymbol{z})$ are in $K' h$, which implies the desired result.
\end{proof}

We note that the largest set induced by $E$ is $K' \Ee$, as it contains every set induced by $E$.  The largest set is also an open subset of $\SO(d, \RR)$ because $\Dd$ is open and $K' \Ee = \Pp^{-1}(\psi(\Dd)).$

\subsection{Thickenings of subsets of $\Ss_T$ and Cholesky factorization}\label{subsecThickSubsetST}

To state the result precisely, recall from Section~\ref{secConstructShrinkTargets} our definition of the neighborhood \[\Cc:=\Cc_{T_-, T_+}:=\Cc(\boldsymbol{\alpha}, \boldsymbol{\gamma}, \widetilde{K},  \beta^-_{ij}, \beta^+_{ij})\] in the fundamental domain $\Ff_d'$.  Let us restrict our consideration of such neighborhoods of $\Ff_d'$ to those for which the values for any $x_{ij}$ are only constrained because they are the coordinates of a point in $\Ff_d'$:  \[\widetilde{\Ss}:=\widetilde{\Ss}(\boldsymbol{z}):=\widetilde{\Ss}_{T_-, T_+}(\boldsymbol{z}):=\widetilde{\Ss}\left(\boldsymbol{\alpha}, \boldsymbol{\gamma}, \widetilde{K}(\boldsymbol{z})\right)\] where $T_-, T_+, \boldsymbol{\alpha},$ and $\boldsymbol{\gamma}$ are as in Section~\ref{secConstructShrinkTargets}, and, for any $\boldsymbol{z}$, $\widetilde{K}(\boldsymbol{z})$ is a subset of $K'$.\footnote{As an aside, which is not necessary for our proof, note that we have precisely the following.  Set  \[\widetilde{\Ss}(\boldsymbol{\alpha}, \boldsymbol{\gamma}, \widetilde{K}(\boldsymbol{z})) = \Cc(\boldsymbol{\alpha}, \boldsymbol{\gamma}, \widetilde{K}(\boldsymbol{z}),  \beta^-_{ij}, \beta^+_{ij})\] with  \begin{align*} \begin{cases}  \begin{cases}   \beta^-_{1j}=0 \textrm{ and } \beta^+_{1j}= 1/2  &\textrm {for } j = 2,\cdots, d-1 \\  \beta^-_{1j} = -1/2 \textrm{ and } \beta^+_{1j} = 1/2 &\textrm {for } j =  d \\ \beta^-_{ij}= -1/2 \textrm{ and } \beta^+_{ij} = 1/2 &\textrm {for } 2\leq i < j \leq  d \end{cases} & \textrm{ if } d \textrm{ is even,} \\
\begin{cases}    \beta^-_{1j} = 0 \textrm { and } \beta^+_{1j} =  1/2 &\textrm {for } j = 2,\cdots, d  \\  \beta^-_{ij} = -1/2 \textrm{ and } \beta^+_{ij} = 1/2  &\textrm {for } 2\leq i < j \leq  d \end{cases} & \textrm{ if } d \textrm{ is odd.}\end{cases}
  \end{align*} See Section~\ref{secConstructShrinkTargets}.}  Set \begin{align}\label{eqnRelationLowerHeightToFlow} T_0:= T_-^{d/2(d-1)}. 
\end{align}  Now apply the flow to these neighborhoods using (\ref{eqnRelCTAndC}):  \begin{align}\label{eqnPullingBackOfSectionalNeigh2}
\widetilde{\Ss}_T:=\widetilde{\Ss}_T(\boldsymbol{z}):=\widetilde{\Ss}^{(TT_0^{-1})^{(d-1)/d}} = \widetilde{\Ss} \Phi^{-\frac 1 d \log(T/T_0)} \end{align} and, using (\ref{eqnFundDomCoord}), has lower height $T^{2(d-1)/d}$.  Note that the lower height of $\Ss_{T_0}$ is the same as that of $\widetilde{\Ss}$ and the lower height of $\Ss_T$ is the same as that of $\widetilde{\Ss}_T$.  Also note that, for any $L \in G$, the set $ \Aa_L$ is as defined in Section~\ref{subsec:GammaDuplicates}.

\begin{theo}\label{thmThickenSubsetSectionSpherical}  Let $0<\eta <1$ be a fixed constant, $L \in G$, $T_+\geq T_- > \h_0^{2(d-1)/d}$, $\Aa \subset \Aa_L$ be a bounded subset with measure zero boundary with respect to $\wrt {\boldsymbol{x}}$, $U \subset \SO(d, \RR)$ be a subset measurable with respect to $\wrt k$ and induced by $E$, and let (\ref{eqnRestrictionsOnD}) hold.  Let \[\widetilde{\Ss}(\boldsymbol{z}):=\widetilde{\Ss}_{T_-, T_+}(\boldsymbol{z}):=\widetilde{\Ss}\left(\boldsymbol{\alpha}, \boldsymbol{\gamma}, \widetilde{K}_{U}(\boldsymbol{z})E(\boldsymbol{z})\right),\] and $\widetilde{\Ss}_T$ be defined using (\ref{eqnRelationLowerHeightToFlow} , \ref{eqnPullingBackOfSectionalNeigh2}).  Then \begin{align}\label{eqnthmNAKResultRankOneFlowSubsetThickTransHoro}
 {T^{d-1}}\int_{\RR^{d-1}} \In_{\Aa \times \{\widetilde{S}_T(\boldsymbol{z}) E^{-1}(\boldsymbol{z}): \boldsymbol{z}\in \Dd\}}(\boldsymbol{x}, {Ln_-(\boldsymbol{x}) \Phi^t})~\wrt {\boldsymbol{x}} \xrightarrow[]{t \rightarrow \infty}T_-^{d/2}\mu(\{\widetilde{S}(\boldsymbol{z}) E^{-1}(\boldsymbol{z}):  \boldsymbol{z} \in \Dd\})\left(\int_{\RR^{d-1}} \In_{\Aa}(\boldsymbol{x})~\wrt{\boldsymbol{x}}\right) \end{align} uniformly for all $T \in [T_-^{d/2(d-1)}, e^{dt \eta}].$

\end{theo}

\begin{rema}  Without loss of generality, we may assume that $T_-^{d/2(d-1)} \leq e^{dt \eta}$.  By Theorem~\ref{thnHaarMeaSTEvsST0ECuspidalVersion}, the right-hand side is also equal to \[ \lim_{t \rightarrow \infty} {T^{d-1}} \mu\left(\left\{\widetilde{S}_T(\boldsymbol{z}) E^{-1}(\boldsymbol{z}): \boldsymbol{z}\in \Dd\right\}\right) \left(\int_{\TT^{d-1}} \In_{\Aa}(\boldsymbol{x})~\wrt{\boldsymbol{x}}\right).\]  Also note that one could obtain an explicit formula for $\mu\left(\{\widetilde{S}_T(\boldsymbol{z}) E^{-1}(\boldsymbol{z}): \boldsymbol{z}\in \Dd\}\right)$ in terms of the variables $x_{ij}$, $y_\ell$ from (\ref{eqnFundDomCoord}) using Grenier coordinates and the induced subset $U \subset \SO(d, \RR)$.

\end{rema}

The proof for the case $L =I_d$ will be a modification of the proof of Theorem~\ref{thmNAKResultRankOneFlow} for the case $L=I_d$ and will be given in Section~\ref{secProofthmThickenSubsetSectionSpherical}.  The proof for generic $L \in G$ is given in Section~\ref{subsecUniEquiForTransHoro}.

\subsubsection{Preliminaries.}\label{subsubsecPrelimForThickSubsetST}  An important precursor to the proof, for $d \geq3$, is to recognize that the left-multiplication of the expression $A - \tensor*[^t]{\boldsymbol{w}}{}\boldsymbol{z'}$ from (\ref{eqnDecompSOintoN}) with any $\widetilde{k} \in \SO(d-1, \RR)$ is, heuristically and non-rigorously speaking, a ``rotation'' up to a symmetric factor.  We now make this precise and explicitly determine the factor to precisely understand its affect on the subset of $\Ss_T$.  Also, note that we are requiring $c >0$ because (\ref{eqnRestrictionsOnD}) holds.  Let us assume in this section (Section~\ref{subsubsecPrelimForThickSubsetST}) that $d \geq3$ unless we explicitly state that $d=2$ is under consideration.

\begin{lemm}\label{lemmAlmostRotationSymFac}  Let $d \geq3$ and $\widetilde{k} \in \SO(d-1, \RR)$.  
We have that \begin{align*}\widetilde{k}\left(A - \tensor*[^t]{\boldsymbol{w}}{}\boldsymbol{z'}\right) \tensor*[^t]{\left(A - \tensor*[^t]{\boldsymbol{w}}{}\boldsymbol{z'}\right)}{} \tensor*[^t]{\widetilde{k}}{}&= I_{d-1} + \frac{\widetilde{k}\tensor*[^t]{\boldsymbol{w}}{}\boldsymbol{w}\tensor*[^t]{\widetilde{k}}{}}{c^2} \\ \tensor*[^t]{\left(A - \tensor*[^t]{\boldsymbol{w}}{}\boldsymbol{z'}\right)}{} \left(A - \tensor*[^t]{\boldsymbol{w}}{}\boldsymbol{z'} \right)&= I_{d-1} + \frac{\tensor*[^t]{\boldsymbol{v}}{}\boldsymbol{v}}{c^2}. 
  \end{align*}  Moreover, these products are invertible:  \[\det\left(I_{d-1} + \frac{\widetilde{k}\tensor*[^t]{\boldsymbol{w}}{}\boldsymbol{w}\tensor*[^t]{\widetilde{k}}{}}{c^2} \right) = c^{-2/(d-1)}= \det\left(I_{d-1} + \frac{\tensor*[^t]{\boldsymbol{v}}{}\boldsymbol{v}}{c^2} \right).\]
\end{lemm}

\begin{proof}
We prove the first assertion.  The second is proved in an analogous manner.
 Recalling (\ref{eqnDefnEInv}), we have that \[\begin{pmatrix}  A &  \tensor*[^t]{\boldsymbol{w}}{} \\ \boldsymbol{v} & c  \end{pmatrix} \tensor*[^t]{\begin{pmatrix}  A &  \tensor*[^t]{\boldsymbol{w}}{} \\ \boldsymbol{v} & c  \end{pmatrix}}{} = I_d,\] which yields \begin{align}\label{eqnRotEquivalence} A \tensor*[^t]{A}{} + \tensor*[^t]{\boldsymbol{w}}{}\boldsymbol{w} &= I_{d-1} \\\nonumber A \tensor*[^t]{\boldsymbol{v}}{} +\tensor*[^t]{ \boldsymbol{w}}{} c &= \tensor*[^t]{\boldsymbol{0}}{} \\\nonumber \boldsymbol{v} \tensor*[^t]{A}{} + c \boldsymbol{w} &= \boldsymbol{0} \\\nonumber \boldsymbol{v}\tensor*[^t]{ \boldsymbol{v} }{} +c^2 &=1.
  \end{align}

Simplifying, we have that  \begin{align*}\left(A - \tensor*[^t]{\boldsymbol{w}}{}\boldsymbol{z'}\right) \tensor*[^t]{\left(A - \tensor*[^t]{\boldsymbol{w}}{}\boldsymbol{z'}\right)}{} & = A \tensor*[^t]{A}{} - \frac 1 c A \tensor*[^t]{\boldsymbol{v}}{} \boldsymbol{w} - \frac 1 c \tensor*[^t]{\boldsymbol{w}}{} \boldsymbol{v} \tensor*[^t]{A}{} + \frac {\|\boldsymbol{v}\|_2^2}{c^2} \tensor*[^t]{\boldsymbol{w}}{}\boldsymbol{w} \\ &= I_{d-1} + \tensor*[^t]{\boldsymbol{w}}{}\boldsymbol{w}\left(1 +  \frac {\|\boldsymbol{v}\|_2^2}{c^2}\right) \\ &= I_{d-1} + \frac{\tensor*[^t]{\boldsymbol{w}}{}\boldsymbol{w}}{c^2},
  \end{align*} where the last two equalities follow by applying (\ref{eqnRotEquivalence}).  Since $\widetilde{k} \tensor*[^t]{\widetilde{k}}{} = I_{d-1}$, we have shown the first assertion.  
  
  The final assertion concerning determinants follows immediately from (\ref{eqnDecompSOintoN}) and our requirement that $c>0$.
  \end{proof}
  
To compute the factor, we wish to find a $d-1 \times d-1$ upper triangular matrix $B$ satisfying the following equation \begin{align}\label{eqnSymmetricFactor}
B \tensor*[^t]{B}{}= I_{d-1} + \frac{\widetilde{k}\tensor*[^t]{\boldsymbol{w}}{}\boldsymbol{w}\tensor*[^t]{\widetilde{k}}{}}{c^2}.\end{align}

Let $\boldsymbol{0}_i$ denote the $i$-vector with all zero entries.

\begin{theo}\label{thmSymmetricFactorCholesky} Let $d \geq3$, $\widetilde{k} \in \SO(d-1, \RR)$, and $(\widetilde{w}_1, \cdots, \widetilde{w}_{d-1}) := {\boldsymbol{w}}\tensor*[^t]{\widetilde{k}}{}$.  The solution to (\ref{eqnSymmetricFactor}) is the matrix \[B =  \begin{pmatrix}  \beta_{d-1}& {\tensor*[^t]{ \boldsymbol{s}}{_{d-2}}}{} & {\tensor*[^t]{ \boldsymbol{s}}{_{d-3}}}{}& \cdots &{\tensor*[^t]{ \boldsymbol{s}}{_1}}{} \\ \boldsymbol{0}_{1} & \beta_{d-2} & & \\ \boldsymbol{0}_{2} & & \beta_{d-3} & & \\ && &\ddots& \\ \boldsymbol{0}_{d-2} &&&&\beta_1 \end{pmatrix}\] where \begin{align*} &\beta_k^2:= \begin{cases} 1 +  \frac{\widetilde{w}_{d-1}^2}{c^2} &\textrm{ for } k=1 \\
1+\frac{\widetilde{w}_{d-k}^2}{c^2+\widetilde{w}_{d +1 -k}^2+ \cdots + \widetilde{w}_{d-1}^2} &\textrm{ for } 2\leq k \leq d-1 \end{cases} \end{align*} and  \begin{align*} &\boldsymbol{s}_k:= \begin{cases} \frac{\widetilde{w}_{d-1}}{c\sqrt{c^2+\widetilde{w}_{d-1}^2}} (\widetilde{w}_1, \cdots, \widetilde{w}_{d-2})&\textrm{ for } k=1 \\
\frac{\widetilde{w}_{d-k}}{\sqrt{\left(c^2+\widetilde{w}_{d +1 -k}^2+ \cdots + \widetilde{w}_{d-1}^2\right)\left(c^2+\widetilde{w}_{d-k}^2+ \cdots + \widetilde{w}_{d-1}^2\right)}} (\widetilde{w}_1, \cdots, \widetilde{w}_{d-1-k})&\textrm{ for } 2\leq k \leq d-2 \end{cases}.\end{align*}This is the only $d-1 \times d-1$ upper triangular matrix that is a solution. 
\end{theo}

The symmetric factor is mentioned above is, indeed, $B$ and the rotation mentioned above is the following.

\begin{coro}\label{coroSymmetricFactorCholesky}  Let $d \geq3$ and $\widetilde{k} \in \SO(d-1, \RR)$. 
We have that \[B^{-1} \widetilde{k} \left(A - \tensor*[^t]{\boldsymbol{w}}{}\boldsymbol{z'}\right) \in \SO(d-1, \RR).\]  Moreover, given an open, respectively measurable, subset $\widehat{K} \subset \SO(d-1, \RR)$, we have a smooth, respectively measurable, mapping \[\Rr:  \widehat{K} \times \Dd \rightarrow \begin{pmatrix}  \SO(d-1, \RR) & \tensor*[^t]{ \boldsymbol{0}}{} \\  \boldsymbol{0} & 1\end{pmatrix} \quad (\widetilde{k}, \boldsymbol{z}) \mapsto \begin{pmatrix} B^{-1} \widetilde{k} \left(A - \tensor*[^t]{\boldsymbol{w}}{}\boldsymbol{z'}\right) & \tensor*[^t]{ \boldsymbol{0}}{} \\  \boldsymbol{0} & 1 \end{pmatrix}\]  such that $\Rr(\cdot, \boldsymbol{z})$ is injective for any fixed $\boldsymbol{z} \in \Dd$.

\end{coro}

\begin{proof}
 Apply Theorem~\ref{thmSymmetricFactorCholesky} and Lemma~\ref{lemmAlmostRotationSymFac} to (\ref{eqnSymmetricFactor}).  Note that the invertibility of $B$ follows from the fact that, for any $\boldsymbol{z} \in \Dd$ and $\widetilde{k} \in \SO(d-1, \RR)$, we have that $\beta_k \geq 1$ for every $k =1, \cdots, d-1$.
 
The proof of the injectivity statement is as follows.  Recall that $B$ depends on $\widetilde{k}$.  Let $B_1$ and $B_2$ be as in the theorem for $ \widetilde{k}_1$ and $ \widetilde{k}_2$, respectively.  Note that $B^{-1}_1=n_1a_1$ and $B^{-1}_2 = n_2 a_2$ for some upper triangular unipotent matrices $n_1$ and $n_2$ and some diagonal matrices $a_1$ and $a_2$ with all positive diagonal entries.  Hence, \[\Rr(\widetilde{k}_1, \boldsymbol{z}) = \Rr(\widetilde{k}_2, \boldsymbol{z})\iff B_1^{-1} \widetilde{k}_1 = B_2^{-1} \widetilde{k}_2\ \iff n_1 a_1 \widetilde{k}_1  = n_2 a_2 \widetilde{k}_2.\]  The latter expression is in Iwasawa coordinates for $\SL(d-1, \RR)$ and, hence, it follows that $n_1 =n_2$, $a_1 =a_2$, and $\widetilde{k}_1 =\widetilde{k}_2$, which gives the desired result concerning injectivity.  This proves the corollary.
\end{proof}

\subsubsection{Proof of Theorem~\ref{thmSymmetricFactorCholesky}}\label{subsubsecProofCholesksyFactor}   Let us continue to assume in this section (Section~\ref{subsubsecProofCholesksyFactor}) that $d \geq3$ unless we explicitly state that $d=2$ is under consideration.  First, we make a simple observation to be used later.  Let us use the notation $0_{i \times j}$ to denote the $i \times j$ matrix of all zero entries, and recall that the notation $[\cdot]$ is defined in (\ref{defnofSquareBraket}).

\begin{lemm}\label{lemmMNFactoring}
Let $\widehat{M}, \widehat{N}$ be square matrices of the same dimension $i$.  We have that \[\begin{pmatrix}  [\widehat{M}] \widehat{N} & 0_{i\times j} \\ 0_{j\times i}& I_j\end{pmatrix} = \begin{bmatrix} \widehat{M} & 0_{i\times j} \\ 0_{j\times i}& I_j  \end{bmatrix} \begin{pmatrix}\widehat{N} & 0_{i\times j} \\ 0_{j\times i}& I_j  \end{pmatrix}\] 
\end{lemm}

\begin{proof}
 Compute.
\end{proof}

Also, we note that Lemma~\ref{lemmAlmostRotationSymFac} applied to the general theory of positive definite matrices (see~\cite[Theorem~4.2.1]{vanGo} for example) implies that \[C_1:=I_{d-1} + \frac{\widetilde{k}\tensor*[^t]{\boldsymbol{w}}{}\boldsymbol{w}\tensor*[^t]{\widetilde{k}}{}}{c^2}\] is invertible and symmetric positive definite.  Consequently, we are guaranteed that there exists a unique decomposition of the form $U \tensor*[^t]{U}{}$ where $U$ is upper triangular using Cholesky factorization (see~\cite[Theorem~4.2.5]{vanGo} for example).  
 
The proof will be a computation. We will follow the general idea of the outer product Cholesky factorization~\cite[Section~4.2.5]{vanGo}, except the block decomposition of the matrices will be different to allows us to obtain a $U \tensor*[^t]{U}{}$ decomposition directly.  The details of this \textit{reverse outer product Cholesky factorization} are as follows.

Recall that we are in the case for which $d \geq 3$.  Set $\ell := d-1$ and $\boldsymbol{u}:= (u_1, \cdots, u_\ell):= (\boldsymbol{w}\tensor*[^t]{\widetilde{k}}{})/c$.  Note that we have \[\tensor*[^t]{\boldsymbol{u}}{} \boldsymbol{u} = \begin{pmatrix}  u_1^2 & u_1 u_2 & u_1 u_3 &\cdots & u_1 u_\ell \\ u_2u_1 & u_2^2 & u_2 u_3 &\cdots & u_2 u_\ell \\ &  &\ddots  &  & \\  u_\ell u_1 & u_\ell u_2 & \cdots & u_\ell u_{\ell-1} &u_\ell^2 \end{pmatrix}.\]  Let us distinguish the upper $\ell-1 \times \ell -1$ block and add it to $I_{\ell-1}$:  \[D_1 := \begin{pmatrix}  1+u_1^2 & u_1 u_2 & u_1u_3& \cdots & u_1 u_{\ell-1} \\ u_2u_1 & 1+u_2^2 & u_2 u_3&\cdots & u_2 u_{\ell-1} \\ &  & \ddots&   & \\  u_{\ell-1} u_1 & u_{\ell-1} u_2 & \cdots& u_{\ell-1} u_{\ell-2}& 1+u_{\ell-1}^2 \end{pmatrix}.\]  Let us also distinguish the last element of the last row and add one to it \begin{align}\label{eqnDefBeta1Cholesky}
  \beta_1^2:= 1+u_\ell^2, \end{align} and let us denote the rest of that the last row \begin{align}\label{eqnDefBoldr1Cholesky}
 \boldsymbol{r}_1 :=\begin{pmatrix}  u_1 u_\ell & u_2 u_\ell &\cdots & u_{\ell-1} u_\ell \end{pmatrix}.  \end{align} Consequently, we have that \[C_1 = \begin{pmatrix}  D_1 &  \tensor*[^t]{\boldsymbol{r}}{_1} \\  \boldsymbol{r}_1 &\beta_1^2 \\ \end{pmatrix}.\]  Recalling the definition of the notation $[\cdot]$ in (\ref{defnofSquareBraket}), a direct computation gives us the following decomposition:  \begin{align}\label{eqnCholeskyStep1Factor}
 C_1 = \begin{bmatrix}  I_{\ell-1} &\frac{\tensor*[^t]{ \boldsymbol{r}}{_1}}{\beta_1} \\ \boldsymbol{0} & \beta_1 \end{bmatrix} \begin{pmatrix}  D_1 - \frac{\tensor*[^t]{ \boldsymbol{r}}{_1} \boldsymbol{r}_1} {\beta_1^2} &  \tensor*[^t]{ \boldsymbol{0}}{}  \\  \boldsymbol{0} & 1\end{pmatrix}. \end{align}  This is the first step in the reverse outer product Cholesky factorization.  

If $d=3$, this first step will suffice.  Note that we have \[\beta_1^2 = 1+u_2^2  \quad\quad \boldsymbol{r}_1 = u_1 u_2 \quad\quad D_1 = 1+u_1^2,\] from which it follows that \[D_1 - \frac{\tensor*[^t]{ \boldsymbol{r}}{_1} \boldsymbol{r}_1} {\beta_1^2} =  1+ \frac{u_1^2}{1+u_2^2}=:\beta_2^2.\]  Applying Lemma~\ref{lemmMNFactoring} to (\ref{eqnCholeskyStep1Factor}) yields \[C_1 = \begin{bmatrix}  \beta_2 &  \frac{\boldsymbol{r}_1}{\beta_1} \\ 0 & \beta_1\end{bmatrix},\] from which Theorem~\ref{thmSymmetricFactorCholesky} for $d=3$ immediately follows.   
 
Consequently, from this point forth, we may assume $d \geq 4$ and thus $\ell \geq 3$.  We now use recursion.  Let us distinguish the upper $\ell-1 \times \ell -1$ block of the second matrix from (\ref{eqnCholeskyStep1Factor}): \[C_2 := D_1 - \frac{\tensor*[^t]{ \boldsymbol{r}}{_1} \boldsymbol{r}_1} {\beta_1^2}.\]  Note that \[\tensor*[^t]{ \boldsymbol{r}}{_1} \boldsymbol{r}_1 = \begin{pmatrix} u_1^2 u_\ell^2 & u_1 u_2 u_\ell^2 & u_1 u_3 u_\ell^2 &\cdots& u_1 u_{\ell-1} u_{\ell}^2 \\  u_2 u_1 u_\ell^2 & u_2^2 u_\ell^2 & u_2 u_3 u_\ell^2 &\cdots& u_2 u_{\ell-1} u_{\ell}^2\\&  & \ddots&   & \\   u_{\ell-1} u_1 u_\ell^2 & u_{\ell-1}u_2 u_\ell^2 & \cdots & u_{\ell-1} u_{\ell-2} u_\ell^2 & u_{\ell-1}^2 u_{\ell}^2\\\end{pmatrix}\] and that $D_1$ and $\frac{\tensor*[^t]{ \boldsymbol{r}}{_1} \boldsymbol{r}_1} {\beta_1^2}$ are symmetric.  Consequently, for each row of $C_2$, we only need specify the elements at and to the right of the principle diagonal.

Let $1 \leq i \leq \ell-1$ and $1 \leq j \leq\ell-1$ such that $i<j$.  Computing the elements on the principle diagonal, we have that \[(C_2)_{ii} = 1+ \frac{u_i^2}{1+u_\ell^2}\geq 1.\]  Computing the elements to the right of the principle diagonal, we have that \[(C_2)_{ij} = \frac{u_iu_j}{1+u_\ell^2}.\]  Thus, we have \[C_2 = \begin{pmatrix} 1+\frac{u_1^2}{1+u_\ell^2} & \frac{u_1u_2}{1+u_\ell^2} &  \frac{u_1u_3}{1+u_\ell^2} & \cdots &  \frac{u_1u_{\ell-1}}{1+u_\ell^2} \\   \frac{u_2u_1}{1+u_\ell^2} &  1+\frac{u_2^2}{1+u_\ell^2} & \frac{u_2u_3}{1+u_\ell^2} & \cdots &  \frac{u_2u_{\ell-1}}{1+u_\ell^2} \\&  & \ddots&   & \\  \frac{u_{\ell-1}u_1}{1+u_\ell^2}& \frac{u_{\ell-1}u_2}{1+u_\ell^2}& \cdots & \frac{u_{\ell-1}u_{\ell-2}}{1+u_\ell^2}&1+\frac{u_{\ell-1}^2}{1+u_\ell^2}\end{pmatrix} =: \begin{pmatrix}  D_2 &\tensor*[^t]{ \boldsymbol{r}}{_2} \\ \boldsymbol{r}_2  & \beta_2^2 \end{pmatrix}\] where we have, analogous to the previous step, denoted the upper $\ell-2 \times \ell-2$ block by $D_2$, the last element of the last row by $\beta_2^2$, and the rest of the last row by $\boldsymbol{r}_2$.  The analogous direct computation from the previous step gives us the following:   \begin{align}\label{eqnCholeskyStep2Factor}
 C_2 = \begin{bmatrix}  I_{\ell-2} &\frac{\tensor*[^t]{ \boldsymbol{r}}{_2}}{\beta_2} \\ \boldsymbol{0} & \beta_2 \end{bmatrix} \begin{pmatrix}  D_2 - \frac{\tensor*[^t]{ \boldsymbol{r}}{_2} \boldsymbol{r}_2} {\beta_2^2} &  \tensor*[^t]{ \boldsymbol{0}}{}  \\  \boldsymbol{0} & 1\end{pmatrix}. \end{align}

If $d=4$, two steps will suffice.  Note that we have \[\beta_1^2 = 1+u_3^2\quad \quad \beta_2^2 = 1+\frac{u_{2}^2}{1+u_3^2} \quad\quad \boldsymbol{r}_1 = (u_1 u_3, u_2u_3) \quad\quad \boldsymbol{r}_2 = \frac{u_2 u_1}{1+u_3^2} \quad\quad D_2 = 1+\frac{u_1^2}{1+u_3^2},\] from which it follows that \[D_2 - \frac{\tensor*[^t]{ \boldsymbol{r}}{_2} \boldsymbol{r}_2} {\beta_2^2} =  1+ \frac{u_1^2}{1+u_2^2+u_3^2}=:\beta_3^2.\]  (Note that this computation is direct and is similar to the computation for $(C_{k+1})_{11}$ in Lemma~\ref{lemmRecursiveCkBlocks} below.)  Applying Lemma~\ref{lemmMNFactoring} to (\ref{eqnCholeskyStep1Factor},~\ref{eqnCholeskyStep2Factor}) yields \begin{align*}
C_1 &= \begin{bmatrix}  1 & 0 & \frac{u_1 u_3} {\sqrt{1 + u_3^2}}\\ 0 &1 &\frac{u_2 u_3} {\sqrt{1 + u_3^2}}\\ 0 & 0 & \beta_1 \end{bmatrix} \begin{bmatrix}  1 &\frac{u_1 u_2}{\sqrt{\left(1 + u_3^2\right)\left(1 + u_2^2+u_3^2\right)}} & 0\\ 0 & \beta_2& 0 \\ 0 & 0 & 1 \end{bmatrix}  \begin{bmatrix}  \beta_3 & 0 & 0 \\ 0 & 1 & 0 \\ 0 & 0 & 1 \end{bmatrix} \\ &=  \begin{bmatrix}  \beta_3 & \frac{u_1 u_2}{\sqrt{\left(1 + u_3^2\right)\left(1 + u_2^2+u_3^2\right)}} & \frac{u_1 u_3} {\sqrt{1 + u_3^2}}\\ 0 &\beta_2 &\frac{u_2 u_3} {\sqrt{1 + u_3^2}}\\ 0 & 0 & \beta_1 \end{bmatrix},\end{align*} from which Theorem~\ref{thmSymmetricFactorCholesky} for $d=4$ immediately follows.   

Consequently, from this point forth, we may assume $d \geq 5$ and thus $\ell \geq 4$.  Setting $C_3$ to be the upper $\ell-2 \times \ell -2$ block of the second matrix \[C_3 := D_2 - \frac{\tensor*[^t]{ \boldsymbol{r}}{_2} \boldsymbol{r}_2} {\beta_2^2},\] we recursively repeat the decomposition in step $2$ until we arrive at $C_\ell$.  We now show that this recursion is valid and that $C_\ell$ is a number.

\begin{lemm}\label{lemmRecursiveCkBlocks} Let $2 \leq k \leq \ell-1$.  Let $1 \leq i \leq \ell +1 -k$ and $1 \leq j \leq \ell +1 -k$ such that $i < j$.  We have that $C_k$ is an $\ell + 1 -k \times \ell +1 -k$ symmetric matrix with entries \begin{align}\label{eqnC2ii}
 (C_k)_{ii} = 1+ \frac{u_i^2}{1+u_{\ell +2 -k}^2+ \cdots + u_\ell^2}\geq 1 \quad \textrm{ and } \quad (C_k)_{ij} = \frac{u_iu_j}{1+u_{\ell +2 -k}^2+ \cdots + u_\ell^2}. \end{align}  Moreover, we have that \begin{align*}
C_k = \begin{pmatrix} 1+\frac{u_1^2}{1+u_{\ell +2 -k}^2+ \cdots + u_\ell^2} & \frac{u_1u_2}{1+u_{\ell +2 -k}^2+ \cdots + u_\ell^2} &  \frac{u_1u_3}{1+u_{\ell +2 -k}^2+ \cdots + u_\ell^2} & \cdots &  \frac{u_1u_{\ell +1 -k}}{1+u_{\ell +2 -k}^2+ \cdots + u_\ell^2} \\   \frac{u_2u_1}{1+u_{\ell +2 -k}^2+ \cdots + u_\ell^2} &  1+\frac{u_2^2}{1+u_{\ell +2 -k}^2+ \cdots + u_\ell^2} & \frac{u_2u_3}{1+u_{\ell +2 -k}^2+ \cdots + u_\ell^2} & \cdots &  \frac{u_2u_{\ell +1 -k}}{1+u_{\ell +2 -k}^2+ \cdots + u_\ell^2} \\&  & \ddots&   & \\  \frac{u_{\ell +1 -k}u_1}{1+u_{\ell +2 -k}^2+ \cdots + u_\ell^2}& \frac{u_{\ell +1 -k}u_2}{1+u_{\ell +2 -k}^2+ \cdots + u_\ell^2}& \cdots & \frac{u_{\ell +1 -k}u_{\ell-k}}{1+u_{\ell +2 -k}^2+ \cdots + u_\ell^2}&1+\frac{u_{\ell+1-k}^2}{1+u_{\ell +2 -k}^2+ \cdots + u_\ell^2}\end{pmatrix} \\ =: \begin{pmatrix}  D_k &\tensor*[^t]{ \boldsymbol{r}}{_k} \\ \boldsymbol{r}_k  & \beta_k^2 \end{pmatrix} \end{align*} where $D_k$ is an $\ell -k \times \ell -k$ matrix and $\beta_k^2$ is a real number.
 
\end{lemm}  
\begin{proof}
The proof is by induction on $k$.  We have already shown the case $k=2$.  Assume that the case $k$ holds.  We have that the (recursive) definition of $C_{k+1}$ is \[C_{k+1} := D_k - \frac{\tensor*[^t]{ \boldsymbol{r}}{_k} \boldsymbol{r}_k} {\beta_k^2}.\]  By the induction hypothesis, $C_k$ and thus $D_k$ are symmetric, which further implies that $C_{k+1}$ is symmetric.  The size of $C_{k+1}$ is the same as that of $D_k$, namely $\ell-k \times \ell-k$.

Let us set \[n:= \ell + 1 -k \quad \textrm{ and } \quad  \alpha:= 1 + u_{\ell+2 -k}^2 + \cdots + u_{\ell}^2 = 1 + u_{n+1}^2 + \cdots + u_{\ell}^2\] for convenience of exposition.  This notation and the induction hypothesis gives us that \[\alpha^2 \beta_k^2 = \alpha (\alpha + u_n^2)\] and that \[\frac{\tensor*[^t]{ \boldsymbol{r}}{_k} \boldsymbol{r}_k} {\beta_k^2} = \frac 1 {\alpha (\alpha + u_n^2)} \begin{pmatrix} u_1^2 u_n^2 & u_1 u_2 u_n^2 & u_1 u_3 u_n^2 &\cdots& u_1 u_{n-1} u_{n}^2 \\  u_2 u_1 u_n^2 & u_2^2 u_n^2 & u_2 u_3 u_n^2 &\cdots& u_2 u_{n-1} u_{n}^2\\&  & \ddots&   & \\   u_{n-1} u_1 u_n^2 & u_{n-1}u_2 u_n^2 & \cdots & u_{n-1} u_{n-2} u_n^2 & u_{n-1}^2 u_{n}^2\\\end{pmatrix}.\]

Let $1 \leq i \leq n-1$ and $1 \leq j \leq n-1$ such that $i <j$.  (Recalling that we are now in the case that $d\geq5$ and thus $\ell \geq 4$, we have that $n-1\geq2$.)  We have that \[\left(\frac{\tensor*[^t]{ \boldsymbol{r}}{_k} \boldsymbol{r}_k} {\beta_k^2} \right)_{ii} = \frac{u_i^2 u_n^2}{\alpha (\alpha + u_n^2)} \quad \textrm{ and } \quad \left(\frac{\tensor*[^t]{ \boldsymbol{r}}{_k} \boldsymbol{r}_k} {\beta_k^2} \right)_{ij} = \frac{u_i u_j u_n^2}{\alpha (\alpha + u_n^2)}.\]  The induction hypothesis further gives \[\left(D_k\right)_{ii} = 1 + \frac{u_i^2}{\alpha} \quad \textrm{ and } \quad \left(D_k\right)_{ij} = \frac{u_i u_j}{\alpha}.\]  Consequently, we have that \[\left(C_{k+1}\right)_{ii} = 1 + \frac{u_i^2}{1 + u_{\ell+1 -k}^2 + \cdots + u_{\ell}^2 } \quad \textrm{ and } \quad \left(C_{k+1}\right)_{ij} = \frac{u_i u_j}{1 + u_{\ell+1 -k}^2 + \cdots + u_{\ell}^2 },\] which yields the desired result.
\end{proof}

\begin{lemm}\label{lemmCell}
We have that \[\beta_\ell^2:=C_\ell = 1 + \frac{u_1^2}{1 + u_{2}^2 + \cdots + u_{\ell}^2 }.\]
\end{lemm}
\begin{proof}

We note that the proof of Lemma~\ref{lemmRecursiveCkBlocks} (namely, the induction hypothesis) yielding \[\left(C_{k+1}\right)_{ii} = 1 + \frac{u_i^2}{1 + u_{\ell+1 -k}^2 + \cdots + u_{\ell}^2 }\] for $1 \leq i \leq n-1$ is still valid when $k= \ell-1$.  This yields the desired result.
\end{proof}

\begin{rema}  Lemma~\ref{lemmCell} also holds for $\ell=2$ and $\ell=3$, as we have shown when we considered the cases $d=3$ and $d=4$, respectively.
\end{rema}

\begin{prop}\label{propSymmetricDecom1}
For $1 \leq k \leq \ell -1$, we have that  \[C_k = \begin{bmatrix}  I_{\ell-k} &\frac{\tensor*[^t]{ \boldsymbol{r}}{_k}}{\beta_k} \\ \boldsymbol{0} & \beta_k \end{bmatrix} \begin{pmatrix}  D_k - \frac{\tensor*[^t]{ \boldsymbol{r}}{_k} \boldsymbol{r}_k} {\beta_k^2} &  \tensor*[^t]{ \boldsymbol{0}}{}  \\  \boldsymbol{0} & 1\end{pmatrix}.\] 
\end{prop}
\begin{proof}
 Applying Lemma~\ref{lemmRecursiveCkBlocks} to the analogous computation as for $C_1$ gives the desired result.
\end{proof}

We now finish the proof of Theorem~\ref{thmSymmetricFactorCholesky}.  We will use zero vectors of different lengths.  For clarity, we use the notation $\boldsymbol{0}_i$ to denote the $i$-vector with all zero entries.  Also, note that the vector $\boldsymbol{r}_k$ is an $\ell -k$-vector and that $C_k$ is an $\ell+1-k \times \ell+1-k$ matrix.

\begin{lemm}\label{lemmSquareRoundComm}
We have that\footnote{Note that the matrix in the square brackets is upper triangular and the matrix in the round brackets is symmetric.} \[C_1 = \begin{bmatrix}  I_{\ell-k} & \frac{\tensor*[^t]{ \boldsymbol{r}}{_k}}{\beta_k} & \frac{\tensor*[^t]{ \boldsymbol{r}}{_{k-1}}}{\beta_{k-1}}& \cdots &\frac{\tensor*[^t]{ \boldsymbol{r}}{_1}}{\beta_1} \\ \boldsymbol{0}_{\ell-k} & \beta_k & & \\ \boldsymbol{0}_{\ell-(k-1)} & & \beta_{k-1} & & \\ && &\ddots& \\ \boldsymbol{0}_{\ell-1} &&&&\beta_1 \end{bmatrix}\begin{pmatrix}  C_{k+1} & \tensor*[^t]{\boldsymbol{0}}{_{\ell-k}} & \tensor*[^t]{\boldsymbol{0}}{_{\ell-(k-1)}} & \cdots  & \tensor*[^t]{\boldsymbol{0}}{_{\ell-1}}\\  \boldsymbol{0}_{\ell-k} & 1 & & \\ \boldsymbol{0}_{\ell-(k-1)} & & 1 & & \\ && &\ddots& \\ \boldsymbol{0}_{\ell-1} &&&&1\end{pmatrix}\] for $1 \leq k \leq \ell -1$.  
\end{lemm}

\begin{proof}

 We induct on $k$.  The case $k=1$ is Proposition~\ref{propSymmetricDecom1}.  Assume that the case $k$ holds.  Applying Proposition~\ref{propSymmetricDecom1} with Lemma~\ref{lemmMNFactoring} and directly multiplying yields \begin{align*}
 \begin{bmatrix}  I_{\ell-k} & \frac{\tensor*[^t]{ \boldsymbol{r}}{_k}}{\beta_k} & \frac{\tensor*[^t]{ \boldsymbol{r}}{_{k-1}}}{\beta_{k-1}}& \cdots &\frac{\tensor*[^t]{ \boldsymbol{r}}{_1}}{\beta_1} \\ \boldsymbol{0}_{\ell-k} & \beta_k & & \\ \boldsymbol{0}_{\ell-(k-1)} & & \beta_{k-1} & & \\ && &\ddots& \\ \boldsymbol{0}_{\ell-1} &&&&\beta_1 \end{bmatrix}\begin{bmatrix}   \widetilde{C}_{k+1} & \tensor*[^t]{\boldsymbol{0}}{_{\ell-k}} & \tensor*[^t]{\boldsymbol{0}}{_{\ell-(k-1)}} & \cdots  & \tensor*[^t]{\boldsymbol{0}}{_{\ell-1}}\\  \boldsymbol{0}_{\ell-k} & 1 & & \\ \boldsymbol{0}_{\ell-(k-1)} & & 1 & & \\ && &\ddots& \\ \boldsymbol{0}_{\ell-1} &&&&1 \end{bmatrix}\\
 \begin{pmatrix}  C_{k+2} & \tensor*[^t]{\boldsymbol{0}}{_{\ell-(k+1)}} & \tensor*[^t]{\boldsymbol{0}}{_{\ell-k}} & \cdots  & \tensor*[^t]{\boldsymbol{0}}{_{\ell-1}}\\  \boldsymbol{0}_{\ell-(k+1)} & 1 & & \\ \boldsymbol{0}_{\ell-k} & & 1 & & \\ && &\ddots& \\ \boldsymbol{0}_{\ell-1} &&&&1\end{pmatrix} \end{align*} where $\widetilde{C}_{k+1}$ is the $\ell-k \times \ell -k$ block \[\widetilde{C}_{k+1}:=  \begin{pmatrix}  I_{\ell-(k+1)} &\frac{\tensor*[^t]{ \boldsymbol{r}}{_{(k+1)}}}{\beta_{(k+1)}} \\ \boldsymbol{0}_{\ell-(k+1)} & \beta_{(k+1)} \end{pmatrix}.\]  Multiplying the matrices in the square brackets yields the desired result.
\end{proof}

\begin{prop}
We have that \[C_1 = \begin{bmatrix}  \beta_\ell& \frac{\tensor*[^t]{ \boldsymbol{r}}{_{\ell-1}}}{\beta_{\ell-1}} & \frac{\tensor*[^t]{ \boldsymbol{r}}{_{\ell-2}}}{\beta_{\ell-2}}& \cdots &\frac{\tensor*[^t]{ \boldsymbol{r}}{_1}}{\beta_1} \\ \boldsymbol{0}_{1} & \beta_{\ell-1} & & \\ \boldsymbol{0}_{2} & & \beta_{\ell-2} & & \\ && &\ddots& \\ \boldsymbol{0}_{\ell-1} &&&&\beta_1 \end{bmatrix}.\]
\end{prop}

\begin{proof}
 Apply Lemma~\ref{lemmSquareRoundComm} with $k=\ell-1$ and, to the result, apply Lemmas~\ref{lemmCell} and~\ref{lemmMNFactoring}.  This yields the desired result.
\end{proof}

Consequently, using the proposition, we have proved Theorem~\ref{thmSymmetricFactorCholesky}.  Note that the values of $\beta_k$ and $\boldsymbol{s}_k$ come from (\ref{eqnDefBeta1Cholesky},~\ref{eqnDefBoldr1Cholesky}) and Lemmas~\ref{lemmRecursiveCkBlocks} and~\ref{lemmCell}.

\subsubsection{Proof of Theorem~\ref{thmThickenSubsetSectionSpherical} for $L = I_d$}\label{secProofthmThickenSubsetSectionSpherical}  We adapt the proof of Theorem~\ref{thmNAKResultRankOneFlow} for the case $L = I_d$ to give a proof of Theorem~\ref{thmThickenSubsetSectionSpherical} for the case $L = I_d$.  The key equation to use is (\ref{eqnDecompSOintoN}), which we now restate in a more convenient form:

\begin{align}\label{eqnDecompSOintoNVerII}
\Phi^{-s} E(\boldsymbol{z})^{-1} n_+(-\boldsymbol{z'}) = \Phi^{-s}  \begin{pmatrix}  I_{d-1}& c^{-1}  ( \tensor*[^t]{\boldsymbol{w}}{})\\  0 & 1\end{pmatrix}\begin{pmatrix}  c^{1/(d-1)}(A - \tensor*[^t]{\boldsymbol{w}}{}\boldsymbol{z'}) & \tensor*[^t]{\boldsymbol{0}}{}\\  0 & 1\end{pmatrix} \Phi^{\log c/ (d-1)}
  \end{align}

Define the natural projection map \[\iota:K' \rightarrow \SO(d-1, \RR), \quad  \begin{pmatrix}  k &\tensor*[^t]{\boldsymbol{0}}{}  \\ \boldsymbol{0} & 1 \end{pmatrix} \mapsto k\]  The analog of Lemma~\ref{lemmMMyCoordinatesForSE} is

\begin{lemm}\label{lemmMMyCoordinatesForSECuspNeigh}  We have that \begin{align*}
\left\{\widetilde{\Ss}_T(\boldsymbol{z}) E^{-1}(\boldsymbol{z}): \boldsymbol{z} \in \Dd\right\} &= \left\{ \widetilde{\Ss}\left(\frac{T}{c^2 T_-}\alpha_{1}, \frac{T}{c^2 T_-}\gamma_{d-1}, 1 \right)n_+(\boldsymbol{z}'):  \boldsymbol{z} \in \Dd\right\} \end{align*} if $d = 2$ and \begin{align*} &\left\{\widetilde{\Ss}_T(\boldsymbol{z}) E^{-1}(\boldsymbol{z}): \boldsymbol{z} \in \Dd\right\} = \left\{ \bigsqcup_{ \widehat{k} \in \widetilde{K}_U(\boldsymbol{z})}\widetilde{\Ss}\left(\left(\frac{\beta_{d-1}}{\beta_{d-2}}\alpha_1, \cdots, \frac{\beta_{2}}{\beta_{1}}\alpha_{d-2},\frac{\beta_1T}{c T_-^{d/2(d-1)}}\alpha_{d-1}\right), \right.\right. \\ &\quad\quad\left.\left. \left(\frac{\beta_{d-1}}{\beta_{d-2}}\gamma_1, \cdots, \frac{\beta_{2}}{\beta_{1}}\gamma_{d-2},\frac{\beta_1T}{c T_-^{d/2(d-1)}}\gamma_{d-1}\right), \Rr\left( \iota\left(\widehat{k}E(\boldsymbol{z})\right), \boldsymbol{z}\right) \right)n_+(\boldsymbol{z}'): \boldsymbol{z} \in \Dd\right\}   \end{align*}  if $d \geq 3$.  Here, $\beta_\ell$ is as in Theorem~\ref{thmSymmetricFactorCholesky}.
 
\end{lemm}

\begin{rema}
Note that application of the flow $\Phi^{-s}$ does not affect the $\widetilde{K}$ coordinate.

\end{rema}

\begin{proof}

Applying (\ref{eqnDecompSOintoNVerII},~\ref{eqnFundDomCoord},~\ref{eqnRelCTAndC},~\ref{eqnPullingBackOfSectionalNeigh2}), we have that an arbitrary element $g \in \widetilde{\Ss}_{T}\Ee$ is of the form \begin{align*} g =  \begin{pmatrix}  1&x_{ij} &\\ &\ddots& \\& & 1\end{pmatrix}&y^{-1/2} \begin{pmatrix}y_{d-1}y_{d-2} \cdots y_{1}&&&& \\ &\ddots&&&&\\ && y_{d-1} y_{d-2}&&\\ &&&y_{d-1}&\\ &&&& 1   \end{pmatrix}\begin{pmatrix} k  &  \tensor*[^t]{\boldsymbol{0}}{} \\ \boldsymbol{0} & 1\end{pmatrix} \\\nonumber & \begin{pmatrix}  I_{d-1}& c^{-1}  ( \tensor*[^t]{\boldsymbol{w}}{})\\  0 & 1\end{pmatrix}\begin{pmatrix}  c^{1/(d-1)}(A - \tensor*[^t]{\boldsymbol{w}}{}\boldsymbol{z'}) & \tensor*[^t]{\boldsymbol{0}}{}\\  0 & 1\end{pmatrix} \Phi^{\log c/ (d-1)} n_+(\boldsymbol{z'})  \end{align*} where $\boldsymbol{z} \in \Dd$, $\gamma_i \geq y_i \geq \alpha_i$ for $i = 1, \cdots, d-2$, $T_-^{-d/2(d-1)}T \gamma_{d-1} \geq y_{d-1} \geq T_-^{-d/2(d-1)}T\alpha_{d-1}$, and $k \in \widetilde{K}_U(\boldsymbol{z})E(\boldsymbol{z})$.  (The only constraint on $x_{ij}$ is that $g$ lies in $\Ff'_d$.)  Since \[\begin{pmatrix}  k & \tensor*[^t]{\boldsymbol{0}}{}\\ \boldsymbol{0} & 1 \end{pmatrix} \begin{pmatrix} I_{d-1} & c^{-1} (\tensor*[^t]{\boldsymbol{w}}{}) \\    \boldsymbol{0} & 1   \end{pmatrix} = \begin{pmatrix} I_{d-1} & k c^{-1} (\tensor*[^t]{\boldsymbol{w}}{}) \\    \boldsymbol{0} & 1   \end{pmatrix} \begin{pmatrix}  k & \tensor*[^t]{\boldsymbol{0}}{}\\ \boldsymbol{0} & 1 \end{pmatrix}\] and the diagonal element \begin{align}\label{eqnDiagonalElementConj}
 y^{-1/2} \begin{pmatrix}y_{d-1}y_{d-2} \cdots y_{1}&&&& \\ &\ddots&&&&\\ && y_{d-1} y_{d-2}&&\\ &&&y_{d-1}&\\ &&&& 1   \end{pmatrix} \end{align} acts by conjugation on the subgroup of upper triangular unipotent matrices, we have that  \begin{align*} g =  \begin{pmatrix}  1&\widetilde{x}_{ij} &\\ &\ddots& \\& & 1\end{pmatrix}&y^{-1/2} \begin{pmatrix}y_{d-1}y_{d-2} \cdots y_{1}&&&& \\ &\ddots&&&&\\ && y_{d-1} y_{d-2}&&\\ &&&y_{d-1}&\\ &&&& 1   \end{pmatrix}\begin{pmatrix} k  &  \tensor*[^t]{\boldsymbol{0}}{} \\ \boldsymbol{0} & 1\end{pmatrix} \\\nonumber & \begin{pmatrix}  c^{1/(d-1)}(A - \tensor*[^t]{\boldsymbol{w}}{}\boldsymbol{z'}) & \tensor*[^t]{\boldsymbol{0}}{}\\  0 & 1\end{pmatrix} \Phi^{\log c/ (d-1)} n_+(\boldsymbol{z'}).\end{align*}

Consider the case $d=2$ first.  Equation (\ref{eqnDecompSOintoNVerII}) implies that $ c^{1/(d-1)}\left(A - \tensor*[^t]{\boldsymbol{w}}{}\boldsymbol{z'}\right) = 1$, from which the desired result follows by an application of (\ref{eqnRelCTAndC}).

Now consider the case $d \geq 3$.  Multiplying matrices, we have \begin{align*} g =  \begin{pmatrix}  1&\widetilde{x}_{ij} &\\ &\ddots& \\& & 1\end{pmatrix}&y^{-1/2} \begin{pmatrix}y_{d-1}y_{d-2} \cdots y_{1}&&&& \\ &\ddots&&&&\\ && y_{d-1} y_{d-2}&&\\ &&&y_{d-1}&\\ &&&& 1   \end{pmatrix} \\\nonumber & \begin{pmatrix}  B & \tensor*[^t]{\boldsymbol{0}}{}\\  0 & c\end{pmatrix} \begin{pmatrix} B^{-1} k \left(A - \tensor*[^t]{\boldsymbol{w}}{}\boldsymbol{z'}\right) & \tensor*[^t]{ \boldsymbol{0}}{} \\  \boldsymbol{0} & 1 \end{pmatrix}  n_+(\boldsymbol{z'})\end{align*} where $B$ is as in Theorem~\ref{thmSymmetricFactorCholesky}.  Also we have \[\begin{pmatrix}  B & \tensor*[^t]{\boldsymbol{0}}{}\\  0 & c\end{pmatrix}=  n(B)  \begin{pmatrix}\beta_{d-1}&&&& \\ &\ddots&&&&\\ && \beta_{2}&&\\ &&&\beta_{1}&\\ &&&& c   \end{pmatrix}\] where $n(B)$ is an upper triangular unipotent matrix.  Using the conjugation action of the diagonal element from (\ref{eqnDiagonalElementConj}) yields the desired result.

Note that the disjointness of the union follows from the injectivity assertion in Corollary~\ref{coroSymmetricFactorCholesky}.

\end{proof}
 
For convenience of notation, let us denote some of the sets in the lemma by \begin{align}\label{eqn:SphericalSectionCoords}
\widehat{\Ss}_T(\boldsymbol{z}):= \begin{cases} \widetilde{\Ss}\left(\frac{T}{c^2 T_-}\alpha_{1}, \frac{T}{c^2 T_-}\gamma_{d-1}, 1 \right) &\textrm{ if } d=2, \\ \bigsqcup_{ \widehat{k} \in \widetilde{K}_U(\boldsymbol{z})}\widetilde{\Ss}\left(\left(\frac{\beta_{d-1}}{\beta_{d-2}}\alpha_1, \cdots, \frac{\beta_{2}}{\beta_{1}}\alpha_{d-2},\frac{\beta_1T}{c T_-^{d/2(d-1)}}\alpha_{d-1}\right), \right. \\ \quad\quad\left. \left(\frac{\beta_{d-1}}{\beta_{d-2}}\gamma_1, \cdots, \frac{\beta_{2}}{\beta_{1}}\gamma_{d-2},\frac{\beta_1T}{c T_-^{d/2(d-1)}}\gamma_{d-1}\right), \Rr\left( \iota \left(\widehat{k}E(\boldsymbol{z})\right), \boldsymbol{z}\right) \right)   &\textrm{ if } d \geq 3.\end{cases}  \end{align}

\begin{lemm}\label{lemmLowerHeightWidehatST}
 For any $\boldsymbol{z} \in \Dd$, the lower height of $\widehat{\Ss}_T(\boldsymbol{z})$ is \[\frac{T^{2(d-1)/d}} {c^{2(d^2 -2d +2)/d(d-1)}}.\]
\end{lemm}
\begin{proof}
The result for $d=2$ is a direct computation.  For $d \geq3$, a direct computation yields that the lower height is \[\left(\frac{T^{2(d-1)}}{c^{2(d-1)}} \beta_1^2 \cdots \beta_{d-1}^2\right)^{1/d}.\]  Applying Theorem~\ref{thmSymmetricFactorCholesky}, (\ref{eqnSymmetricFactor}), and Lemma~\ref{lemmAlmostRotationSymFac} yields the desired result. 
\end{proof}

\begin{coro}\label{coroLowerHeightWidehatST}
We have that \[\bigcup_{\boldsymbol{z} \in \Dd}\widehat{\Ss}_T(\boldsymbol{z}) \subset \Ss_T \quad \textrm{ and } \quad \bigcup_{\boldsymbol{z} \in \Dd}\widehat{\Ss}_{T_0}(\boldsymbol{z}) \subset \Ss_{T_0}.\]
\end{coro}
\begin{proof}
Apply the lemma with the fact that $0 < c \leq 1$.
\end{proof}

We now follow the proof of Theorem~\ref{thmNAKResultRankOneFlow} for the case $L = I_d$, but replacing \begin{align}\label{eqnUniReplaceSToWidtildeS}
 \Ss_{Tc^{-d/(d-1)}} \textrm{ with } \widehat{\Ss}_T(\boldsymbol{z}) \textrm{ and } \Ss_{T_0c^{-d/(d-1)}} \textrm{ with } \widehat{\Ss}_{T_0}(\boldsymbol{z}). \end{align}  We have the analogs of \[\widehat{f}_{T, \Dd}(\boldsymbol{r}, n_-(\boldsymbol{r}) \Phi^t) \textrm{ and } \widehat{f}_{T, \Dd, \alpha}(\boldsymbol{r},  n_-(\boldsymbol{r}) \Phi^t).\]  Letting \[\widetilde{u} \in \bigcup_{\boldsymbol{z} \in \Dd}\widehat{\Ss}_T(\boldsymbol{z}) \textrm{ and } {u} \in \bigcup_{\boldsymbol{z} \in \Dd}\widehat{\Ss}_{T_0}(\boldsymbol{z}),\] we have the analog of (\ref{eqnSphereHoroEquivalence2}) also.

Also, the statement of Lemma~\ref{lemmUnstableVolumesNAKCase} is unchanged.  Its proof is changed as follows.  First, we also make the changes in (\ref{eqnUniReplaceSToWidtildeS}).  Next, the definitions of $\I(\boldsymbol{r}, t, T)$ and $\J(\boldsymbol{r}, t, T)$ are slightly altered as follows :   \begin{align*}
\I(\boldsymbol{r}, t, T) &:= \Ga\tensor*[^t]{\left(n_-(\boldsymbol{r}) \Phi^{t}\right)}{}^{-1}  \cap  \bigcup_{\boldsymbol{z} \in \Dd}\widehat{\Ss}_T(\boldsymbol{z})  \\ \J\left(\boldsymbol{r}, t ,T\right) &:=\Ga \tensor*[^t]{\left(n_-(\boldsymbol{r}) \Phi^{t - \frac 1 d \log(T/T_0)}\right)}{}^{-1}  \cap \bigcup_{\boldsymbol{z} \in \Dd}\widehat{\Ss}_{T_0}(\boldsymbol{z}).\end{align*}  The reason that $ \J\left(\boldsymbol{r}, t ,T\right)$ is nonempty for all $t$ large enough is exactly the same as in the proof of Lemma~\ref{lemmUnstableVolumes2}.  The mapping $\varphi$ is the same and a direct computation using (\ref{eqnRelCTAndC}) shows that it is a bijection.  The proof up to and including (\ref{eqnCompareExpandingHaarVolNAKCase}) is the same.  To prove the analog of (\ref{eqnRatioHaarMeasuresNAKCase}), we use the same proof along with Corollary~\ref{coroLowerHeightWidehatST}.  The rest of the proof of the lemma is the same. 

The proof of the analog of (\ref{eqnMainDerivUsingFrobNAKCase}) is also the same, except we replace Theorem~\ref{thmHaarMeaSE} with Theorem~\ref{thmHaarMeaWidetildeSE}.  The remaining proof to obtain the analog of (\ref{eqnNAKMainResultConclusion}), which is also the desired result, is the same.  This concludes the proof of Theorem~\ref{thmThickenSubsetSectionSpherical} for the case $L = I_d$.

\subsection{Thickenings by subsets of $\SO(d, \RR)$}  In this section, we generalize Theorem~\ref{thmThickenSubsetSectionSpherical} by removing the restriction (\ref{eqnRestrictionsOnD}).  As discussed in Section~\ref{subsubsecSphericalCase}, there are two different ways of removing this restriction, leading to two slightly different generalizations (Theorems~\ref{thmTwoChartGeneralSpherical} and~\ref{thmThickenSubsetSectionSphericalGeneral}).

\subsubsection{Two charts with $c(\boldsymbol{z})$ bounded away from zero}  The first generalization is in the context of Theorem~\ref{thmSphericalThickeningConjCharts}, namely for the induced diffeomorphism $\psi$ (i.e., the mapping from (\ref{eqnInducedDiffeoUsinged})) to still satisfy the restriction (\ref{eqnRestrictionsOnD}), but from which we construct two charts $\widetilde{\Ee}$ for $S^{d-1}$ (see (\ref{subsubsecSphericalCase})), one for one hemisphere and the other for the other.  For this context of the two charts, we generalize the notion of an induced subset of $\SO(d, \RR)$ (see Section~\ref{subsec:SubsetsOfRotationsInduced}).  First note that, in order for a subset $U \subset \SO(d, \RR)$ to satisfy the constraint given by Theorem~\ref{thmSphericalThickeningConjCharts}, we require \begin{align}\label{eqnConstUTwoChartsCase}
 U = \begin{cases} - U  & \textrm{ if } d \textrm{ is even,} \\ U \begin{pmatrix}  \widetilde{I}_{d-1} &  \tensor*[^t]{\boldsymbol{0}}{} \\ \boldsymbol{0} & -1  \end{pmatrix}  & \text{ if } d \textrm{ is odd,}\end{cases}\end{align}  where $\widetilde{I}_{d-1}$ is defined in (\ref{eqnDefTildeIdMin1}).  A subset $U \subset \SO(d, \RR)$ is \textit{induced by $\widetilde{\Ee}$} if it satisfies (\ref{eqnConstUTwoChartsCase}) and if \[\Pp(U) =  \psi(\Dd) \sqcup \left(- \psi(\Dd)\right).\]  For any $U \subset \SO(d, \RR)$ induced by $\widetilde{\Ee}$, define the \textit{slice of $U$ through $E^{-1}(\boldsymbol{z})$} as the following:\[\widetilde{K}^{+}_U({\boldsymbol{z}}):=K'E^{-1}(\boldsymbol{z}) \cap U\] and \textit{slice of $U$ through $\begin{cases}  -E^{-1}(\boldsymbol{z}) & \textrm{ if } d \textrm{ is even,} \\  \begin{pmatrix}  \widetilde{I}_{d-1} &  \tensor*[^t]{\boldsymbol{0}}{} \\ \boldsymbol{0} & -1  \end{pmatrix}E(\boldsymbol{z})^{-1} & \text{ if } d \textrm{ is odd}\end{cases}$} as the following:\[\widetilde{K}^{-}_U({\boldsymbol{z}}):=\begin{cases} K'\left(-E^{-1}(\boldsymbol{z}) \right)\cap U & \text{ if } d \textrm{ is even,} \\   K'\begin{pmatrix}  \widetilde{I}_{d-1} &  \tensor*[^t]{\boldsymbol{0}}{} \\ \boldsymbol{0} & -1  \end{pmatrix}E(\boldsymbol{z})^{-1} \cap U & \text{ if } d \textrm{ is odd.}\end{cases}\]

We also have the analog of Lemma~\ref{lemmDisjKSlices}.
\begin{lemm} For any subset $U$ induced by $\widetilde{\Ee}$, we have that \[\bigsqcup_{\boldsymbol{z} \in \Dd} \left(\widetilde{K}^+_U({\boldsymbol{z}})\sqcup  \widetilde{K}^-_U({\boldsymbol{z}}) \right)= U\] 
\end{lemm}
\begin{proof}
 The disjointness follows by Lemma~\ref{lemmInducedDiffeoDeterByvAndc} and the fact that the charts are in different hemispheres and thus correspond to disjoint sets.  Let $g \in U$.  Then $\Pp(g) = \psi(\boldsymbol{z})$ or $- \psi(\boldsymbol{z})$ for some $\boldsymbol{z} \in \Dd$ and the rest of the proof is similar to that of Lemma~\ref{lemmDisjKSlices}.

\end{proof}

\begin{theo}\label{thmTwoChartGeneralSpherical} Let $0<\eta <1$ be a fixed constant, $L \in G$, $T_+\geq T_- > \h_0^{2(d-1)/d}$, $\Aa \subset \Aa_L$ be a bounded subset with measure zero boundary with respect to $\wrt {\boldsymbol{x}}$, and $U \subset \SO(d, \RR)$ be a subset measurable with respect to $\wrt k$ and induced by $\widetilde{\Ee}$.  Let \begin{align*}
\widetilde{\Ss}^+(\boldsymbol{z})&:=\widetilde{\Ss}^+_{T_-, T_+}(\boldsymbol{z}):=\widetilde{\Ss}\left(\boldsymbol{\alpha}, \boldsymbol{\gamma}, \widetilde{K}^+_{U}(\boldsymbol{z})E(\boldsymbol{z})\right) \\ \widetilde{\Ss}^-(\boldsymbol{z})&:=\widetilde{\Ss}^-_{T_-, T_+}(\boldsymbol{z}):=\begin{cases}  \widetilde{\Ss}\left(\boldsymbol{\alpha}, \boldsymbol{\gamma}, \widetilde{K}^-_{U}(\boldsymbol{z})\left(-E(\boldsymbol{z})\right)\right) & \textrm{ if } d \textrm{ is even,} \\  \widetilde{\Ss}\left(\boldsymbol{\alpha}, \boldsymbol{\gamma}, \widetilde{K}^-_{U}(\boldsymbol{z}) E(\boldsymbol{z})\begin{pmatrix}  \widetilde{I}_{d-1} &  \tensor*[^t]{\boldsymbol{0}}{} \\ \boldsymbol{0} & -1  \end{pmatrix}\right) & \textrm{ if } d \textrm{ is odd,} \end{cases}  \end{align*} and $\widetilde{\Ss}^+_T$, $\widetilde{\Ss}^-_T$ be defined using $\widetilde{\Ss}:= \widetilde{\Ss}^+(\boldsymbol{z})$,  $\widetilde{\Ss}:= \widetilde{\Ss}^-(\boldsymbol{z})$ in (\ref{eqnRelationLowerHeightToFlow} , \ref{eqnPullingBackOfSectionalNeigh2}), respectively.  Let \[\widetilde{\Ss}^\pm (\boldsymbol{z}) \widetilde{\Ee}(\boldsymbol{z}):=\begin{cases}  \left\{\widetilde{S}^+_T(\boldsymbol{z}) E^{-1}(\boldsymbol{z}): \boldsymbol{z}\in \Dd\right\} \sqcup \left\{\widetilde{S}^-_T(\boldsymbol{z}) \left(-E^{-1}(\boldsymbol{z})\right): \boldsymbol{z}\in \Dd \right\} & \textrm{ if } d \textrm{ is even,} \\   \left\{\widetilde{S}^+_T(\boldsymbol{z}) E^{-1}(\boldsymbol{z}): \boldsymbol{z}\in \Dd\right\} \sqcup \left\{\widetilde{S}^-_T(\boldsymbol{z}) \left(\begin{pmatrix}  \widetilde{I}_{d-1} &  \tensor*[^t]{\boldsymbol{0}}{} \\ \boldsymbol{0} & -1  \end{pmatrix}E^{-1}(\boldsymbol{z})\right): \boldsymbol{z}\in \Dd \right\} & \textrm{ if } d \textrm{ is odd,}\end{cases}\]Then \begin{align}\label{eqnthmNAKResultRankOneFlowSubsetThickTwoChartTransHoro}
 {T^{d-1}}\int_{\RR^{d-1}} \In_{\Aa \times \left(\widetilde{\Ss}^\pm (\boldsymbol{z}) \widetilde{\Ee}(\boldsymbol{z})\right)}(\boldsymbol{x}, {Ln_-(\boldsymbol{x}) \Phi^t})~\wrt {\boldsymbol{x}} \xrightarrow[]{t \rightarrow \infty}  T_-^{d/2}\mu\left(\widetilde{\Ss}^\pm (\boldsymbol{z}) \widetilde{\Ee}(\boldsymbol{z})\right)\left(\int_{\RR^{d-1}} \In_{\Aa}(\boldsymbol{x})~\wrt{\boldsymbol{x}}\right) \end{align} uniformly for all $T \in [T_-^{d/2(d-1)}, e^{dt \eta}].$

\end{theo}

\begin{proof}  There are two charts whose points are given in pairs by (\ref{eqnTwoChartsUpperLowerHemi}).   Fix a point $\boldsymbol{z} \in \Dd$.  Let $\pf \in \widetilde{S}_T(\boldsymbol{z})$.   
 
Consider $d$ even. Then Theorem~\ref{thmSphericalThickeningConjCharts} gives the points $\pf E^{-1}(\boldsymbol{z})$ and $\pf \left (-E^{-1}(\boldsymbol{z})\right)$ are the same point in $\Gamma \backslash G$.   Since $-I_d \in \Gamma$ commutes with all elements of $G$ and (\ref{eqnConstUTwoChartsCase}) holds, we can apply Theorem~\ref{thmThickenSubsetSectionSpherical} to both charts to obtain the desired result.

Finally consider $d$ odd.  Applying Lemmas~\ref{lemmeqnFundDomCoord} and~\ref{lemmSkewIdConjH} to $\widetilde{S}_T(\boldsymbol{z})$, we have that \begin{align}\label{eqn:Spherical:ConjST}
 \begin{pmatrix}  \widetilde{I}_{d-1} &  \tensor*[^t]{\boldsymbol{0}}{} \\ \boldsymbol{0} & -1  \end{pmatrix}\widetilde{S}_T(\boldsymbol{z}) \begin{pmatrix}  \widetilde{I}_{d-1} &  \tensor*[^t]{\boldsymbol{0}}{} \\ \boldsymbol{0} & -1  \end{pmatrix} = \widetilde{S}_T(\boldsymbol{z}). \end{align} Now (\ref{eqn:Spherical:ConjST}) and Theorem~\ref{thmSphericalThickeningConjCharts} give that there exists a unique point $\qf$ in the embedded submanifold $\widetilde{S}_T(\boldsymbol{z})$ such that  $\pf E^{-1}(\boldsymbol{z})$ and \[\qf \begin{pmatrix}  \widetilde{I}_{d-1} &  \tensor*[^t]{\boldsymbol{0}}{} \\ \boldsymbol{0} & -1  \end{pmatrix}E^{-1}(\boldsymbol{z})\] are the same point in $\Gamma \backslash G$.  Since (\ref{eqn:Spherical:ConjST}, \ref{eqnConstUTwoChartsCase}) hold, we can apply Theorem~\ref{thmThickenSubsetSectionSpherical} to both charts to obtain the desired result.\end{proof}

\begin{rema}  This proof works for both the case $L=I_d$ and the case of generic $L \in G$.  For $L \in G$, we apply Theorem~\ref{thmThickenSubsetSectionSpherical} for this $L$ and note that the proof of Theorem~\ref{thmThickenSubsetSectionSpherical} for generic $L$ is in Section~\ref{subsecUniEquiForTransHoro}.
\end{rema}

\subsubsection{The general case: removing the restriction on $c(\Dd)$}  In this section, we remove the restriction (\ref{eqnRestrictionsOnD}) and, in particular, we have that $-1 \leq c(\Dd) \leq 1$.  We will use one chart.  

Let us first show that, when $c=0$, there is no contribution.  We will show this in general.  Define \[K_0:=\left\{\begin{pmatrix}  A &  \tensor*[^t]{\boldsymbol{w}}{} \\ \boldsymbol{v} & c  \end{pmatrix}  \in \SO(d, \RR): c=0 \right\}.\]  

\begin{lemm}\label{lemm:HorosphereInSphereEdge}  There exists a null set $\Nn$ with respect to $\wrt {\boldsymbol{x}}$ such that, for every $t \in \RR$, we have that \[\{\boldsymbol{x} \in \RR^{d-1} : n_-(\boldsymbol{x}) \Phi^t \in \Ss_1 K_0 \} \subset \Nn.\]  Moreover, $\Nn$ can be chosen to be a countable union of affine hyperplanes in $\RR^{d-1}$ of dimension $d-2$.
 
\end{lemm}

\begin{rema}
Note that $\Nn$ is independent of $t$.

\end{rema}

\begin{proof}
 Let $\boldsymbol{x} \in \RR^{d-1}$ and $t \in \RR$ be elements such that $n_-(\boldsymbol{x}) \Phi^t \in \Ss_1 K_0$.  Then there exists an $\gamma \in \Gamma$ such that \begin{align}\label{eqn:HorosphereInSphereEdgeMain}  \gamma^{-1}\tensor*[]{\begin{pmatrix} B &  \tensor*[^t]{\boldsymbol{b}}{} \\ \boldsymbol{0} & 1\end{pmatrix}}{}&\begin{pmatrix}y^{1/2(d-1)} I_{d-1}&\tensor*[^t]{\boldsymbol{0}}{}  \\ \boldsymbol{0} & y^{-1/2}   \end{pmatrix}  \begin{pmatrix}  A &  \tensor*[^t]{\boldsymbol{w}}{} \\ \boldsymbol{v} & 0  \end{pmatrix} =  n_-(\boldsymbol{x}) \Phi^t.\end{align}  Let \begin{align} \label{eqn:HorosphereInSphereEdgeMain2} (\boldsymbol{p}, q) = (\boldsymbol{0}, 1) \gamma, \end{align} and applying it on the left yields \[(y^{-1/2}\boldsymbol{v}, 0) = (e^{-t}\boldsymbol{p}, e^{(d-1)t}(\boldsymbol{p}\tensor*[^t]{\boldsymbol{x}}{} +q)).\]  In particular we have \begin{align}\label{eqn:HorosphereInSphereEdge} \boldsymbol{p}\tensor*[^t]{\boldsymbol{x}}{} +q=0.
  \end{align}  If $\boldsymbol{p}= \boldsymbol{0}$, then (\ref{eqn:HorosphereInSphereEdge}) implies that $q=0$ and, thus, $\gamma$ has a row with all entries zero, which is not possible.  Thus $\boldsymbol{p} \neq \boldsymbol{0}$, and there exists an entry $p_i \neq 0$.  We now have that $x_i $ is some linear combination of the other entries translated by $-q/p_i$ along the $i$-th coordinate, yielding an affine hyperplane in $\RR^{d-1}$ of dimension $d-2$.  As $q$ and the entries of $\boldsymbol{p}$ are integers, there are only a countable number of such affine hyperplanes.  This yields the desired result.
\end{proof}

We have the following generalization of Theorem~\ref{thmThickenSubsetSectionSpherical}.

\begin{theo}\label{thmThickenSubsetSectionSphericalGeneral} Let $0<\eta <1$ be a fixed constant, $L \in G$, $T_+\geq T_- > \h_0^{2(d-1)/d}$, $\Aa \subset \Aa_L$ be a bounded subset with measure zero boundary with respect to $\wrt {\boldsymbol{x}}$, and $U \subset \SO(d, \RR)$ be a subset measurable with respect to $\wrt k$ and induced by $E$.  Let \[\widetilde{\Ss}(\boldsymbol{z}):=\widetilde{\Ss}_{T_-, T_+}(\boldsymbol{z}):=\widetilde{\Ss}\left(\boldsymbol{\alpha}, \boldsymbol{\gamma}, \widetilde{K}_{U}(\boldsymbol{z})E(\boldsymbol{z})\right),\] and $\widetilde{\Ss}_T$ be defined using (\ref{eqnRelationLowerHeightToFlow} , \ref{eqnPullingBackOfSectionalNeigh2}).  Then \begin{align}\label{eqnthmNAKResultRankOneFlowSubsetThickSphericalGeneralTransHoro}
 {T^{d-1}}\int_{\RR^{d-1}} \In_{\Aa \times \{\widetilde{S}_T(\boldsymbol{z}) E^{-1}(\boldsymbol{z}): \boldsymbol{z}\in \Dd\}}(\boldsymbol{x}, {L n_-(\boldsymbol{x}) \Phi^t})~\wrt {\boldsymbol{x}} \xrightarrow[]{t \rightarrow \infty}T_-^{d/2}\mu(\{\widetilde{S}(\boldsymbol{z}) E^{-1}(\boldsymbol{z}):  \boldsymbol{z} \in \Dd\})\left(\int_{\RR^{d-1}} \In_{\Aa}(\boldsymbol{x})~\wrt{\boldsymbol{x}}\right) \end{align} uniformly for all $T \in [T_-^{d/2(d-1)}, e^{dt \eta}].$

\end{theo}

\begin{proof}  We prove the case $L=I_d$ and defer the proof of the case of generic $L \in G$ to Section~\ref{subsecUniEquiForTransHoro}. We will partition the chart into a sequence of upper and lower charts as follows.  Let $m \geq 2$ be an integer.  Replace the open bounded set $\Dd$ with the  open bounded sets \begin{align*}
&\Dd_m^+:= \Dd \cap E^{-1}\left(\left\{ \begin{pmatrix}  A &  \tensor*[^t]{\boldsymbol{w}}{} \\ \boldsymbol{v} & c  \end{pmatrix} \in \SO(d, \RR): c > \frac 1 m\right\}\right) \\ &\Dd_m^-:= \Dd \cap E^{-1}\left(\left\{ \begin{pmatrix}  A &  \tensor*[^t]{\boldsymbol{w}}{} \\ \boldsymbol{v} & c  \end{pmatrix} \in \SO(d, \RR): c <- \frac 1 m\right\}\right).\end{align*}  The restriction of the chart to $\Dd_m^+$ and $\Dd_m^-$ give the upper and lower charts (for a given $m$), respectively.  Note that $E^{-1}$ in the definitions of $\Dd_m^+$ and $\Dd_m^-$ denotes the preimage of the smooth mapping $E$ (defined in (\ref{eqn:DefnSmoothMapE})) and should not be confused $E^{-1}(\boldsymbol{z})$, which is the inverse in $\SO(d, \RR)$ of $E(\boldsymbol{z})$.

The restrictions of the smooth mapping $E$ to $\Dd_m^+$ and $\Dd_m^-$ are smooth mappings.  Thus, replacing the chart given by $\Dd$ by the upper chart given by $\Dd_m^+$ allows us to apply Theorem~\ref{thmThickenSubsetSectionSpherical} 

For the lower chart given by  $\Dd_m^-$, we have two cases:  $d$ is even and $d$ is odd.  For both cases, we note that Theorem~\ref{thmSphericalThickeningGeneral} implies that the thickenings given by $E^{-1}(\boldsymbol{z})$ and $E^{-1}(\boldsymbol{z}_{ap})$ are the same set.  Moreover, given a point $\pf E^{-1}(\boldsymbol{z})$ in the thickening (thus in $\Gamma \backslash G$), there exists a unique point $\qf E^{-1}(\boldsymbol{z}_{ap})$ in the thickening of $E^{-1}(\boldsymbol{z}_{ap})$ such that $\pf E^{-1}(\boldsymbol{z})=\qf E^{-1}(\boldsymbol{z}_{ap})$ by Proposition~\ref{propConjAntipodalPoints}.  (Note that $\pf, \qf \in \widetilde{S}_T(\boldsymbol{z})$.)  Consider the case of even $d$ first.  Since $-I_d \in \Gamma$ commutes with all elements of $G$ , we can apply Theorem~\ref{thmThickenSubsetSectionSpherical} to lower chart  given by $\Dd_m^-$ by replacing $E^{-1}(\boldsymbol{z}_{ap})$ with $-E^{-1}(\boldsymbol{z}_{ap})$.  This concludes the case $d$ even.

Consider the $d$ odd.  As in the proof of Theorem~\ref{thmTwoChartGeneralSpherical} , we have (\ref{eqn:Spherical:ConjST}), which implies there exists a unique $\widetilde{\qf} \in \widetilde{S}_T(\boldsymbol{z})$ such that \[ \begin{pmatrix}  \widetilde{I}_{d-1} &  \tensor*[^t]{\boldsymbol{0}}{} \\ \boldsymbol{0} & -1  \end{pmatrix}\qf E^{-1}(\boldsymbol{z}_{ap}) = \widetilde{\qf} \begin{pmatrix}  \widetilde{I}_{d-1} &  \tensor*[^t]{\boldsymbol{0}}{} \\ \boldsymbol{0} & -1  \end{pmatrix}E^{-1}(\boldsymbol{z}_{ap}).\]  This allows us to apply Theorem~\ref{thmThickenSubsetSectionSpherical} to lower chart given by $\Dd_m^-$ by replacing $E^{-1}(\boldsymbol{z}_{ap})$ with \[\begin{pmatrix}  \widetilde{I}_{d-1} &  \tensor*[^t]{\boldsymbol{0}}{} \\ \boldsymbol{0} & -1  \end{pmatrix}E^{-1}(\boldsymbol{z}_{ap}).\]  This concludes the case $d$ odd.

Thus, we have the desired result \textit{restricted} to the upper and lower charts given by $\Dd_m^+$ and $\Dd_m^-$, respectively, for all $m \geq 2$.  Explicitly, we have \begin{align}\label{eqnthmNAKResultRankOneFlowSubsetThickRestricted} 
 {T^{d-1}}\int_{\TT^{d-1}} \In_{\Aa \times \{\widetilde{S}_T(\boldsymbol{z}) E^{-1}(\boldsymbol{z}): \boldsymbol{z}\in \Dd_m^+ \cup \Dd_m^-\}}&(\boldsymbol{x}, {n_-(\boldsymbol{x}) \Phi^t})~\wrt {\boldsymbol{x}} \xrightarrow[]{t \rightarrow \infty}  \\ & \nonumber T_-^{d/2}  \mu(\{\widetilde{S}(\boldsymbol{z}) E^{-1}(\boldsymbol{z}):  \boldsymbol{z} \in  \Dd_m^+ \cup \Dd_m^-\})\left(\int_{\TT^{d-1}} \In_{\Aa}(\boldsymbol{x})~\wrt{\boldsymbol{x}}\right) \end{align} for $m \geq2$.  To lift this restriction, what remains to be shown lies in the restriction of the (original) chart to the sequence of domains \[\Dd \backslash \left(\Dd_m^+ \cup \Dd_m^- \right) \supset \Dd \backslash \left(\Dd_{m+1}^+ \cup \Dd_{m+1}^- \right).\]  Note that \begin{align}\label{eqnTwoSubchartsApprox}\bigcup_{m \geq 2} \left( \Dd_m^+ \cup \Dd_m^- \right)=: \Dd_\infty.
  \end{align}

 We now use the following approximation technique.  First note that \[ \mu(\{\widetilde{S}(\boldsymbol{z}) E^{-1}(\boldsymbol{z}):  \boldsymbol{z} \in  \Dd_\infty\}) =  \mu(\{\widetilde{S}(\boldsymbol{z}) E^{-1}(\boldsymbol{z}):  \boldsymbol{z} \in  \Dd\})\] because  $K_0$ is a null set of $K$ (see~\cite[Pages 19-20]{Mec19}).  Recall that $T$ depends on $t$ and, for the duration of this proof, we will make this explicit by using the notation $T(t)$.  Using (\ref{eqnTwoSubchartsApprox}) with the continuity of $\mu$ from below and using (\ref{eqnthmNAKResultRankOneFlowSubsetThickRestricted}), we have that, for every $\varepsilon>0$, there exists $t_0>0$ and $m_0\geq2$ such that, for every $t \geq t_0$ and $m \geq m_0$, we have that \begin{align}\label{eqnthmNAKResultRankOneFlowSubsetThickRestricted2} & \left | {T(t)^{d-1}}\int_{\TT^{d-1}} \In_{\Aa \times \{\widetilde{S}_{T(t)}(\boldsymbol{z}) E^{-1}(\boldsymbol{z}): \boldsymbol{z}\in \Dd_m^+ \cup \Dd_m^-\}} (\boldsymbol{x}, {n_-(\boldsymbol{x}) \Phi^t})~\wrt {\boldsymbol{x}}  \right. \\ \nonumber & \quad\quad  \left. - T_-^{d/2}  \mu(\{\widetilde{S}(\boldsymbol{z}) E^{-1}(\boldsymbol{z}):  \boldsymbol{z} \in  \Dd\})\left(\int_{\TT^{d-1}} \In_{\Aa}(\boldsymbol{x})~\wrt{\boldsymbol{x}}\right) \right | < \varepsilon.  \end{align}

 Using (\ref{eqnTwoSubchartsApprox}) to apply the monotone convergence theorem, we have for every $t >0$, we have that \begin{align}\label{eqnTwoSubchartsApproxMCT}
\lim_{m \rightarrow \infty} {T(t)^{d-1}}\int_{\TT^{d-1}} \In_{\Aa \times \{\widetilde{S}_{T(t)}(\boldsymbol{z}) E^{-1}(\boldsymbol{z}): \boldsymbol{z}\in \Dd_m^+ \cup \Dd_m^-\}} (\boldsymbol{x}, {n_-(\boldsymbol{x}) \Phi^t})~\wrt {\boldsymbol{x}}  \\ \nonumber =  {T(t)^{d-1}}\int_{\TT^{d-1}} \In_{\Aa \times \{\widetilde{S}_{T(t)}(\boldsymbol{z}) E^{-1}(\boldsymbol{z}): \boldsymbol{z}\in \Dd_\infty\}} (\boldsymbol{x}, {n_-(\boldsymbol{x}) \Phi^t})~\wrt {\boldsymbol{x}}.\end{align}  Taking the limit as $m \rightarrow \infty$ in (\ref{eqnthmNAKResultRankOneFlowSubsetThickRestricted2}) and applying (\ref{eqnTwoSubchartsApproxMCT}) yields \begin{align} \label{eqnthmNAKResultRankOneFlowSubsetThickRestricted3}& \left | {T(t)^{d-1}}\int_{\TT^{d-1}} \In_{\Aa \times \{\widetilde{S}_{T(t)}(\boldsymbol{z}) E^{-1}(\boldsymbol{z}): \boldsymbol{z}\in \Dd_\infty\}} (\boldsymbol{x}, {n_-(\boldsymbol{x}) \Phi^t})~\wrt {\boldsymbol{x}}  \right. \\ \nonumber & \quad\quad  \left. - T_-^{d/2}  \mu(\{\widetilde{S}(\boldsymbol{z}) E^{-1}(\boldsymbol{z}):  \boldsymbol{z} \in  \Dd\})\left(\int_{\TT^{d-1}} \In_{\Aa}(\boldsymbol{x})~\wrt{\boldsymbol{x}}\right) \right | < \varepsilon \end{align} for every $t \geq t_0$.  Taking the limit $t \rightarrow \infty$, we obtain  \begin{align} \label{eqnthmNAKResultRankOneFlowSubsetThickRestricted4}& \left | \lim_{t \rightarrow \infty} {T(t)^{d-1}}\int_{\TT^{d-1}} \In_{\Aa \times \{\widetilde{S}_{T(t)}(\boldsymbol{z}) E^{-1}(\boldsymbol{z}): \boldsymbol{z}\in \Dd_\infty\}} (\boldsymbol{x}, {n_-(\boldsymbol{x}) \Phi^t})~\wrt {\boldsymbol{x}}  \right. \\ \nonumber & \quad\quad  \left. - T_-^{d/2}  \mu(\{\widetilde{S}(\boldsymbol{z}) E^{-1}(\boldsymbol{z}):  \boldsymbol{z} \in  \Dd\})\left(\int_{\TT^{d-1}} \In_{\Aa}(\boldsymbol{x})~\wrt{\boldsymbol{x}}\right) \right | < \varepsilon.\end{align}  As $\varepsilon$ is arbitrary, we have shown the desired result when $c\neq0$.  

For $c=0$, apply Lemma~\ref{lemm:HorosphereInSphereEdge}.  This yields the desired result for the case $L=I_d$. \end{proof}

\begin{rema}
Note that, as long as \[\{\boldsymbol{z}, \boldsymbol{z}_{ap} \} \subset  \Dd_m^+ \cup \Dd_m^-\] holds, we obtain the thickening given by $E^{-1}(\boldsymbol{z})$ for the upper and lower charts.  Thus, different open bounded sets $\Dd$ may yield the same thickenings and, as long as one of every antipodal pair is in the open bounded sets, then we obtain the same thickening.  

\end{rema}

\section{Translated horospheres}\label{secTranslatedHorospheres}  In this section, we consider STHE for translated horospheres.   All of the STHE results that we have proved for the horosphere can also be proved for any translated horosphere, as we show in Section~\ref{subsecUniEquiForTransHoro}.  First, we need to find the correct analog of the Farey points, which is done in Section~\ref{subsecTranFareySeq}, and to prove a suitable generalization of Theorem~\ref{thmMarklof}, namely Theorem~\ref{thmGeneralMarklof}.

The proofs of our STHE results for translated horospheres are essentially the same as that for the horosphere considered in the previous sections.  The changes are, briefly, checking  that (a piece of) a translated horosphere intersects $\Ss_1$ through a unique point along the $N^+$-directions and that our generalization of Theorem~\ref{thmMarklof} can be applied.  The details are as follows.

Recall that a translated horosphere corresponds to a fixed $L \in G$.  Checking the intersection property is easy.  Let us assume that $\Ga\tensor*[^t]{\left(Ln_-(\boldsymbol{x}) \Phi^{t}\right)}{}^{-1}$ intersects $\Ss_1$.  Then a nearby point $\Ga\tensor*[^t]{\left(Ln_-(\boldsymbol{y}) \Phi^{t}\right)}{}^{-1} = \Ga \tensor*[^t]{L}{^{-1}}n_+(-\boldsymbol{y}) \Phi^{-t}$ on the $\Phi^{t}$-translate of the translated horosphere is in the $N^+$-directions with respect to the intersection point $\Ga\tensor*[^t]{\left(Ln_-(\boldsymbol{x}) \Phi^{t}\right)}{}^{-1} = \Ga \tensor*[^t]{L}{^{-1}}n_+(-\boldsymbol{x}) \Phi^{-t}$, an observation which follows by applying (\ref{eqnConjy1}).  We now simply need to check that the $N^+$ directions do not locally lie in the section $\Ss_1$, namely that the intersection is transverse.

\begin{lemm}\label{lemmTransverseOnG}  
Let $L \in G$ and $n, n' \in N^+$.  If $Ln, L n' \in H\{\Phi^{-s}: s \in \RR\}$, then $n = n'$. 
\end{lemm}
\begin{proof}
Let \[L := \begin{pmatrix}  A & B \\ C& D \end{pmatrix} \quad n := \begin{pmatrix} I_{d-1} & \tensor*[^t]{\boldsymbol{0}}{} \\ \boldsymbol{\widetilde{x}} & 1  \end{pmatrix} \quad n' := \begin{pmatrix} I_{d-1} & \tensor*[^t]{\boldsymbol{0}}{} \\ \boldsymbol{\widetilde{x}}' & 1  \end{pmatrix}.\]  Thus we have \[Ln = \begin{pmatrix}  A + B \boldsymbol{\widetilde{x}} & B \\ C + D \boldsymbol{\widetilde{x}} & D \end{pmatrix} \quad Ln' = \begin{pmatrix}  A + B \boldsymbol{\widetilde{x}}' & B \\ C + D \boldsymbol{\widetilde{x}}' & D \end{pmatrix}.\]  Since $Ln, L n' \in H\{\Phi^{-s}: s \in \RR\}$, we have that $C + D \boldsymbol{\widetilde{x}} = \boldsymbol{0}$ and $C + D \boldsymbol{\widetilde{x}}' = \boldsymbol{0}$ and that the diagonal blocks must not have determinant $0$.  In particular, we have that the number $D \neq 0$ and thus the desired result.

\end{proof}

\begin{lemm}\label{lemmTransverseOnS}  Let $L \in G$ and $n, n' \in N^+$ such that $n, n'$ lie in a small neighborhood. If $Ln, L n' \in \Ss_1$, then $n = n'$. 
\end{lemm}
\begin{proof}
Since the left action by an element of $G$ is continuous, $Ln$ and $Ln'$ lie in a small neighborhood, which we may assume is small enough so that it lies in some fundamental domain for the $\Ga$ action.  Applying Lemma~\ref{lemmTransverseOnG} yields the desired result.

\end{proof}

\begin{theo}\label{thm:TranHoroInterSectS1AtAPoint}
The intersection of any left-translated horosphere with $\Ss_1$ is transverse.  Moreover, the local intersection is a point. 
\end{theo}
\begin{proof}
 Apply Lemma~\ref{lemmTransverseOnS}.
\end{proof}

Consequently, the (local) \textit{geometry} of intersections between our original horosphere (i.e. when one sets $L = I_d$) and $\Ss_1$ and between a translated horosphere and $\Ss_1$ are the same for us.   Note that we can, using Lemma~\ref{lemmDisjointnessTransSections}, give a good estimate of what nearby means in this case, but it is not necessary for this proof.

\subsection{Translated Farey sequences} \label{subsecTranFareySeq}

Up to now, we have restricted the Farey sequence to lie in $\TT^{d-1}$ (see Section~\ref{secEotFs}), but it is more convenient in motivating the definition of translated Farey sequences to regard the Farey sequence as lying in $\RR^{d-1}$.  Thus, we consider the Farey sequence as the following subset of $\RR^{d-1}$:  \[\widetilde{\F}:=\bigcup_{Q=1}^\infty \widetilde{\F}_Q:=\bigcup_{Q=1}^\infty\left(\F_Q + \ZZ^{d-1}\right).\]  It is shown in \cite{Mar10} that $\widetilde{\F}_Q$ corresponds to the intersection points of the horosphere $\tensor*[^t]{\left(\Ga N^-\Phi^t\right)}{^{-1}}$ with the closed embedded submanifold $\Ga \backslash \Ga H\{\Phi^{-s}: s\geq \sigma\}$.  (Recall that $Q := e^{(d-1)(t - \sigma)}$ for a fixed $\sigma \in \RR$.)

We now use the analogous correspondence to define the elements of a translated Farey sequence on $\RR^{d-1}$.  The {\em translated Farey sequence (corresponding to $\Ga L  N^-$)} is the subset $\If$ of $\RR^{d-1}$ where $\If:=\If(L):=\cup_{Q=1}^\infty \If_Q $ and \begin{align} \label{eqn:DefnTranFareySeq}
 \If_Q := & \If_Q(L):= \\\nonumber &\left\{\left(\frac{\alpha_1}{\alpha_d}, \cdots,  \frac{\alpha_{d-1}}{\alpha_d}\right) \in \RR^{d-1} : \left(\alpha_1, \cdots, \alpha_d\right)=(\boldsymbol{0}, 1) \gamma \tensor*[^t]{L}{^{-1}}, \gamma \in \Ga_H \backslash \Ga, 0< \alpha_d\leq Q  \right\}. \end{align}

First we note that, using the bijection (\cite[(3.18)]{Mar10} or~\cite{Sie89}) \[\Ga_H \backslash \Ga \rightarrow \widehat{\ZZ}^d \quad \Ga_H \gamma \mapsto (\boldsymbol{0},1) \gamma,\] we have that $\{(\boldsymbol{0}, 1) \gamma \tensor*[^t]{L}{^{-1}}:  \gamma \in \Ga_H \backslash \Ga\} = \widehat{\ZZ}^d \tensor*[^t]{L}{^{-1}}$ is the set of primitive points of the lattice $\ZZ^d \tensor*[^t]{L}{^{-1}}$, or, equivalently, the set of the points of the lattice $\ZZ^d \tensor*[^t]{L}{^{-1}}$ visible from the origin of $\RR^d$~\cite[Section~2.4]{MS10}.  Let \[\left(\widehat{\ZZ}^d \tensor*[^t]{L}{^{-1}}\right)_Q^+ := \left\{(\alpha_1, \cdots, \alpha_d) \in \widehat{\ZZ}^d \tensor*[^t]{L}{^{-1}}: 0<\alpha_d \leq Q\right\}\] and \[\left(\widehat{\ZZ}^d \tensor*[^t]{L}{^{-1}}\right)^+ := \left\{(\alpha_1, \cdots, \alpha_d) \in \widehat{\ZZ}^d \tensor*[^t]{L}{^{-1}}: 0<\alpha_d \right\}.\] It is easy to see that $\If_Q$ and $\left(\widehat{\ZZ}^d \tensor*[^t]{L}{^{-1}}\right)^+_Q$ are in bijection for every $Q \in \NN$ and that $\If$ is in bijection with $\left(\widehat{\ZZ}^d \tensor*[^t]{L}{^{-1}}\right)^+$.

Let $A \in G$ and $\delta(A) = A_{d,d}$ be its $(d,d)-$th entry.  Recall that we have the parametrization (\ref{eqnParametrizationForG}).

\begin{lemm}\label{lemmUniqueDecomIfYdNot0}
Let $\gamma \in \Ga$.  We have that $\delta(\gamma A) > 0$ if and only if there exist $M \in H$ and $\boldsymbol{y}:=(\boldsymbol{y'}, y_d) \in \RR^{d-1} \times (0, \infty)$ such that \[\gamma A = M \begin{pmatrix}  y_d^{-1/(d-1)} I_{d-1} & \tensor*[^t]{\boldsymbol{0}}{} \\\boldsymbol{y'} &y_d   \end{pmatrix}\]  and that $\delta(\gamma A) < 0$ if and only if there exist $M \in H$ and $\boldsymbol{y}:=(\boldsymbol{y'}, y_d) \in \RR^{d-1} \times (-\infty, 0)$ such that \[\gamma A = \begin{cases} -I_d M \begin{pmatrix}  (-y_d)^{-1/(d-1)} I_{d-1} & \tensor*[^t]{\boldsymbol{0}}{} \\ -\boldsymbol{y'} & -y_d  \end{pmatrix}&\text{ if } d \text{ is even}\\ \begin{pmatrix}  \widetilde{I}_{d-1} &  \tensor*[^t]{\boldsymbol{0}}{} \\ \boldsymbol{0} & -1  \end{pmatrix} M \begin{pmatrix}  (-y_d)^{-1/(d-1)} I_{d-1} & \tensor*[^t]{\boldsymbol{0}}{} \\ -\boldsymbol{y'} &-y_d   \end{pmatrix} &\text{ if } d \text{ is odd.}\end{cases}\] Moreover, whenever it exists, the decomposition is unique.
\end{lemm}
\begin{rema}
Recall that the use of \[\begin{pmatrix}  \widetilde{I}_{d-1} &  \tensor*[^t]{\boldsymbol{0}}{} \\ \boldsymbol{0} & -1  \end{pmatrix}\] is one of several natural possible choices.  We can regard this decomposition as coordinates and will refer to it \textit{$(H, \RR^d)$-coordinates}.  The \textit{$H$-coordinate} and \textit{$\RR^d$-coordinate} will refer to the $H$-part and the $\RR^d$-part, respectively, of the $(H, \RR^d)$-coordinate.  Note that, regardless of whether $\delta(\gamma A) > 0$ or $\delta(\gamma A)<0$, the $d$-th entry of the $\RR^d$-coordinate of $\gamma A$ is always $>0$.
\end{rema}
\begin{proof}
($\implies$).  Set $\boldsymbol{y} = (\boldsymbol{0}, 1) \gamma A.$ Thus we have that $y_d \neq 0$.  If $y_d >0$, then we may set \[M =  \gamma A \begin{pmatrix}  y_d^{-1/(d-1)} I_{d-1} & \tensor*[^t]{\boldsymbol{0}}{} \\\boldsymbol{y'} &y_d   \end{pmatrix}^{-1},\] which it is easy to see is in $H$.

If $y_d <0$, then set \[\begin{cases} (\widetilde{\boldsymbol{y}}', \widetilde{y}_d)=:\widetilde{\boldsymbol{y}}= (\boldsymbol{0}, 1) \left(-I_d \gamma A \right) & \text{ if } d \text{ is even} \\  (\widetilde{\boldsymbol{y}}', \widetilde{y}_d)=:\widetilde{\boldsymbol{y}}= (\boldsymbol{0}, 1)\begin{pmatrix}  \widetilde{I}_{d-1} &  \tensor*[^t]{\boldsymbol{0}}{} \\ \boldsymbol{0} & -1  \end{pmatrix} \gamma A & \text{ if } d \text{ is odd.} \end{cases}\]  Thus, we have $\widetilde{\boldsymbol{y}} = - \boldsymbol{y}$.  In particular, we have that $\widetilde{y}_d >0$ and applying the above proof for positive $d$-th coordinate with $\widetilde{\boldsymbol{y}}$ in place of $\boldsymbol{y}$ and \[\begin{cases} -I_d  \gamma & \text{ if } d \text{ is even} \\ \begin{pmatrix}  \widetilde{I}_{d-1} &  \tensor*[^t]{\boldsymbol{0}}{} \\ \boldsymbol{0} & -1  \end{pmatrix} \gamma & \text{ if } d \text{ is odd} \end{cases} \] in place of $\gamma$ yields the desired result.

($\impliedby$).  This is immediate.

To see uniqueness, we note that the intersection of the following subgroups of $\SL(d,\RR)$ is the trivial group:  \[H \cap \left\{ \begin{pmatrix}  y_d^{-1/(d-1)} I_{d-1} & \tensor*[^t]{\boldsymbol{0}}{} \\\boldsymbol{y'} &y_d   \end{pmatrix} : \boldsymbol{y'}  \in \RR^{d-1}, y_d > 0 \right\} = \{I_d\}.\]

\end{proof}

The lemma has an immediate corollary:

\begin{coro}\label{coroUniqueDecomIfYdNot0}
Let $\gamma \in \Ga$ and $t \in \RR$.  We have that  $\gamma A N^+\Phi^{-t}$ intersects \[  H \left\{ \begin{pmatrix}  y_d^{-1/(d-1)} I_{d-1} & \tensor*[^t]{\boldsymbol{0}}{} \\\boldsymbol{y'} &y_d   \end{pmatrix}:\boldsymbol{y}:=(\boldsymbol{y'}, y_d) \in \RR^{d-1} \times (0, \infty) \right\}\] if and only if $\delta(\gamma A) > 0$ and $\gamma A N^+\Phi^{-t}$ intersects \[\begin{cases}   -I_d H \left\{ \begin{pmatrix}  y_d^{-1/(d-1)} I_{d-1} & \tensor*[^t]{\boldsymbol{0}}{} \\\boldsymbol{y'} &y_d   \end{pmatrix}:\boldsymbol{y}:=(\boldsymbol{y'}, y_d) \in \RR^{d-1} \times (0, \infty) \right\} & \text{ when } d \text{ is even,} \\ \begin{pmatrix}  \widetilde{I}_{d-1} &  \tensor*[^t]{\boldsymbol{0}}{} \\ \boldsymbol{0} & -1  \end{pmatrix}  H \left\{ \begin{pmatrix}  y_d^{-1/(d-1)} I_{d-1} & \tensor*[^t]{\boldsymbol{0}}{} \\\boldsymbol{y'} &y_d   \end{pmatrix}:\boldsymbol{y}:=(\boldsymbol{y'}, y_d) \in \RR^{d-1} \times (0, \infty) \right\} & \text{ when } d \text{ is odd,}\end{cases}\] if and only if $\delta(\gamma A) < 0$.
\end{coro}

\begin{lemm}\label{lemmChoiceLDeltaNonZero}
There exists an $\gamma \in \Ga$ such that $\delta (\tensor*[^t]{L}{}\gamma) >0$.
\end{lemm}

\begin{proof}
We first show that there exists an $\gamma \in \Ga$ such that $\delta (\tensor*[^t]{L}{}\gamma) \neq 0$.
 Assume not.  Then $\delta (\tensor*[^t]{L}{}\gamma) = 0$ for all $\gamma \in \Ga$.  Hence $\delta(\tensor*[^t]{L}{})=0$ and $(\boldsymbol{w}', 0) = (\boldsymbol{0}, 1) \tensor*[^t]{L}{}$.  Consequently, for every $\gamma \in \Ga$, we have that \[(\boldsymbol{0}, 1) \tensor*[^t]{L}{} \gamma = (\boldsymbol{w}', 0) \gamma= (\boldsymbol{u}, 0)\] for some $\boldsymbol{u}$ depending on $\gamma$.  Now, since the last column of an element of $\Ga$ can be any standard basis vector of $\RR^d$, it follows that $\RR^d \subset \tensor*[^t]{(\boldsymbol{w}', 0)}{^{\perp}}$, which implies that $\boldsymbol{w}' = \boldsymbol{0}$ and that $L$ is not invertible.  This yields a contradiction and shows that  there exists an $\gamma \in \Ga$ such that $\delta (\tensor*[^t]{L}{}\gamma) \neq 0$.
 
For this $\gamma \in \Ga$, if $\delta (\tensor*[^t]{L}{}\gamma) >0$, then we have shown the desired result.  Otherwise, we have $\delta (\tensor*[^t]{L}{}\gamma) < 0$ and, in this case, taking the $\gamma$ and replacing it by \[\begin{cases} \gamma (-I_d) & \text{ if } d \text{ is even} \\ \gamma \begin{pmatrix}  \widetilde{I}_{d-1} &  \tensor*[^t]{\boldsymbol{0}}{} \\ \boldsymbol{0} & -1  \end{pmatrix} &\text{ if } d \text{ is odd}  \end{cases}\] yields the desired result.

\end{proof}

For $i \in \{1, \cdots, d\}$, let $s_{i,d} \in \Ga$ be the matrix with all entries either $0, 1,$ or $-1$ such that left-multiplying by it performs the elementary row operation of replacing the $i$-th row with the $d$-th row and either replacing the $d$-th row with the $i$-th row or replacing the $d$-th row with the negation of the $i$-th row, depending on which choice corresponds to $\det  (s_{i,d})=1$. 

Now, for \[M := \begin{pmatrix} A_M & \tensor*[^t]{\boldsymbol{b}}{_M} \\\boldsymbol{c}_M &  a_M  \end{pmatrix} \in G,\] we make the following observation:  \begin{align}\label{eqnDecomp2InG}
 M= \begin{cases}  \begin{pmatrix}  I_{d-1} &\tensor*[^t]{\boldsymbol{0}}{}\\ \boldsymbol{c}_MA_M^{-1} &-\boldsymbol{c}_MA_M^{-1}\tensor*[^t]{\boldsymbol{b}}{_M}+ a_M \end{pmatrix}  \begin{pmatrix}  A_M &\tensor*[^t]{\boldsymbol{b}}{_M}\\ \boldsymbol{0} &1 \end{pmatrix}  & \textrm{ if } \det(A_M) \neq 0 \\  s_{i,d} \begin{pmatrix}  I_{d-1} &\tensor*[^t]{\boldsymbol{0}}{}\\ \widehat{\boldsymbol{c}}_M\widehat{A}_M^{-1} &-\widehat{\boldsymbol{c}}_M\widehat{A}_M^{-1}\tensor*[^t]{\widehat{\boldsymbol{b}}}{_M}+\widehat{a}_M \end{pmatrix}  \begin{pmatrix}  \widehat{A}_M &\tensor*[^t]{\widehat{\boldsymbol{b}}}{_M}\\ \boldsymbol{0} &1 \end{pmatrix}  & \textrm{ if } \det(A_M) = 0\end{cases} \end{align} where  \[\begin{pmatrix}  \widehat{A}_M &\tensor*[^t]{\widehat{\boldsymbol{b}}}{_M}\\ \widehat{\boldsymbol{c}}_M &\widehat{a}_M \end{pmatrix} := s_{i,d}^{-1} M \] and the index $i$ is chosen such that $\det(\widehat{A}_M) \neq 0$.  Since $M \in G$, we have that \begin{align}\label{Decomp2InG2} -\boldsymbol{c}_MA_M^{-1}\tensor*[^t]{\boldsymbol{b}}{_M}+ a_M = \frac 1 {\det(A_M)} \quad \textrm{ and } -\widehat{\boldsymbol{c}}_M\widehat{A}_M^{-1}\tensor*[^t]{\widehat{\boldsymbol{b}}}{_M}+\widehat{a}_M  = \frac 1 {\det(\widehat{A}_M)}.
  \end{align}

\begin{lemm}\label{lemmChoiceLDeltaNonZero2}
Let $M \in G$ such that $\delta (M) \neq 0$ holds.  Then there exists $i \in \{1, \cdots, d\}$ such that $\delta (s^{-1}_{i,d} M) \neq 0$ and $\delta (M^{-1} s_{i,d}) \neq 0$.
\end{lemm}
\begin{proof}

Let \[M:= \begin{pmatrix} A_M & \tensor*[^t]{\boldsymbol{b}}{_M} \\\boldsymbol{c}_M &  a_M  \end{pmatrix}.\]  Using (\ref{eqnDecomp2InG}), we have two cases.

\textit{Case:  $\det(A_M) \neq 0$.}  We will show that the lemma holds with $i=d$.  It is enough to show that $\delta (M^{-1}) \neq 0.$  Note that we have \[M^{-1} =  \begin{pmatrix}  A_M^{-1} & -  A_M^{-1}\tensor*[^t]{\boldsymbol{b}}{_M}\\ \boldsymbol{0} &1 \end{pmatrix}  \begin{pmatrix}  I_{d-1} &\tensor*[^t]{\boldsymbol{0}}{}\\ \frac{-\boldsymbol{c}_MA_M^{-1}}{-\boldsymbol{c}_MA_M^{-1}\tensor*[^t]{\boldsymbol{b}}{_M}+ a_M} &\frac 1{-\boldsymbol{c}_MA_M^{-1}\tensor*[^t]{\boldsymbol{b}}{_M}+ a_M} \end{pmatrix}.\]  Now $-\boldsymbol{c}_MA_M^{-1}\tensor*[^t]{\boldsymbol{b}}{_M}+ a_M \neq 0$ for, otherwise, $\det(M)=0$, which leads to a contradiction.  Thus \begin{align}\label{eqnInverseMatrixddCoordInvertibleA}
 \delta (M^{-1}) = \frac 1{-\boldsymbol{c}_MA_M^{-1}\tensor*[^t]{\boldsymbol{b}}{_M}+ a_M} \neq 0 \end{align} is the desired result.

\textit{Case:  $\det(A_M) = 0$.}  If $\widehat{a}_M  \neq 0$, then replacing $M$ with $s_{i,d}^{-1} M$ in the previous case yields the desired result.  Now for any $j \in \{1, \cdots, d\}$, we can consider \[\begin{pmatrix}  \widehat{A}_j &\tensor*[^t]{\widehat{\boldsymbol{b}}}{_j}\\ \widehat{\boldsymbol{c}}_j &\widehat{a}_j \end{pmatrix} := s_{j,d}^{-1} M,\] and, if there does not exist $j$ such that $\det(\widehat{A}_j) \neq 0$ and $\widehat{a}_j  \neq 0$, then $\det(M) =0$, which leads to a contradiction.  Thus such a $j$ exists; setting $i=j$ gives the desired result.

This completes the proof of the lemma.

\end{proof}

\begin{rema}\label{remaTotIrrNonIntersect}  Note that we have $\gamma \tensor*[^t]{L}{^{-1}} N^+\Phi^{-t}$ intersects \[\Gamma H \left\{ \begin{pmatrix}  y_d^{-1/(d-1)} I_{d-1} & \tensor*[^t]{\boldsymbol{0}}{} \\\boldsymbol{y'} &y_d   \end{pmatrix}:\boldsymbol{y}:=(\boldsymbol{y'}, y_d) \in \RR^{d-1} \times (0, \infty) \right\}\] for every $\gamma \in \Gamma$ (i.e. regardless of the value of $\delta(\gamma \tensor*[^t]{L}{^{-1}})$.)  Corollary~\ref{coroUniqueDecomIfYdNot0} implies this for $\delta(\gamma \tensor*[^t]{L}{^{-1}}) \neq 0$.  If $\delta(\gamma \tensor*[^t]{L}{^{-1}}) = 0$, then, for some choice of $i$, $\delta(s^{-1}_{i,d}\gamma \tensor*[^t]{L}{^{-1}}) \neq 0$ and Corollary~\ref{coroUniqueDecomIfYdNot0} will again imply this.

\end{rema}

Given $A \in G$ such that $\delta(A) > 0$, let $A_H A_{\boldsymbol{y}}$ be the unique decomposition from Lemma~\ref{lemmUniqueDecomIfYdNot0} where $A_H \in H$ and $\boldsymbol{y} = (\boldsymbol{0},1) A$.  Note that, for all $M \in H$, we have that $(MA)_H = M A_H$ and $(MA)_{\boldsymbol{y}}=A_{\boldsymbol{y}}$.

\begin{theo}\label{thmTransFareyInterPointCorr}  Fix $\sigma, t \in \RR$ and let $Q := e^{(d-1)(t - \sigma)}$.
The set of intersection points of the translated horosphere $\tensor*[^t]{\left(\Ga L  N^- \Phi^t\right)}{^{-1}}$ with $\Ga \backslash \Ga H\{\Phi^{-s}: s\geq \sigma\}$ is \[\left\{\tensor*[^t]{\left(\Ga L  n_-\left({\boldsymbol{y'}}/{y_d}\right) \Phi^t\right)}{^{-1}} : (\boldsymbol{y'},y_d) \in \If_Q \right\} = \left\{\tensor*[^t]{\left(\Ga L  n_-\left({\boldsymbol{y'}}/{y_d}\right) \Phi^t\right)}{^{-1}} : (\boldsymbol{y'},y_d) \in \left(\widehat{\ZZ}^d \tensor*[^t]{L}{^{-1}}\right)_Q^+ \right\}.\]

\end{theo}

\begin{proof}
Let $\widetilde{L} :=  \tensor*[^t]{L}{^{-1}}$. First, let us show that these points are intersection points.  Let $\boldsymbol{y}:=(y_1, \cdots, y_d) \in \If_Q$.  Then there exists $\gamma \in \Ga_H \backslash \Ga$ such that $\boldsymbol{y} = (\boldsymbol{0}, 1) \gamma \widetilde{L}$ and $0<y_d\leq Q$.  Thus $\delta(\gamma  \widetilde{L}) > 0$ and Lemma~\ref{lemmUniqueDecomIfYdNot0} gives the unique decomposition \[\gamma  \widetilde{L} = (\gamma  \widetilde{L})_H \begin{pmatrix}  y_d^{-1/(d-1)} I_{d-1} & \tensor*[^t]{\boldsymbol{0}}{} \\\boldsymbol{y'} &y_d   \end{pmatrix},\] which yields \[\gamma  \widetilde{L}n_+(-\boldsymbol{y'}/{y_d})\Phi^{-t} = (\gamma  \widetilde{L})_H \Phi^{-t+\frac 1{d-1}\log (y_d)} \in  H \{\Phi^{-s}: s\geq \sigma\}.\]  This shows that these are the desired intersection points.

Let us show all intersection points are of the desired form.  Let \[\tensor*[^t]{\left(\Ga L  N^- \Phi^t\right)}{^{-1}} \cap \Ga \backslash \Ga H\{\Phi^{-s}: s\geq \sigma\} \neq \emptyset.\]  Then there exists $\gamma \in \Ga$ and $\boldsymbol{x} \in \RR^{d-1}$ such that \[\gamma \widetilde{L}n_+(\boldsymbol{x})\Phi^{-t} = M \Phi^{-r}\] where $M \in H$ and $r \geq \sigma$.  Hence, we have \[\gamma \widetilde{L} = M \Phi^{-r+t}n_+(-\boldsymbol{x}).\]  A direct computation shows that $\delta(\gamma \widetilde{L}) = e^{(d-1)(t-r)}> 0.$  Lemma~\ref{lemmUniqueDecomIfYdNot0} gives that \[(\gamma  \widetilde{L})_H \begin{pmatrix}  y_d^{-1/(d-1)} I_{d-1} & \tensor*[^t]{\boldsymbol{0}}{} \\\boldsymbol{y'} &y_d   \end{pmatrix} = M \left(\Phi^{-r+t}n_+(-\boldsymbol{x})\right)\] where $(\boldsymbol{y'},y_d) = (\boldsymbol{0}, 1) \gamma \widetilde{L}$ and that this decomposition is unique.  This uniqueness implies that $M = (\gamma  \widetilde{L})_H$ and \[\begin{pmatrix}  y_d^{-1/(d-1)} I_{d-1} & \tensor*[^t]{\boldsymbol{0}}{} \\\boldsymbol{y'} &y_d   \end{pmatrix} =  \Phi^{-r+t}  n_+(-\boldsymbol{x}).\]  Thus $\boldsymbol{x} = - \boldsymbol{y'}/{y_d}$ and $y_d = e^{(d-1)(t-r)} \leq Q$.  This concludes the proof.

\end{proof}

\begin{rema}
Let $(\boldsymbol{y'},y_d) \in \If_Q$ and $\gamma_{(\boldsymbol{y'},y_d)}$ be an element of $\Gamma$ for which \begin{align}\label{eqnDefnOtherpq}
 (\boldsymbol{0},1) \gamma_{(\boldsymbol{y'},y_d)} \widetilde{L}=(\boldsymbol{y'},y_d). \end{align}  It is easy to see that the only elements of $\gamma \in \Gamma$ for which $ (\boldsymbol{0},1) \gamma \widetilde{L}=(\boldsymbol{y'},y_d)$ holds are exactly the elements $\gamma$ of the coset $\Gamma_H  \gamma_{(\boldsymbol{y'},y_d)}$.
Note that each element $(\boldsymbol{y'},y_d) \in \If_Q$ is in bijective correspondence to the following collection of intersection points of $\tensor*[^t]{\left(\Ga L  N^- \Phi^t\right)}{^{-1}}$ with $H\{\Phi^{-s}: s\geq \sigma\}$:  \[\left\{\gamma  \widetilde{L}n_+(-\boldsymbol{y'}/{y_d})\Phi^{-t}: \textit{ for all } \gamma \in \Gamma \textit{ such that } (\boldsymbol{0},1) \gamma \widetilde{L}=(\boldsymbol{y'},y_d)\right\} = \Gamma_H \gamma_{(\boldsymbol{y'},y_d)} \widetilde{L}n_+(-\boldsymbol{y'}/{y_d})\Phi^{-t}.\]  Consequently, each element $(\boldsymbol{y'},y_d) \in \If_Q$ is in bijective correspondence to exactly one intersection of the translated horosphere with $\Gamma_H \backslash H\{\Phi^{-s}: s\geq \sigma\}$.

\end{rema}

Let us consider the following finite subset of $\If_Q$, \begin{align}\label{eqn:subsetTransFareySeq}
 \widehat{\If}_Q :=\left\{\left(\frac{\alpha_1}{\alpha_d}, \cdots,  \frac{\alpha_{d-1}}{\alpha_d}\right) \in \If_Q : 0 \leq \alpha_i \leq Q \textrm{ for } i = 1, \cdots, d-1 \right\}, \end{align} for which $\If = \cup_{Q=1}^\infty \widehat{\If}_Q$.\footnote{Note that the sets $\widehat{\If}_Q$ are defined in a slightly different way than the sets $\F_Q$ from (\ref{eqn:DefnFareySeq}).}  Hence the set $\widehat{\If}_Q$ is comprised of the points of the set $\widehat{\ZZ}^d \tensor*[^t]{L}{^{-1}}$ in the region $[0,Q]^{d-1} \times (0, Q]$ and, using the Gauss circle problem or its various generalizations (see~\cite[Proposition~3.2]{MS14} for example), we have the following asymptotics: \begin{align}\label{eqnAsyHatIf}
 \#(\widehat{\If}_Q) \sim \frac{Q^d}{\zeta(d)}  \quad \textrm{ as }  \quad Q \rightarrow \infty. \end{align}

\subsection{Points on a translated horosphere in the same $\Gamma$-coset}\label{subsec:GammaDuplicates}  In this section, fix $t \in \RR$.  Let $\boldsymbol{x}$ and $\boldsymbol{y}$ be two distinct elements of $\RR^{d-1}$.  When the cosets $\Gamma L n_-(\boldsymbol{x})\Phi^t$ and $\Gamma L n_-(\boldsymbol{y})\Phi^t$ are equal, we refer to them as {\em $\Gamma$-duplicates}.  As $\Gamma$-duplicates are identified on $\Gamma \backslash G$, we wish to find a (maximal) region $\Aa_L$ in $\RR^{d-1}$ which corresponds to points of the translated horosphere with no $\Gamma$-duplicates.  As the following theorem shows, $\Aa_L$ is either all of $\RR^{d-1}$ or the fundamental domain of a torus lying in $\RR^{d-1}$.  Let $s_{j,d}$ and $\delta(\cdot)$ be as defined in Section~\ref{subsecTranFareySeq}.

\begin{theo}\label{thm:HorosphereDuplicates}  We have $\Aa_L = \RR^{d-1}$ for Haar-almost all $L \in G$.  Otherwise, for $L$ in the complementary subset of $G$, we have \[\Aa_L =\begin{cases} \TT^{d-1}\left(\frac{A^{-1}}{\det(A)}\right) & \text{ if } \det (A) \neq 0 \\ \TT^{d-1}\left(\frac{\widehat{A}^{-1}}{\det(\widehat{A})}\right) & \text{ if } \det (A) = 0\end{cases}.\]  Here, \[\begin{pmatrix}  A &  \tensor*[^t]{\boldsymbol{u}}{} \\ \boldsymbol{v} & c  \end{pmatrix} := L^{-1} \] and, if $\det(A) = 0$, \begin{align}\label{eqnHorosphereDuplicates}\begin{pmatrix}  \widehat{A} &\tensor*[^t]{\widehat{\boldsymbol{u}}}{}\\ \widehat{\boldsymbol{v}} &\widehat{c} \end{pmatrix} := L^{-1} s_{j,d} \end{align} where the index $j$ is chosen such that $\det(\widehat{A}) \neq 0$.

\end{theo}

\begin{rema}
Note that such a $j$ always exists.  Also, as the proof below is independent of $t$, the region $\Aa_L$ is independent of $t$.  

\end{rema}

\begin{proof}
Two points, $\Gamma L n_-(\boldsymbol{x})\Phi^t$ and $\Gamma L n_-(\boldsymbol{y})\Phi^t$, on the $\Phi^t$-translate of the translated horosphere are $\Gamma$-duplicates if and only if there exists a $\gamma \in \Gamma \backslash \{I_d\}$ such that\footnote{Note that, if we let $\gamma = I_d$ in (\ref{eqn:GammaDuplicateCond}), then we have that  $\boldsymbol{x} = \boldsymbol{y}$.  The converse of this statement also holds.} \begin{align} \label{eqn:GammaDuplicateCond}
 \gamma^{-1} L = L n_-(\boldsymbol{y} - \boldsymbol{x}). \end{align}  Letting \[L^{-1} = \begin{pmatrix}  A &  \tensor*[^t]{\boldsymbol{u}}{} \\ \boldsymbol{v} & c  \end{pmatrix} \quad \gamma = \begin{pmatrix}  N &  \tensor*[^t]{\boldsymbol{r}}{} \\ \boldsymbol{p} & q  \end{pmatrix} \quad \boldsymbol{s} = \boldsymbol{x} - \boldsymbol{y},\] we have that (\ref{eqn:GammaDuplicateCond}) can be rewritten as \begin{align}\label{eqn:GammaDuplicateCondMatForm}
 \begin{pmatrix}  A &  \tensor*[^t]{\boldsymbol{u}}{} \\ \boldsymbol{v} & c  \end{pmatrix} \begin{pmatrix}  N &  \tensor*[^t]{\boldsymbol{r}}{} \\ \boldsymbol{p} & q  \end{pmatrix} = \begin{pmatrix}  I_{d-1} & \tensor*[^t]{\boldsymbol{0}}{} \\\boldsymbol{s} & 1   \end{pmatrix} \begin{pmatrix}  A &  \tensor*[^t]{\boldsymbol{u}}{} \\ \boldsymbol{v} & c  \end{pmatrix}, \end{align} which yields four equations, two of which are \begin{align}\label{eqnGammaDuplicates1} A(N - I_{d-1}) + \tensor*[^t]{\boldsymbol{u}}{} \boldsymbol{p} &= 0_{d-1,d-1} \\\nonumber
A \tensor*[^t]{\boldsymbol{r}}{} + \tensor*[^t]{\boldsymbol{u}}{} (q-1) &= \tensor*[^t]{\boldsymbol{0}}{}.
  \end{align}  Expressing this pair of equations in matrix form, we obtain \[\begin{pmatrix} A  & \tensor*[^t]{\boldsymbol{u}}{}   \end{pmatrix}\begin{pmatrix}  N - I_{d-1} &  \tensor*[^t]{\boldsymbol{r}}{} \\ \boldsymbol{p} & q -1  \end{pmatrix}  = \begin{pmatrix} A  & \tensor*[^t]{\boldsymbol{u}}{}   \end{pmatrix} (\gamma - I_d) = 0_{d-1, d},\] or, equivalently, \[(\tensor*[^t]{\gamma}{} - I_d) \begin{pmatrix} \tensor*[^t]{A }{} \\ \boldsymbol{u}   \end{pmatrix} =  0_{d, d-1}.\]  Here, $0_{i, j}$ denotes, for $i, j \in \NN$, the $i \times j$ matrix with all zero entries.  
  
Consequently, a necessary condition for the two points to be $\Gamma$-duplicates is for the first $d-1$ rows of $L^{-1}$ to lie in the eigenspace of some $\gamma \in \Gamma \backslash \{I_d\}$ for the eigenvalue $1$.  By the rank-nullity theorem, the geometric multiplicity of this eigenspace for any element of $\Gamma \backslash \{I_d\}$ is at most $d-1$.  As $\Gamma$ is countable, we have that, when $\Gamma$-duplicates exist, the first $d-1$ rows of $L^{-1}$ lie in a countable union of $d-1$-dimensional hyperplanes in $\RR^d$.  

We now show that such $L$ lie in a set of Haar measure zero.  Let $M^{-1} \in G$.  Now, if the first row of $M^{-1}$ lies in a $d-1$-dimensional hyperplane $\mathfrak{p}$ in $\RR^d$, then the first row of $t^{1/d} M^{-1}$, for every real number $t >0$, also lies in $\mathfrak{p}$.  Thus, by (\ref{eqnHaarMeasAsMultOFLebMeas}), we have $M^{-1}$ lies in a set of Haar measure zero.  As the Haar measure on $G$ is invariant under taking inverses, we have that $M$ lies in a set of Haar measure zero.  By the countable additivity of measure, we have shown that $\Gamma$-duplicates exist only for elements $L$ in a set of Haar measure zero.  This proves the first assertion.

We now prove the rest of the theorem.  Now let $L$ lie in the set of Haar measure zero for which $\Gamma$-duplicates exist.  Consequently, (\ref{eqn:GammaDuplicateCondMatForm}) holds for some $\boldsymbol{s} \in \RR^d \backslash \{\boldsymbol{0}\}$.  Now (\ref{eqn:GammaDuplicateCondMatForm}) is equivalent to the two equations from (\ref{eqnGammaDuplicates1}) and the following two additional equations \begin{align*} \boldsymbol{s} A &= \boldsymbol{v} \left(N -I_{d-1} \right) + c \boldsymbol{p} \\  \boldsymbol{s}\tensor*[^t]{\boldsymbol{u}}{} &= \boldsymbol{v} \tensor*[^t]{\boldsymbol{r}}{} +c(q-1).\end{align*}  Moreover, these four equations are equivalent to \begin{align}\begin{pmatrix} A  & \tensor*[^t]{\boldsymbol{u}}{}  \\ \boldsymbol{v} & c  \end{pmatrix}\begin{pmatrix}  N - I_{d-1} &  \tensor*[^t]{\boldsymbol{r}}{} \\ \boldsymbol{p} & q -1  \end{pmatrix}  = \begin{pmatrix} 0_{d-1, d-1} & \tensor*[^t]{\boldsymbol{0}}{} \\ \boldsymbol{s} A & \boldsymbol{s}\tensor*[^t]{\boldsymbol{u}}{}  \end{pmatrix}.
  \end{align}  Multiplying by $L$ on both sides and solving for $\boldsymbol{p}$ yields \begin{align}\boldsymbol{p} =  \delta(L)\boldsymbol{s} A.
  \end{align}
  
There are two cases:  $\det(A) \neq 0$ and $\det(A) = 0$.   For $\det(A) \neq 0$, we have that $A$ is invertible and, thus, $\boldsymbol{s} = \boldsymbol{p} \frac{A^{-1}}{\delta(L)}$ where $\delta(L)$ is computed as in (\ref{eqnInverseMatrixddCoordInvertibleA}).  Using (\ref{Decomp2InG2}) gives the desired result.  Finally, for $\det(A) = 0$, we have that $\det(\widehat{A}) \neq 0$.  Using (\ref{eqnHorosphereDuplicates}) and \[\begin{pmatrix}  \widehat{N} &  \tensor*[^t]{\widehat{\boldsymbol{r}}}{} \\ \widehat{\boldsymbol{p}} & \widehat{q}  \end{pmatrix} :=s^{-1}_{j,d} \gamma s_{j,d} \in  \Gamma \backslash \{I_d\}\] in (\ref{eqn:GammaDuplicateCondMatForm}) allows us to apply the previous case to obtain the desired result.  This proves the theorem.

\end{proof}

\subsection{Equidistribution of translated Farey sequences:  preparation}\label{subsecEquTranslatedFareySeq}

To prove the equidistribution of translated Farey sequences (namely, Theorem~\ref{thmGeneralMarklof}), we need the following version of Shah's theorem (\cite[Theorem~1.4]{Shah}), namely~\cite[Theorem~5.8]{MS10}.

\begin{theo}\label{thm:TransShah:TransHoro}  Let $\lambda$ be a Borel probability measure on $\RR^{d-1}$ which is absolutely continuous with respect to the Lebesgue measure and $f: \RR^{d-1} \times \Ga \backslash G \rightarrow \RR$ be a bounded continuous function.  Then, for every $L \in G$, we have that \[\lim_{t \rightarrow \infty} \int_{\RR^{d-1}} f(\boldsymbol{x}, L n_-(\boldsymbol{x}) \Phi^t)~\wrt\lambda(\boldsymbol{x}) = \int_{\RR^{d-1} \times \Ga \backslash G}  f(\boldsymbol{x}, M)~\wrt\lambda(\boldsymbol{x})~\wrt\mu(M). \]
 
\end{theo}

\begin{rema}\label{rmk:TransShah:TransHoro}
Note that the Portmanteau theorem (see~\cite[Chapter~III]{Shi95} for example) allows us to extend Theorem~\ref{thm:TransShah:TransHoro} to the indicator functions of continuity sets (i.e. sets whose topological boundary has measure zero).  Also note that Theorem~\ref{thm:TransShah:TransHoro} is a consequence of mixing and does not require Ratner's theorems (see the end of Section~5.4 in~\cite{MS10}).

\end{rema}

\subsection{Equidistribution of translated Farey sequences:  proof of Theorem~\ref{thmGeneralMarklof}}\label{sec:proofGeneralMarklof}

The proof is an adaptation of the proof of Theorem~\ref{thmMarklof} in~\cite{Mar10}, but with a different version of Shah's theorem (namely with Theorem~\ref{thm:TransShah:TransHoro}).  Since we are adapting the proof in \cite{Mar10}, \textit{Steps 0 --  6} will refer to the original proof in~\cite{Mar10}.  First, note that multiplying $L$ on the left by an element of $\Ga$ leaves the intersection points invariant and it follows that, without loss of generality, we may assume that \begin{align} \label{eqnCondOnL}
 \delta(L) > 0  \end{align} because, if this were not so, then we replace $L$ by $\tensor*[^t]{\gamma}{} L$ for some $\gamma \in \Ga$ for which  Lemma~\ref{lemmChoiceLDeltaNonZero} holds.  Now, let $\widetilde{L} :=  \tensor*[^t]{L}{^{-1}}$ and let us define the following subsets of $\widehat{\If}_Q$:

\[\widehat{\If}_{Q,\theta} :=\left\{\left(\frac{\alpha_1}{\alpha_d}, \cdots,  \frac{\alpha_{d-1}}{\alpha_d}\right) \in \widehat{\If}_Q : \theta Q < \alpha_d \leq Q \right\} \] for any $\theta \in (0,1)$.  The constant $\theta$ remains fixed until the very end of {\em Step~6}.  By the bijection between $\If$ and $\left(\widehat{\ZZ}^d \tensor*[^t]{L}{^{-1}}\right)^+$, we have that each element of \[\left(\frac{\alpha_1}{\alpha_d}, \cdots,  \frac{\alpha_{d-1}}{\alpha_d}\right) \in\bigcup_{Q=1}^\infty \widehat{\If}_{Q,\theta}\] determines a unique element $(\boldsymbol{\alpha}', \alpha_d) := \left(\alpha_1, \cdots,  \alpha_d \right)$ for which $\alpha_d >0$ and also determines a unique element $\gamma_{(\boldsymbol{\alpha}', \alpha_d)} \in \Ga_H \backslash \Ga$.

With the change of the version of Shah's theorem used to the proof in~\cite{Mar10} and the replacement of $\F_{Q, \theta}$ with $\widehat{\If}_{Q,\theta}$, we need to make the following changes to {\em Steps~0 -- 6}.  (The main changes are to \textit{Steps 1, 5}.)

\subsubsection{Step 0:  Uniform continuity}  The assertions in this step are the same for us.  However, since we have replaced the torus $\TT^{d-1}$ with $\Aa \subset \RR^{d-1}$, we need to specify the compact support of $f$.  The only change in this step is that we have compact sets $\widehat{\Aa} \subset \RR^{d-1}$ and $\Cc \subset G$ such that $\supp(f), \supp(\widetilde{f}) \subset \widehat{\Aa} \times \Gamma \backslash \Gamma \Cc$.

Note that the uniform continuity is expressed as the following:  given ${\delta} >0$, there exists $\epsilon >0$ such that for all $(\boldsymbol{x}, M), (\boldsymbol{x}', M') \in \RR^{d-1} \times G$, \begin{align}\label{eqn:UnifContEpDelTranHoro}\|\boldsymbol{x} - \boldsymbol{x}'\|_2 < \epsilon \quad \textrm{ and } \quad d(M, M') < \epsilon \implies |{f}(\boldsymbol{x}, M) - {f}(\boldsymbol{x}', M')| < {\delta}.
\end{align} The same applies to $\widetilde{f}$.

\subsubsection{Step 1: Thicken the translated Farey sequence.}  We thicken the points in the set $\widehat{\If}_{Q,\theta}$ under the correspondence given by Theorem~\ref{thmTransFareyInterPointCorr} as follows.  Let \[\Cf_\epsilon := \left\{(y_1, \cdots, y_d) \in \RR^d : \|(y_1, \cdots, y_{d-1})\|_2 < \epsilon y_d, \theta < y_d \leq 1 \right\}.\]  For any $\boldsymbol{u} \in \RR^d$ (thought of as a row vector), define \[ \Hc_\epsilon(\boldsymbol{u}) := \left\{\widetilde{M} \in G :  \boldsymbol{u} \widetilde{M} \in \Cf_\epsilon  \right\} .\]  An important relationship between $G$ and $\RR^d$ is the following observation.

 \begin{lemm}\label{lemmRelGEuclid}
For every $M\in G$ and every $\boldsymbol{u} \in \RR^d$, we have that \[M \Hc_\epsilon(\boldsymbol{u} M)  =  \Hc_\epsilon(\boldsymbol{u}).\]
\end{lemm}
\begin{proof}
 The proof follows from the fact that \[\boldsymbol{u} \widetilde{M} \in \Cf_\epsilon \iff  \boldsymbol{u} M M^{-1} \widetilde{M} \in \Cf_\epsilon.\]
\end{proof}

Moreover, we have that (\cite[(3.20)]{Mar10}) \begin{align}\Hc^1_\epsilon:=\Hc_\epsilon((\boldsymbol{0}, 1)) = H\left\{\begin{pmatrix}  y_d^{-1/(d-1)} I_{d-1} & \tensor*[^t]{\boldsymbol{0}}{} \\ \boldsymbol{y}' & y_d \end{pmatrix} : (\boldsymbol{y}', y_d) \in \Cf_\epsilon \right\}.
  \end{align}

We claim that, for each $\boldsymbol{\alpha}'/\alpha_d \in \widehat{\If}_{Q,\theta}$ and its unique corresponding $\gamma :=  \gamma_{(\boldsymbol{\alpha}', \alpha_d)} \in \Ga_H \backslash \Ga$, we have that their unique corresponding intersection point \[\pf:= \pf_{(\boldsymbol{\alpha}', \alpha_d)} := \tensor*[^t]{\left(\tensor*[^t]{\gamma}{^{-1}} L n_-(\boldsymbol{\alpha}'/\alpha_d) \Phi^{t} \right)}{^{-1}}\] lies in $\Hc^1_\epsilon$.  We now prove the claim.  Since $\boldsymbol{\alpha}'/\alpha_d \in \widehat{\If}_{Q,\theta}$, we have that $\delta(\gamma \widetilde{L}) > 0$ and Lemma~\ref{lemmUniqueDecomIfYdNot0} gives a unique decomposition \[\gamma \widetilde{L} = (\gamma \widetilde{L})_H \begin{pmatrix}  \alpha_d^{-1/(d-1)} I_{d-1} & \tensor*[^t]{\boldsymbol{0}}{} \\\boldsymbol{\alpha'} &\alpha_d   \end{pmatrix}\] because $(\boldsymbol{\alpha}', \alpha_d) = (\boldsymbol{0}, 1) \gamma \widetilde{L}.$  Consequently, we have that \[\pf = (\gamma \widetilde{L})_H \begin{pmatrix}  \alpha_d^{-1/(d-1)} I_{d-1} & \tensor*[^t]{\boldsymbol{0}}{} \\\boldsymbol{\alpha'} &\alpha_d   \end{pmatrix} n_+\left(-\frac{\boldsymbol{\alpha}'}{\alpha_d}\right) \Phi^{-t} = (\gamma \widetilde{L})_H \begin{pmatrix}  (\frac{\alpha_d}{Q})^{-1/(d-1)} I_{d-1} & \tensor*[^t]{\boldsymbol{0}}{} \\\boldsymbol{0} & \frac{\alpha_d} {Q}  \end{pmatrix}\] where in {\em Step~0} we have set $\sigma =0$ and thus have $Q= e^{(d-1)t}$.  This shows that $\pf \in \Hc^1_\epsilon$.  Finally, to see that the element of $\Ga_H \backslash \Ga$ corresponding to  $\boldsymbol{\alpha}'/\alpha_d$ is unique, we set \[(\gamma \widetilde{L})_H^{-1} \gamma \widetilde{L} = (\widehat{\gamma }\widetilde{L})_H^{-1} \widehat{\gamma }\widetilde{L},\] from which it follows that $\gamma, \widehat{\gamma}$ are in the same coset in $\Ga_H \backslash \Ga$.  This proves the claim.

Note that the claim allows us to regard  $\Hc^1_\epsilon$ as a thickening of the translated Farey points.

\subsubsection{Step 2: Disjointness.}  This step is unchanged.  The result is the following~\cite[(3.22)]{Mar10}.  Given a compact subset $\Cc \subset G$, there exists an $\epsilon_0>0$ such that \[\gamma \Hc^1_\epsilon \cap \Hc^1_\epsilon \cap \Gamma \Cc = \emptyset \] for every $\epsilon \in (0, \epsilon_0], \gamma \in \Gamma \backslash \Gamma_H$.

\subsubsection{Step 3:  Apply Theorem~\ref{thm:TransShah:TransHoro}}  This step is essentially unchanged, except we apply a translated version of Shah's theorem.  The details are as follows.   For all $M \in \Gamma \Cc$, \textit{Step 2} gives \[\chi_\epsilon(M) = \sum_{\gamma \in \Gamma_H \backslash \Gamma} \chi_\epsilon^1(\gamma M)\] where $\chi_\epsilon^1$ is the indicator function of $\Hc^1_\epsilon$ and $\chi_\epsilon(M)$ is the indicator function of the disjoint union \[\Hc_\epsilon \cap \Gamma \Cc := \bigcup_{\gamma \in \Gamma / \Gamma_H} \left(\gamma \Hc^1_\epsilon \cap \Gamma \Cc \right).\]

Applying Theorem~\ref{thm:TransShah:TransHoro} and Remark~\ref{rmk:TransShah:TransHoro}, we have \begin{align}\label{eqn:EquDisTranTransHoro} \lim_{t \rightarrow \infty} \int_{\RR^{d-1}}  f(\boldsymbol{x}, L n_-(\boldsymbol{x}) \Phi^t)\widetilde{\chi_\epsilon}(Ln_-(\boldsymbol{x}) \Phi^{t})~\wrt{\boldsymbol{x}} &= \int_{\RR^{d-1} \times \Gamma \backslash G}  f(\boldsymbol{x}, M)\widetilde{\chi_\epsilon}(M)~\wrt{\boldsymbol{x}}~\wrt\mu(M) \\ \nonumber& = \int_{\RR^{d-1} \times \Gamma \backslash G}  \widetilde{f}(\boldsymbol{x}, M)\chi_\epsilon(M)~\wrt{\boldsymbol{x}}~\wrt\mu(M) 
  \end{align}

  \subsubsection{Step 4:  A volume computation}  Except for the fact that the domain of the first variable of the function $f$ is a compact subset $\widehat{\Aa}$ of $\RR^{d-1}$, this step is unchanged.  In particular, the proofs are exactly the same.  The result is the following~\cite[(3.46)]{Mar10}:

   \begin{align}\label{eqn:localVerTransHoro} &\left|  \lim_{t \rightarrow \infty} \int_{\RR^{d-1}}  f(\boldsymbol{x}, L n_-(\boldsymbol{x}) \Phi^t)\widetilde{\chi_\epsilon}(Ln_-(\boldsymbol{x}) \Phi^{t})~\wrt{\boldsymbol{x}}   \right. \\ \nonumber & \left. \quad - \frac{(d-1) \vol(B_1^{d-1}) \epsilon^{d-1} }{\zeta(d)} \int_{0}^{|\log \theta|/(d-1)}  \int_{\RR^{d-1} \times \Gamma_H \backslash H} \widetilde{f} (\boldsymbol{x}, M \Phi^{-s})~\wrt{\boldsymbol{x}}~\wrt\mu_H(M)e^{-d(d-1)s}~\wrt s\right|  \\ \nonumber & \quad \quad \quad \quad <\frac {\vol(B_1^{d-1}){\delta}\epsilon^{d-1}\vol(\widehat{\Aa})}{d \zeta(d)}\left( 1 -\theta^d\right)  
\end{align} for $\delta, \epsilon >0$ coming from \textit{Steps 0, 2}.  Here $B_1^{d-1}$ denotes the open ball around $\boldsymbol{0}$ of radius $1$ in $\RR^{d-1}$ and $\vol$ is the Lebesgue measure on $\RR^{d-1}$.

\subsubsection{Step 5:  Distance estimates}

\begin{lemm}
Let $\boldsymbol{\alpha}'/\alpha_d \in \widehat{\If}_{Q,\theta}$ and $\gamma :=  \gamma_{(\boldsymbol{\alpha}', \alpha_d)}$ be its unique corresponding element of $\Ga_H \backslash \Ga$.  Then \[\gamma \widetilde{L}n_+(-\boldsymbol{x}) \Phi^{-t} \in \Hc^1_\epsilon\] if and only if \[\left\|\frac{\boldsymbol{\alpha}'}{\alpha_d}- \boldsymbol{x}\right\|_2 < \epsilon e^{-dt} \quad \textrm { and } \quad \theta Q < \alpha_d \leq Q\] hold.
\end{lemm}

\begin{proof}

Recall that $Q = e^{(d-1)t}$.  Consider \begin{align*}
 (\boldsymbol{0}, 1)\gamma \widetilde{L}n_+(-\boldsymbol{x}) \Phi^{-t} & = (\boldsymbol{0}, 1)  (\gamma \widetilde{L})_H \begin{pmatrix}  \alpha_d^{-1/(d-1)} I_{d-1} & \tensor*[^t]{\boldsymbol{0}}{} \\\boldsymbol{\alpha'} &\alpha_d   \end{pmatrix}n_+(-\boldsymbol{x}) \Phi^{-t} \\\nonumber & =  (\boldsymbol{0}, 1) \begin{pmatrix}  e^t\alpha_d^{-1/(d-1)} I_{d-1} & \tensor*[^t]{\boldsymbol{0}}{} \\ e^t(\boldsymbol{\alpha'} -\boldsymbol{x}\alpha_d)  &e^{-(d-1)t}\alpha_d   \end{pmatrix}.\end{align*}  The desired result now follows by the definitions of the sets $\Hc^1_\epsilon$ and $\Cf_\epsilon$.

\end{proof}

\textit{Step 2}, the lemma, and the fact that $f$ and $\widetilde{f}$ are both left $\Gamma$-invariant give us that \begin{align}\label{eqnDisjointDist:TransHoro}\int_{\RR^{d-1}}  {f}(\boldsymbol{x}, {L} & n_-(\boldsymbol{x}) \Phi^{t}){\chi_\epsilon}(\widetilde{L}n_+(-\boldsymbol{x}) \Phi^{-t})~\wrt{\boldsymbol{x}} \\\nonumber & =
\int_{\RR^{d-1}}  \widetilde{f}(\boldsymbol{x}, \widetilde{L} n_+(-\boldsymbol{x}) \Phi^{-t}){\chi_\epsilon}(\widetilde{L}n_+(-\boldsymbol{x}) \Phi^{-t})~\wrt{\boldsymbol{x}} \\\nonumber &= \sum_{\gamma \in \Gamma_H \backslash \Gamma} \int_{\RR^{d-1}}  \widetilde{f}(\boldsymbol{x}, \gamma \widetilde{L} n_+(-\boldsymbol{x}) \Phi^{-t})  \chi_\epsilon^1(\gamma \widetilde{L}n_+(-\boldsymbol{x}) \Phi^{-t})~\wrt{\boldsymbol{x}}
 \\\nonumber &= \sum_{\frac{\boldsymbol{\alpha}'}{\alpha_d} \in \widehat{\If}_{Q,\theta}} \int_{\left\|\frac{\boldsymbol{\alpha}'}{\alpha_d}- \boldsymbol{x}\right\|_2 < \epsilon e^{-dt} }  \widetilde{f}(\boldsymbol{x},  \gamma \widetilde{L} n_+(-\boldsymbol{x}) \Phi^{-t})~\wrt{\boldsymbol{x}}
 \\\nonumber &= \sum_{\frac{\boldsymbol{\alpha}'}{\alpha_d} \in \widehat{\If}_{Q,\theta}} \int_{\left\|\frac{\boldsymbol{\alpha}'}{\alpha_d}- \boldsymbol{x}\right\|_2 < \epsilon e^{-dt} }  {f}(\boldsymbol{x}, {L} n_-(\boldsymbol{x}) \Phi^{t})~\wrt{\boldsymbol{x}}. \end{align}
 
Applying the uniform continuity with $\delta, \epsilon>0$ as above gives \begin{align}\label{eqnDisjointDist2:TransHoro} & \left |  \int_{\left\|\frac{\boldsymbol{\alpha}'}{\alpha_d}- \boldsymbol{x}\right\|_2 < \epsilon e^{-dt} }  {f}(\boldsymbol{x},  {L} n_-(\boldsymbol{x}) \Phi^{t})~\wrt{\boldsymbol{x}} - \frac {\vol(B_1^{d-1}) \epsilon^{d-1}}{e^{d(d-1)t}} {f}\left(\frac{\boldsymbol{\alpha}'}{\alpha_d},  {L} n_-\left(\frac{\boldsymbol{\alpha}'}{\alpha_d}\right) \Phi^{t}\right)\right| \\ \nonumber& \quad \quad \quad < \frac {\vol(B_1^{d-1}) \delta \epsilon^{d-1}}{e^{d(d-1)t}}
  \end{align} uniformly for all $t \geq 0$.

\subsubsection{Step 6:  Conclusion}  This step is essentially unchanged.  Using (\ref{eqn:localVerTransHoro}), (\ref{eqnDisjointDist:TransHoro}), (\ref{eqnDisjointDist2:TransHoro}) and the uniform continuity, we have, as $\delta \rightarrow 0$, \begin{align}\label{eqn:MainConclusion:TransHoro}\lim_{t \rightarrow \infty} \frac {1} {e^{d(d-1)t}} & \sum_{\frac{\boldsymbol{\alpha}'}{\alpha_d} \in \widehat{\If}_{Q,\theta}} {f}\left(\frac{\boldsymbol{\alpha}'}{\alpha_d},  {L} n_-\left(\frac{\boldsymbol{\alpha}'}{\alpha_d}\right) \Phi^{t}\right) \\ \nonumber &= \frac{d-1}{\zeta(d)} \int_{0}^{|\log \theta|/(d-1)}  \int_{\RR^{d-1} \times \Gamma_H \backslash H} \widetilde{f} (\boldsymbol{x}, M \Phi^{-s})~\wrt{\boldsymbol{x}}~\wrt\mu_H(M)e^{-d(d-1)s}~\wrt s.\end{align}  The asymptotics (\ref{eqnAsyHatIf}) show that \[\limsup_{t \rightarrow \infty}\frac{\#\left(\widehat{\If}_{Q} \backslash \widehat{\If}_{Q,\theta} \right)}{e^{d(d-1)t}} \leq \frac{\theta}{\zeta(d)},\] which allows us to take the limit $\theta \rightarrow 0$ in (\ref{eqn:MainConclusion:TransHoro}).  This proves Theorem~\ref{thmGeneralMarklof} for $\sigma=0$ and $f$ compactly supported.  The general case follows in exactly the same way as in~\cite{Mar10}.  This completes the proof of Theorem~\ref{thmGeneralMarklof}.

\subsection{Shrinking target horospherical equidistribution for translated horospheres}\label{subsecUniEquiForTransHoro}  We now prove our STHE results for translated horospheres by adapting our proofs for the horosphere.

\begin{proof}[Proof of Theorem~\ref{thmNANResultRankOneFlow} for any $L \in G$]  In the proof of Theorem~\ref{thmNANResultRankOneFlow} for the case $L = I_d$, replace the Farey sequence with the translated Farey sequence corresponding to $\Gamma L N^-$.  All of the proof up to (\ref{eqnEquiOnSectTransA}) remains the same.  Applying Theorem~\ref{thmGeneralMarklof} in place of Theorem~\ref{thmMarklof}, we obtain the analog of (\ref{eqnEquiOnSectTransA}):  \begin{align}\label{eqnEquiOnSectTransATransHoro}\lim_{t \rightarrow \infty} &\frac {T^{d-1}} {\#(\widehat{\If}_Q)} \sum_{\frac{\boldsymbol{\alpha}'}{\alpha_d} \in \widehat{\If}_{Q}}  f_{T, \varepsilon}\left(\frac{\boldsymbol{\alpha}'}{\alpha_d},  {L} n_-\left(\frac{\boldsymbol{\alpha}'}{\alpha_d}\right) \Phi^{t}\right)  
\\ \nonumber&
  = (d-1)T_0^{d-1}\varepsilon^{d-1} \int_0^\infty \int_{\RR^{d-1} \times \Ga_H \backslash H} \In_{\Aa \times \Ss_{T_0}}(\boldsymbol{x},  M \Phi^{-s})~\wrt{\boldsymbol{x}}~\wrt \mu_H(M) e^{-d(d-1)s}~\wrt s 
    \\\nonumber & = \frac {\varepsilon^{d-1}} {d}  \int_{\RR^{d-1}} \In_{\Aa}(\boldsymbol{x})~\wrt{\boldsymbol{x}}.
   \end{align}  
   
Finally, as in Theorem~\ref{thmNANResultRankOneFlow} for the case $L = I_d$, we normalize with respect to the Haar measure on $N^-$, namely the Haar measure of  $\Phi^{-t}N_-(\TT^{d-1})\Phi^t =  \Phi^{-t}N_-([0,1]^{d-1})\Phi^t$ to obtain the desired result.

\end{proof}

\begin{rema}
The sets $\widehat{\If}_Q$, $\F_Q$ are defined slightly differently and, in particular, have different asymptotics, namely (\ref{eqnAsypFQOriginal}) and (\ref{eqnAsyHatIf}), respectively.  That these two sets are different does not affect any of our STHE results.
\end{rema}

\begin{proof}[Proof of Theorem~\ref{thmShrinkCuspNeigh} for any $L \in G$]
  In the proof of Theorem~\ref{thmShrinkCuspNeigh} for the case $L = I_d$, replace the Farey sequence with the translated Farey sequence corresponding to $\Gamma L N^-$.  All of the proof up to (\ref{eqnEquiOnSectTransB}) remains the same.  Applying Theorem~\ref{thmGeneralMarklof} in place of Theorem~\ref{thmMarklof}, we obtain the analog of (\ref{eqnEquiOnSectTransB}):     
   \begin{align}\label{eqnEquiOnSectTransBTransHoro}\lim_{t \rightarrow \infty} &\frac {T^{d-1}} {\#(\widehat{\If}_Q)} \sum_{\frac{\boldsymbol{\alpha}'}{\alpha_d} \in \widehat{\If}_{Q}}  \widehat{f}_{T, \varepsilon}\left(\frac{\boldsymbol{\alpha}'}{\alpha_d},  {L} n_-\left(\frac{\boldsymbol{\alpha}'}{\alpha_d}\right) \Phi^{t}\right)  \\ \nonumber&
  = (d-1)T_0^{d-1}\varepsilon^{d-1} \int_0^\infty \int_{\RR^{d-1} \times \Ga_H \backslash H} \In_{\Aa \times \Cc}(\boldsymbol{x},  M \Phi^{-s})~\wrt{\boldsymbol{x}}~\wrt \mu_H(M) e^{-d(d-1)s}~\wrt s 
  \\ \nonumber&
  = T_0^{d-1}  \zeta(d) \mu\left(\Cc N^+_\varepsilon(\boldsymbol{\widetilde{y}})\right)\int_{\RR^{d-1}} \In_{\Aa}(\boldsymbol{x})~\wrt{\boldsymbol{x}}.
  \end{align}

Finally, as in Theorem~\ref{thmShrinkCuspNeigh} for the case $L = I_d$, we normalize with respect to the Haar measure on $N^-$, namely the Haar measure of  $\Phi^{-t}N_-(\TT^{d-1})\Phi^t =  \Phi^{-t}N_-([0,1]^{d-1})\Phi^t$ to obtain the desired result.

\end{proof}

\begin{proof}[Proof of Theorem~\ref{thmNAKResultRankOneFlow} for any $L \in G$]  In the proof of Theorem~\ref{thmNAKResultRankOneFlow} for the case $L = I_d$, replace the Farey sequence with the translated Farey sequence corresponding to $\Gamma L N^-$.  All of the proof up to (\ref{eqnMainDerivUsingFrobNAKCase}) remains the same.  Applying Theorem~\ref{thmGeneralMarklof} in place of Theorem~\ref{thmMarklof}, we obtain the analog of (\ref{eqnMainDerivUsingFrobNAKCase}):

\begin{align}\label{eqnMainDerivUsingFrobNAKCaseTransHoro}\lim_{t \rightarrow \infty} &\frac {T^{d-1}} {\#(\widehat{\If}_Q)} \sum_{\frac{\boldsymbol{\alpha}'}{\alpha_d} \in \widehat{\If}_{Q}}  \widehat{f}_{T, \Dd}\left(\frac{\boldsymbol{\alpha}'}{\alpha_d},  {L} n_-\left(\frac{\boldsymbol{\alpha}'}{\alpha_d}\right) \Phi^{t}\right)    
  = \zeta(d) T_0^{d-1}\mu(\Ss_{T_0} \Ee)\left(\int_{\RR^{d-1}} \In_{\Aa}(\boldsymbol{x})~\wrt{\boldsymbol{x}}\right).
        \end{align}

Finally, as in Theorem~\ref{thmNAKResultRankOneFlow} for the case $L = I_d$, we normalize with respect to the Haar measure on $N^-$, namely the Haar measure of  $\Phi^{-t}N_-(\TT^{d-1})\Phi^t =  \Phi^{-t}N_-([0,1]^{d-1})\Phi^t$ to obtain (\ref{eqnthmNAKResultRankOneFlow}).

\end{proof}

\begin{proof}[Proof of Theorem~\ref{thmThickenSubsetSectionSpherical} for any $L \in G$]
  In the proof of Theorem~\ref{thmThickenSubsetSectionSpherical} for the case $L = I_d$, replace the Farey sequence with the translated Farey sequence corresponding to $\Gamma L N^-$ and follow the proof of Theorem~\ref{thmNAKResultRankOneFlow}.

\end{proof}

The analog of Lemma~\ref{lemm:HorosphereInSphereEdge} is 

\begin{lemm}\label{lemm:HorosphereInSphereEdgeGenL}  Let $L \in G$.  Then there exists a null set $\Nn$ with respect to $\wrt {\boldsymbol{x}}$ such that, for every $t \in \RR$, we have that \[\{\boldsymbol{x} \in \RR^{d-1} : L n_-(\boldsymbol{x}) \Phi^t \in \Ss_1 K_0 \} \subset \Nn.\]  Moreover, $\Nn$ can be chosen to be a countable union of affine hyperplanes in $\RR^{d-1}$ of dimension $d-2$.
 
\end{lemm}

\begin{proof}
Replace $n_-(\boldsymbol{x})$ with $L n_-(\boldsymbol{x})$ in (\ref{eqn:HorosphereInSphereEdgeMain}) and replace (\ref{eqn:HorosphereInSphereEdgeMain2}) with \begin{align} \label{eqn:HorosphereInSphereEdgeMain2GenL} (\boldsymbol{p}, q) = (\boldsymbol{0}, 1) \gamma L.\end{align}  This yields (\ref{eqn:HorosphereInSphereEdge}) and an affine hyperplane as in Lemma~\ref{lemm:HorosphereInSphereEdge}.  Finally, the latter replacement gives that $(\boldsymbol{p}, q) \in \If(\tensor*[^t]{L}{^{-1}})$, which is comprised of the points of the set $\widehat{\ZZ}^d L$, implying that there are at most a countable number of such affine hyperplanes.  This yields the desired result.
\end{proof}

Finally, we have

\begin{proof}[Proof of Theorem~\ref{thmThickenSubsetSectionSphericalGeneral} for any $L \in G$]
  In the proof of Theorem~\ref{thmThickenSubsetSectionSphericalGeneral} for the case $L = I_d$, replace the Farey sequence with the translated Farey sequence corresponding to $\Gamma L N^-$ and use Theorem~\ref{thmThickenSubsetSectionSpherical}.  For the case $c=0$, replace Lemma~\ref{lemm:HorosphereInSphereEdge} with Lemma~\ref{lemm:HorosphereInSphereEdgeGenL}.  \end{proof}

\section{More on volumes}\label{secMoreOnVolumes}  In this section, we compute more volumes.  These results are used in previous sections.  Recall the definition of $\widetilde{\Ss}_T(\boldsymbol{z})$ in (\ref{eqnPullingBackOfSectionalNeigh2}) and that of $\widehat{\Ss}_T(\boldsymbol{z})$ in (\ref{eqn:SphericalSectionCoords}).  The following generalizes Theorem~\ref{thmHaarMeaSE}:
\begin{theo}\label{thmHaarMeaWidetildeSE}  We have that  \begin{align*}\mu\left(\left\{\widetilde{\Ss}_T(\boldsymbol{z}) E^{-1}(\boldsymbol{z}): \boldsymbol{z} \in \Dd\right\} \right)& \\ =\frac{d-1}{\zeta(d)} \int \int \int &\In_{\left\{\widehat{\Ss}_T(\boldsymbol{z}) n_+(\boldsymbol{z'}):  \boldsymbol{z} \in \Dd \right\}}\left(M \Phi^{-s}n_+(\widetilde{\boldsymbol{x}})\right)~\wrt \mu_H(M) e^{-d(d-1)s}~\wrt s~\wrt{\widetilde{\boldsymbol{x}}}.
\end{align*}

\end{theo}

\begin{proof}
Use Lemma~\ref{lemmMMyCoordinatesForSECuspNeigh} in place of Lemma~\ref{lemmMMyCoordinatesForSE} in the proof of Theorem~\ref{thmHaarMeaSE} to obtain the desired result.
\end{proof}

The following generalizes Theorem~\ref{thnHaarMeaSTEvsST0E}:
\begin{theo}\label{thnHaarMeaSTEvsST0ECuspidalVersion} 
We have that \begin{align*}\mu\left(\left\{\widetilde{S}_T(\boldsymbol{z}) E^{-1}(\boldsymbol{z}): \boldsymbol{z}\in \Dd\right\}\right)& = \frac {T_0^{d-1}} {T^{d-1}} \mu\left(\left\{\widetilde{S}(\boldsymbol{z}) E^{-1}(\boldsymbol{z}):  \boldsymbol{z} \in \Dd\right\}\right).\end{align*}

\end{theo}
\begin{proof}

Using Iwasawa coordinates, we have that \[\wrt{\mu} = \widetilde{C}_0 \rho(a)~\wrt n~\wrt a~\wrt k,\] where $\widetilde{C}_0>0$ is a constant. Equation (\ref{eqnPullingBackOfSectionalNeigh2}) states that $\widetilde{S}_T(\boldsymbol{z})$ and $\widetilde{S}(\boldsymbol{z})$ are the same in Iwasawa coordinates except for a factor of $\Phi^{-\frac 1 d \log(T/T_0)}$.  Equation (\ref{eqnModularFct}) implies that we obtain the factor \[\frac{T_0^{d-1}}{T^{d-1}},\] which is the desired result.
\end{proof}

\end{document}